\documentclass[reqno,10pt]{amsart}
\usepackage{amssymb,amsmath,amsthm,}
\usepackage{amscd}
\usepackage{amsfonts}
\usepackage{amssymb}
\usepackage{latexsym}
\usepackage{color}
\usepackage{esint}
\usepackage{graphicx}
\usepackage{float}
\graphicspath{{Figures/}}

\usepackage{graphicx}

\usepackage[makeroom]{cancel}

\newtheorem{theorem}{Theorem}

\newtheorem{corollary}[theorem]{Corollary}
\newtheorem{definition}{Definition}
\newtheorem{lemma}{Lemma}
\newtheorem{proposition}[theorem]{Proposition}
\newtheorem{remark}{Remark}
\let\a=\alpha
\let\e=\varepsilon

\let\p=\partial

\let\O=\Omega

\numberwithin{equation}{section}

\let\hide\iffalse

\newcommand{\R}{\mathbb{R}}

\newcommand{\be}{\begin{equation}}
\newcommand{\bm}{\begin{multline}}
\newcommand{\ee}{\end{equation}}
\newcommand{\dd}{\mathrm{d}}

\newcommand{\xb}{x_{\mathbf{b}}}
\newcommand{\tb}{t_{\mathbf{b}}}

\newcommand{\Bes}{\begin{eqnarray*}}
\newcommand{\Ees}{\end{eqnarray*}}
\newcommand{\Be}{\begin{equation} }
\newcommand{\Ee}{\end{equation}}

\def\p{\partial}

\def\O{\Omega}
\def\R{\mathbb{R}}

\def\B{\begin{equation}}
\def\E{\end{equation}}
\def\BN{\begin{eqnarray*}}
\def\EN{\end{eqnarray*}}

\begin{document}
\title{Regularity of Boltzmann equation with Cercignani-Lampis boundary in convex domain.}

\author{Hongxu Chen}
\address{Mathematics Department, University of Wisconsin-Madison, 480 Lincoln Dr., Madison, WI 53705 USA.}
\email{hchen463@wisc.edu}

 \begin{abstract}
The Boltzmann equation is a fundamental kinetic equation that describes the dynamics of dilute gas. In this paper we study the regularity of both dynamical and steady Boltzmann equation in strictly convex domain with the Cercignani-Lampis (C-L) boundary condition. The C-L boundary condition  describes the intermediate reflection law between diffuse reflection and specular reflection via two accommodation coefficients.
We construct local weighted $C^1$ dynamical solution using repeated interaction through the characteristic. When we assume small fluctuation to the wall temperature and accommodation coefficients, we construct weighted $C^1$ steady solution.
 \end{abstract}

\maketitle

\section{Introduction}
In this paper we consider the classical Boltzmann equation, which describes the dynamics of dilute particles. Denoting $F(t,x,v)$ the phase-space-distribution function of particles at time $t$, location $x\in\Omega$ moving with velocity $v\in\mathbb{R}^3$, the equation writes:
\begin{equation}\label{eqn: VPB equation}
\partial_t F + v\cdot \nabla_x F = Q(F,F)\,.
\end{equation}
The collision operator $Q$ describes the binary collisions between particles:
\begin{equation}\label{eqn: Q}
\begin{split}
   & Q(F_1,F_2)(v)=Q_{\text{gain}}-Q_{\text{loss}} = Q_{\text{gain}}(F_1,F_2)-\nu(F_1)F_2 \\
    & :=\iint_{\mathbb{R}^3\times \mathbb{S}^2} B(v-u,\omega)F_1(u')F_2(v') \dd \omega \dd u-F_2(v)\left(\iint_{\mathbb{R}^3\times \mathbb{S}^2} B(v-u,\omega)F_1(u) \dd \omega \dd u\right)\,.
\end{split}
\end{equation}
In the collision process, we assume the energy and momentum are conserved. We denote the post-velocities:
\begin{equation}\label{eqn: u' v'}
u'=u-[(u-v)\cdot \omega]\omega ,\quad \quad v'=v+[(u-v)\cdot \omega]\omega\,,
\end{equation}
then they satisfy:
\begin{equation}\label{eqn: conservation}
  u'+v'=u+v\,,\quad |u'|^2+|v'|^2=|u|^2+|v|^2\,.
\end{equation}
In equation~\eqref{eqn: Q}, $B$ is called the collision kernel. In this paper we only consider the hard sphere case, which is given by
\[
B(v-u,\omega)=|v-u|^{\mathcal{K}}q_0(\frac{v-u}{|v-u|}\cdot \omega)\,,\quad\text{with }\mathcal{K}= 1\,,\quad 0\leq q_0(\frac{v-u}{|v-u|}\cdot \omega)\leq C\Big|\frac{v-u}{|v-u|}\cdot \omega\Big|\,.
\]

To describe the boundary condition for $F$, we denote the collection of coordinates on phase space at the boundary:
\[\gamma:=\{(x,v)\in \partial \Omega\times \mathbb{R}^3\}.\]
And we denote $n=n(x)$ as the outward normal vector at $x\in \Omega$. We split the boundary coordinates $\gamma$ into the incoming ($\gamma_-$) and the outgoing ($\gamma_+$) set:
\[\gamma_\mp:=\{(x,v)\in \partial \Omega\times \mathbb{R}^3 :n(x)\cdot v\lessgtr 0\}\,.\]
The boundary condition determines the distribution on $\gamma_-$, and shows how particles back-scattered into the domain. In our model, we use the scattering kernel $R(u \rightarrow v;x,t)$:
\begin{equation}\begin{split}\label{eqn:BC}
&F(t,x,v) |n(x) \cdot v|= \int_{n(x) \cdot u>0}
R(u \rightarrow v;x,t) F(t,x,u)
\{n(x) \cdot u\} \dd u, \quad \text{ on }\gamma_-\,.
\end{split}
\end{equation}
$R(u\to v;x,t)$ represents the probability of a molecule striking in the boundary at $x\in\partial\Omega$ with velocity $u$, and to be sent back to the domain with velocity $v$ at the same location $x$ and time $t$.
In this paper we consider a scattering kernel proposed by Cercignani and Lampis in~\cite{CIP,CL}:
\begin{equation}\label{eqn: Formula for R}
\begin{split}
&R(u \rightarrow v;x,t)\\
:=& \frac{1}{r_\perp r_\parallel (2- r_\parallel)\pi/2} \frac{|n(x) \cdot v|}{(2T_w(x))^2}
\exp\left(- \frac{1}{2T_w(x)}\left[
\frac{|v_\perp|^2 + (1- r_\perp) |u_\perp|^2}{r_\perp}
+ \frac{|v_\parallel - (1- r_\parallel ) u_\parallel|^2}{r_\parallel (2- r_\parallel)}
\right]\right)\\
& \times  I_0 \left(
 \frac{1}{2T_w(x)}\frac{2 (1-r_\perp)^{1/2} v_\perp u_\perp}{r_\perp}
\right),
\end{split}
\end{equation}
where $T_w(x)$ is the wall temperature for $x\in \partial \Omega$ and
\begin{equation}\label{I0}
I_0 (y) := \pi^{-1} \int^{\pi}_0e^{y \cos \phi } \dd \phi\,.
\end{equation}
In the formula, $v_\perp$ and $v_\parallel$ denote the normal and tangential components of the velocity respectively:
   \begin{equation}\label{eqn: def of vperppara}
   v_\perp= v\cdot n(x)\,,\quad v_\parallel = v- v_\perp n(x)\,.
\end{equation}
Similarly $u_\perp= u\cdot n(x)$ and $u_\parallel = u- u_\perp n(x)$. There are other derivations of C-L model besides the original one, and we refer interested readers to~\cite{C,CIP,CC}.

The Cercignani-Lampis(C-L) model satisfies the following properties:
\begin{itemize}
\item the reciprocity property:
  \begin{equation}\label{eqn: reciprocity}
    R(u\to v;x,t)=R(-v\to -u;x,t) \frac{e^{-|v|^2/(2T_w(x))}}{e^{-|u|^2/(2T_w(x))}}\frac{|n(x)\cdot v|}{|n(x)\cdot u|}\,,
  \end{equation}
\item the normalization property(see Lemma \ref{Lemma: Prob measure})
\begin{equation}\label{eqn: normalization}
\int_{n(x)\cdot v<0} R(u\to v;x,t) \dd v=1\,.
\end{equation}
\end{itemize}

The normalization~\eqref{eqn: normalization} property immediately leads to null-flux condition for $F$:
\begin{equation}\label{eqn: Null flux condition}
  \int_{\mathbb{R}^3}F(t,x,v)\{n(x)\cdot v\}\dd v=0\,,\quad \text{for }x\in \partial\Omega.
\end{equation}
This condition guarantees the conservation of total mass:
\begin{equation}\label{eqn: Mass conservation}
  \int_{\Omega\times \mathbb{R}^3}F(t,x,v)\dd v \dd x=\int_{\Omega\times \mathbb{R}^3}F(0,x,v)\dd v\dd x \text{ for all }t\geq 0\,.
\end{equation}

\begin{remark}The C-L model encompasses pure diffusion and pure reflection.

The pure diffuse boundary condition is given by
\begin{equation}\label{eqn: diffuse}
 F(t,x,v)= \frac{2}{\pi (2T_w(x))^2}e^{-\frac{|v|^2}{2T_w(x)}}\int_{n(x)\cdot u>0} F(t,x,u)\{n(x)\cdot u\}\dd u \text{ on }(x,v)\in\gamma_-,
\end{equation}
\[R(u\to v;x,t)=\frac{2}{\pi (2T_w(x))^2}e^{-\frac{|v|^2}{2T_w(x)}} |n(x)\cdot v|.\]
It corresponds to the scattering kernel in~\eqref{eqn: Formula for R} with $r_\perp=1,r_\parallel=1$.

Other basic boundary conditions can be considered as a special case with singular $R$: specular reflection boundary condition:
\[F(t,x,v)=F(t,x,\mathfrak{R}_xv) \text{ on }(x,v)\in \gamma_-,  \quad \mathfrak{R}_xv=v-2n(x)(n(x)\cdot v),\]
\[R(u\to v;x,t)=\delta(u-\mathfrak{R}_xv),\]
where $r_\perp=0,r_\parallel=0$.

Bounce-back reflection boundary condition:
\[ F(t,x,v)=F(t,x,-v) \text{ on } (x,v)\in \gamma_-,\]
\[R(u\to v;x,t)=\delta(u+v),\]
where $r_\perp=0,r_\parallel=2$.

\end{remark}

Due to the generality of the C-L model, it has been vastly used in many field, on the rarefied gas flow in~\cite{KB,KB2,SF1,SF2,SF3}; extension to the gas surface interaction model in fluid dynamics~\cite{L,L2,WR}; on the linearized Boltzmann equation in~\cite{Gar,SI,LS,CS}; on S-model kinetic equation in~\cite{CES} etc.

In this paper we will study the regularity of both the dynamical and steady Boltzmann equation with C-L boundary. The Boltzmann equation with scattering type boundary condition~\eqref{eqn:BC} has been studied in many aspects. \cite{G,G2,KL,KL2,KY} studied the dynamical solution with diffuse, specular and bounce back boundary condition. With such boundary condition, \cite{CJ1,CJ2,EGKM3,CKJ} studied the fluid limit of the Boltzmann equation. Moreover, a unique stationary solution has been constructed in~\cite{EGKM,EGKM2,Duan}. Inspired by these studies, in~\cite{chen} the author constructed a unique local dynamical solution and a unique steady solution with C-L boundary in bounded domain.

In non-convex domain the Boltzmann equation possess a boundary singularity \cite{K}, and BV is the best estimate we can expect~\cite{GKTT2}. In convex domain \cite{GKTT} proposed a kinetic weight to construct a unique $C^1$ and $W^{1,p}$ dynamical solution. With convex domain the kinetic weight can be further applied to study the Vlasov-Poisson-Boltzmann system~\cite{CKL,CKQ,CaoK,Cao,Cao2}. In terms of the steady solution, \cite{Ikun} studied the regularity of the stationary linearized Boltzmann equation. Recently a unique weighted $C^1$ steady solution in convex domain has been constructed in~\cite{CK}. Our work in this paper originate from these studies and focus on both the dynamical and steady solution.

Throughout this paper we assume the domain is $C^2$ and defined as $\O = \{x \in \R^3: \xi(x) <0\}$ via a $C^2$ function $\xi : \R^3 \rightarrow \R$. We further assume that the domain is strictly convex in the following sense:
 \Be\label{convex}
\sum_{i,j=1}^3\zeta_i \zeta_j\p_i\p_j \xi(x) \gtrsim |\zeta|^2  \  \text{ for   all }   x \in \bar{\O}  \text{ and }  \zeta \in \R^3.
 \Ee
Without loss of generality we may assume that $\nabla \xi \neq 0$ near $\p\O$.

Denote the maximum and minimum wall temperature as:
\begin{equation}\label{eqn: T}
T_M:=\max\{T_w(x)\}<\infty\, \quad T_m:=\min\{T_w(x)\}>0.
\end{equation}

It is well known that singularity propagates for the derivative in the boundary value problem~\cite{K}. In order to control the generic singularity at the boundary we adopt the following weight of \cite{GKTT}:
\begin{definition} For sufficiently small $0<\e\ll \Vert \xi\Vert_{C^2}$, we define a \text{kinetic distance}:
\Be\label{kinetic_distance}
\begin{split}
\alpha(x,v)   : = \chi_\e ( \tilde{\alpha}(x,v)  ) ,  \ \
\tilde{\alpha}(x,v)   : = \sqrt{ |v \cdot \nabla_x \xi(x)|^2 - 2 \xi(x) (v \cdot \nabla_x^2 \xi(x) \cdot v)}, \ \ (x,v) \in \bar{\O} \times \R^3,
\end{split}
\Ee
where $\chi_a: [0,\infty) \rightarrow [0,\infty)$ stands for a non-decreasing smooth function such that
\Be\label{chi}
\chi_\a (s) = s \ for \ s \in [0, a ],  \ \chi_a (s) = 2a \ for \ s \in [ 4 a, \infty ], \ and \  | \chi^\prime_a(s) | \leq 1 \ for  \  \tau  \in [0,\infty).
\Ee
\end{definition}
The definition of $\xi$ in~\eqref{convex} implies that $\xi(x)=0, x\in \partial \Omega$,
\begin{equation}\label{cancel n dot v}
\text{when } x\in \partial \Omega \text{ and } |n(x)\cdot v| \ll 1, \alpha(x,v) \sim n(x) \cdot v.
\end{equation}
We will use this kinetic weight to cancel the singularity on the boundary. Lemma \ref{Lemma: velocity lemma} indicates that such weight is almost invariant along the trajectory.

Denote
\begin{equation}\label{Def: w_theta}
  w_{\theta}:=e^{\theta |v|^2},
\end{equation}
\begin{equation}\label{eqn: langle v rangle}
  \langle v\rangle:=\sqrt{|v|^2+1}.
\end{equation}
\subsection{Result of dynamical Boltzmann equation}

Define the global Maxwellian using the maximum wall temperature:
\begin{equation}\label{eqn: def for weight}
\mu:=e^{-\frac{|v|^2}{2T_M}}\,,
\end{equation}
and weight $F$ in~\eqref{eqn: VPB equation} with it: $F=\sqrt{\mu}f$. Then $f$ satisfies
\begin{equation}\label{equation for f}
  \partial_t f+v\cdot \nabla_x f=\Gamma(f,f)\,,
\end{equation}
where the collision operator becomes:
\begin{equation}\label{Def: Gamma}
\Gamma(f_1,f_2)=\Gamma_{\text{gain}}(f_1,f_2)-\nu(F_1)F_2/\mu=\frac{1}{\sqrt{\mu}}Q_{\text{gain}}(\sqrt{\mu}f_1,\sqrt{\mu}f_2)-\nu(F_1)f_2\,.
\end{equation}

The weighted $C^1$ estimate is given in the following theorem.
\begin{theorem}\label{Thm: dynamic C1}
Assume $\Omega \subset \mathbb{R}^3$ is bounded, convex and $C^2$. Let $0< \theta <\frac{1}{4T_M}$. Assume
\begin{equation}\label{eqn: r condition}
  0<r_\perp\leq 1,\quad 0<r_\parallel<2\, ,
\end{equation}
\begin{equation}\label{eqn: Constrain on T}
\frac{T_m}{T_M}>\max\Big(\frac{1-r_\parallel}{2-r_\parallel},\frac{\sqrt{1-r_\perp}-(1-r_\perp)}{r_\perp}\Big),
\end{equation}
where $T_M,T_m$ are defined in~\eqref{eqn: T}.

Also assume the initial condition has bound
\begin{equation}\label{initial C1}
  \Vert \alpha \nabla_{x,v} f_0\Vert_\infty<\infty.
\end{equation}
Then for some
\[t_\infty = t_\infty(\Vert w_{\theta}f\Vert_\infty,r_\perp,r_\parallel,\theta,T_M,T_m,\Omega )\ll 1,\]
we can construct a unique solution $F=\sqrt{\mu} f$ satisfies
\begin{equation}\label{C1 estimate}
\sup_{0\leq t\leq t_\infty}\Vert e^{-\lambda \langle v\rangle t}\alpha \nabla_{x,v} f\Vert_\infty \lesssim \Vert \alpha\nabla_{x,v} f_0\Vert_\infty.
\end{equation}
Here $\alpha$ is the kinetic weight defined in~\eqref{kinetic_distance}, $\lambda\geq 1$ is a constant specified in~\eqref{nu m geq}, $\Vert w_{\theta}f\Vert_\infty$ is the $L^\infty$ estimate given in Theorem \ref{Thm: local existence} with $w_\theta(v)=e^{\theta|v|^2}$.
\end{theorem}

\begin{remark}
The well-posedness of the solution $F=\sqrt{\mu} f$ and $L^\infty$ estimate $\Vert w_{\theta} f\Vert_\infty$ are proved in~\cite{chen}, in this paper we will focus on the weighted $C^1$ estimate~\eqref{C1 estimate}. We record the well-posedness and $L^\infty$ estimate in Theorem \ref{Thm: local existence} in section 2.
\end{remark}

\begin{remark}
In Theorem \ref{Thm: dynamic C1} the accommodation coefficient can be any number except $r_\perp=0,r_\parallel=0,2$, which corresponds to pure reflection or bounce back reflection.
For wall temperature we have a relaxed condition~\eqref{eqn: Constrain on T} rather than the small fluctuation. In particular, for the pure diffuse reflection, i.e, $r_\parallel=r_\perp=1$, there is no constraint to the temperature(except $T_M<\infty,T_m>0$).
\end{remark}

\subsection{Result of steady Boltzmann equation}
We also establish the weighted $C^1$-estimate for the steady problem. The steady Boltzmann equation is given as
\begin{equation}\label{eqn: Steady Boltzmann}
  v\cdot \nabla_x F_s=Q(F_s,F_s),\quad (x,v)\in \Omega \times \mathbb{R}^3,
\end{equation}
with $F_s$ satisfying the C-L boundary condition. Here we note that we use $F_s$ to represent the steady solution.

We use the short notation $\mu_0$ to denote the global Maxwellian with temperature $T_0$,
\[\mu_0:=\frac{1}{2\pi (T_0)^2}\exp\Big(-\frac{|v|^2}{2T_0} \Big).\]
Here we mark that $\mu_0$ is the global Maxwellian for the steady problem while the $\mu$ defined in~\eqref{eqn: def for weight} is the global Maxwellian for the dynamical problem.

Let $F_s=\mu_0+\sqrt{\mu_0} f_s$. The equation of $f_s$ reads
\begin{equation}\label{equation of f_s}
v\cdot \nabla_x f_s + Lf_s = \Gamma(f_s,f_s).
\end{equation}
Here $L$ is the standard linearized Boltzmann operator
\begin{equation}\label{eqn: L operator}
\begin{split}
   & Lf_s:=-\frac{1}{\sqrt{\mu_0}}\big[ Q(\mu_0,\sqrt{\mu_0}f_s)+Q(\sqrt{\mu_0}f_s,\mu_0)\big]=\nu(v)f_s-Kf_s
\end{split}
\end{equation}
with the collision frequency $\nu(v)\equiv \iint_{\mathbb{R}^3\times \mathbb{S}^2}B(v-v_*,w)\mu_0(v_*)\dd w \dd v_*\sim \{1+|v|\}$. When we assume small fluctuation of the wall temperature and the accommodation coefficient, the steady problem is well-posed~\cite{chen}.

In this paper we also derive the weighted-$C^1$ regularity of the steady solution in the following theorem.
\begin{theorem}\label{thm: C1 steady}
For given $T_0>0$, there exists $\delta_0>0$ such that if
\begin{equation}\label{eqn: small pert condition}
\sup_{x\in \partial \Omega}|T_w(x)-T_0|<\delta_0,\quad \max\{|1-r_\perp|,|1-r_\parallel|\}<\delta_0,
\end{equation}
then we can construct a unique steady solution $F_s=\mu_0+\sqrt{\mu_0}f_s$ satisfies:
\begin{equation}\label{C1 steady estimate}
\Vert \alpha \nabla_x f_s\Vert_\infty \lesssim \Vert w_\vartheta f_s\Vert_\infty  \lesssim 1.
\end{equation}
Here $w_\vartheta=e^{\vartheta|v|^2}$ for some $\vartheta>0$.

\end{theorem}

\begin{remark}
The well-posedness of $f_s$ and the $L^\infty$ bound $\Vert w_\vartheta f_s\Vert_\infty$ of the steady solution $F_s=\mu_0+\sqrt{\mu_0}f_s$ is proved in~\cite{chen}, we record the result in Corollary \ref{Thm: steady solution} in section 2. In this paper we focus on proving the regularity estimate~\eqref{C1 steady estimate}.
\end{remark}

\begin{remark}
In Theorem \ref{thm: C1 steady}, different to Theorem \ref{Thm: dynamic C1}, we need to restrict these two coefficients to be close to $1$ in~\eqref{eqn: small pert condition}. To be more specific, we require the C-L boundary to be close to the diffuse boundary condition.
\end{remark}

\subsection{Difficulty and proof strategy}

\textbf{Dynamical solution.}
First we illustrate the difficulty and strategy for the dynamical solution in Theorem \ref{Thm: dynamic C1}. A common approach for the boundary value problem is to iterate along the backward characteristic until hitting the boundary or the initial datum.
In order to clearly state and address the difficulty, we briefly recall the strategy for the well-posedness of the dynamical solution as stated in~\cite{chen}. We define the stochastic cycle $v_k,v_{k-1},\cdots,v_1$ in Definition \ref{Def:Back time cycle}. The backward characteristic may hit the boundary for $k$-times before reaching the initial datum. The boundary condition~\eqref{eqn:C-L boundary condition in pro measure} will generate a $k$-fold integration. Due to the probability measure $\dd \sigma(v_k,v_{k-1})$ ( see~\eqref{eqn:probability measure}),
the integral of $v_k$ is roughly
\begin{equation}\label{eqn: proof strat}
 \int_{n(x)\cdot v_k>0}   e^{-[\frac{1}{4T_M}-\frac{1}{2T_w(x)}]|v_k|^2}  \dd \sigma(v_k,v_{k-1}).
\end{equation}
Indeed the integrand is of the form of exponential, we can explicitly compute the above integration as a function of $v_{k-1}$ and adapt the result to the integration over $v_{k-1}$. In such way
we can derive an induction formula to compute the $k$-fold integration.

For the rest stochastic cycle, i.e, $v_{k+1},v_{k+2},\cdots$, for large $k$, physically it means the characteristic does not reach the initial datum after a large number of interaction with boundary.
We follow the idea in~\cite{GKTT} to introduce the grazing set
\[\gamma_+^\delta = \{u\in \gamma_+: |n\cdot u|>\delta, |u|\leq \delta^{-1}\}.\]
In such subspace characteristic need to take certain time to reach the boundary. One can derive the lower bound of the time as $O(\frac{1}{\delta^3})$. For bounded $t$ there can be at most $N=O(\frac{1}{\delta^3})$ many $v_i$ belong to such subspace.
For the rest $v_i \in \gamma_+\backslash \gamma_+^\delta$, the integration over such subspace results in a small magnitude number $O(\delta)$. Thus for large $k$, we get a large power of $O(\delta)$ and thus derive that the measure of
the rest cycle $v_{k+1},v_{k+2},\cdots$ is small. Hence we will choose a proper $k$ depend on $N=O(\frac{1}{\delta^3})$.

When it comes to the regularity, it is well-known that singularity occurs at the backward exit position $\xb(x,v):= x-  \tb(x,v) v$ which is defined through a backward exit time $\tb$:
\Be\label{BET}
 \tb(x,v) :=   \sup \{ s>0:x-sv \in \O\}.
\Ee
Thus while estimating the regularity, the singularity occurs at the boundary. In Theorem \ref{Thm: dynamic C1} we include the kinetic weight~\eqref{kinetic_distance} since such weight can cancel the singularity as stated in Lemma \ref{Lemma: velocity lemma}. Besides the singularity the boundary condition for $\partial f$ actually has a nice form as stated in Lemma~\ref{Lemma: bc for partial f}, which looks similar to the boundary condition of $f$ in~\eqref{eqn:C-L boundary condition in pro measure}.

Even though the boundary condition in our case is similar to the case of $f$, the extra term $\langle v\rangle^2$ in Lemma~\ref{Lemma: bc for partial f}, brings difficulty to our analysis. Since the computation involves various integration with exponential, it is natural to bound polynomial term $\langle v\rangle$ by exponential and adapt it into the computation. For a single integration such upper bound does not have big effect.
However, as stated above, we trace back along the characteristic for large $k$ times.
Thus in order to follow the induction formula for the $k$-fold integration, we need to bound
\begin{equation}\label{intro: poly}
\langle v\rangle^2 \lesssim \frac{1}{\e} e^{\e|v|^2}
\end{equation}
with small enough coefficient $\e$. Such extra exponential term $e^{\e|v|^2}$ will slightly increase the coefficient of the exponential after an integration. With a $k$-fold integration we need to impose the $k$-dependence on $\e$.
Since $k$ depends on $N=O(1/\delta^3)$, the term $\frac{1}{\e(k)}$ in~\eqref{intro: poly} depends on $\delta$ as well. It will be combined with the small magnitude number $O(\delta)$ for the nongrazing set $\gamma_+\backslash \gamma_+^\delta$.
Then in order to derive the smallness, we need to ensure $\delta\ll \e(\delta)$. Unfortunately, with such properties, the $k$-fold integration does not remain bounded.

To overcome such difficulty a key observation is: since we consider local-in-time $[0,t]$, we can obtain a better bound for $N$ as $O(\frac{t}{\delta^3})$.
Thus we can write $\delta=t^{1/3}\delta'$ for some $\delta'\ll 1$. Since we are considering local-in-time regularity, $t$ can be finally designed to be small and depend on all the other variables $k,\delta,\cdots$. In such setting $k=k(N)=k(\delta')$, which does not depend on $t$. With the extra $t^{1/3}$ we can choose proper $\e$ to satisfy the condition $\delta \ll \e$ as follow:
instead of imposing the $k$ dependence on $\e$, we directly impose the $t$ dependence as in Lemma~\ref{Lemma: extra term}. Then we assume $\e=t^{c}$ for some $\frac{1}{3}>c>0$ and incorporate $\frac{1}{t^c}$ with $\delta=t^{1/3}\delta'$ in the computation. Finally, we choose $t$ to be small to ensure the $k$-fold integration is bounded. In order to obtain the smallness for the rest cycle $v_{k+1},v_{k+2},\cdots$ in~\eqref{c}, we specify $c=1/15$.

\textbf{Steady solution.}
Then we come to the steady solution in Theorem~\ref{thm: C1 steady}. We express the steady solution as perturbation around a global Maxwellian $F=\mu_0+\sqrt{\mu_0}f, \mu_0=e^{-\frac{|v|^2}{2T_0}}$ and trace back along the characteristic as~\eqref{f1}-\eqref{f5}. The weighted $C^1$ regularity with pure diffuse boundary condition is established in~\cite{CK}, we use the same method to deal with the collision term(not related to the boundary). Then the new difficulty comes from the boundary term.
The boundary condition for $f$ can be computed as in Lemma \ref{Lemma: bc for fs}. Thus the most singular term from the boundary reads
\begin{equation}\label{intro: bdr sing}
\nabla_x \xb  e^{[\frac{1}{4T_0}-\frac{1}{2T_w(\xb)}]|v|^2} \int_{n(\xb)\cdot v_1>0} \nabla_{\xb} f(\xb,v_1) e^{-[\frac{1}{4T_0}-\frac{1}{2T_w(\xb)}]|v_1|^2} \dd \sigma(v_1,v).
\end{equation}

Using the characteristic once again for $f(\xb,v_1)$, the contribution of the collision operator (ignoring the singularity of $Q$ for simplicity) can be viewed as
\Be\label{intro_boundary1_int}
\nabla_x \xb  e^{[\frac{1}{4T_0}-\frac{1}{2T_w(\xb)}]|v|^2} \int_{n(\xb)\cdot v_1>0} \int_0^{\tb(\xb,v_1)}\nabla_x f(\xb-sv_1,v_1) e^{-[\frac{1}{4T_0}-\frac{1}{2T_w(\xb)}]|v_1|^2} \dd u \dd s\dd \sigma(v_1,v).\Ee
We can exchange the $x-$derivative into $v_1$-derivative as
\Be\label{x_to_v}
\nabla_x f(\xb-(t_1-s)v_1,u)= \frac{\nabla_{v_1} [F(\xb-(t_1-s)v_1,u) ]}{-(t_1-s)}.
\Ee
Since the accommodation coefficient and wall temperature are assumed to have a small fluctuation as in~\eqref{eqn: small pert condition}, such integration is ``close" to the integration of the pure diffuse boundary condition.
Then we can apply the change of variable to remove the $v^1$-derivative completely from $f$. Different to the pure diffuse boundary condition, the C-L boundary will generate more polynomial factors due to the
normal and tangential components in~\eqref{eqn: Formula for R}. Thanks to exponential decay term in the integrand, the polynomial factors will not affect the integrability. In Lemma \ref{Lemma: integrable} we compute these integration with extra polynomial terms, extra derivative in detail.
Thus the integration can be bounded by $\nabla_x \xb \Vert f\Vert_\infty$.

Another singular term is the boundary contribution of (\ref{intro: bdr sing}) along the characteristic:
\Be\label{intro: bdr sing2}
\nabla_x \xb  e^{[\frac{1}{4T_0}-\frac{1}{2T_w(\xb)}]|v|^2} \int_{n(\xb)\cdot v_1>0} \nabla_{\xb} f(\xb(x_1,v_1),v_1) e^{-[\frac{1}{4T_0}-\frac{1}{2T_w(\xb)}]|v_1|^2} \dd \sigma(v_1,v).
\Ee
The key idea is to convert $v_1$-integration to the integration in
$(x_2,  \tb(x_1,v_1) )=(\xb(x_1,v_1), \tb(x_1,v_1) )$, with Jacobian given in Lemma \ref{Lemma: change of variable}. Then we are able to remove $\nabla_{\xb}$-derivative from $f$ via the integration by parts.
Similar to the collision term~\eqref{intro_boundary1_int}, the integration by parts will generate more polynomial factors. These factors won't affect the integrability. Thus we can again remove the derivative and bound such contribution by $\nabla_x \xb \Vert f\Vert_\infty$.

\subsection{Outline}

In section 2 we list several lemmas as preparation. In section 3 we derive the weighted $C^1$ bound for the dynamical solution and conclude Theorem \ref{Thm: dynamic C1}. In section 4 we derive the weighted $C^1$ bound for the steady solution and conclude Theorem \ref{thm: C1 steady}.

\section{Preliminary}
\subsection{Basic setting}
Throughout this paper we will use the following notation:
\begin{equation}\label{lesssim}
f \lesssim  g \Leftrightarrow \text{there exists $0<C<\infty$ such that } f\leq Cg.
\end{equation}
\begin{equation}\label{big O}
f = O(g) \Leftrightarrow \text{there exists $0<C<\infty$ such that } f = Cg.
\end{equation}
\begin{equation}\label{little o}
f = o(g) \Leftrightarrow \text{there exists $c \ll 1$ such that } f = cg.
\end{equation}

First we record the local well-posedness of the dynamical Boltzmann equation with the C-L boundary.
\begin{theorem}\label{Thm: local existence}
Assume $\Omega \subset \mathbb{R}^3$ is bounded and $C^2$. Let $0< \theta <\frac{1}{4T_M}$. Assume wall temperature satisfies~\eqref{eqn: r condition} and~\eqref{eqn: Constrain on T}. If $F_0= \sqrt{\mu}f_0\geq 0$ and $f_0$ satisfies
\begin{equation}\label{eqn: w f_0}
\| w_\theta f_0 \|_\infty		< \infty,
\end{equation}
then there exists a unique solution $F(t,x,v) =  \sqrt{\mu} f(t,x,v)\geq 0$ to~\eqref{eqn: VPB equation} and~\eqref{eqn:BC} in $[0, t_{dym}] \times \O \times \R^3$ for some $t_{dym}\ll 1$.
Moreover, the solution $F=\sqrt{\mu}f$ satisfies
		\Be\begin{split}\label{infty_local_bound} \sup_{0 \leq t \leq t_{dym}}
			\| w_{\theta}e^{-|v|^2 t} f  (t) \|_{\infty}
			\lesssim \| w_\theta f_0 \|_\infty .
		\end{split}\Ee
\end{theorem}

Then we record the well-posedness of the steady Boltzmann equation with the C-L boundary.
\begin{corollary}\label{Thm: steady solution}
For given $T_0>0$, if the wall temperature and accommodation coefficient satisfies~\eqref{eqn: small pert condition}, then there exists a unique non-negative solution $F_s=\mu_0+\sqrt{\mu_0}f_s\geq 0$ with $\iint_{\Omega\times \mathbb{R}^3}f_s\sqrt{\mu_0}\dd x\dd v=0$ to the steady problem~\eqref{eqn: Steady Boltzmann}. And for some $\vartheta>0$,
\[\Vert e^{\vartheta |v|^2}f_s\Vert_\infty \lesssim \delta_0\ll 1.\]
\end{corollary}

\begin{definition}\label{Def:Back time cycle}
Let $\left(X^1(s;t,x,v),v\right)$ be the location and velocity along the backward trajectory before hitting the boundary,
\begin{equation}\label{eqn: trajectory for Xm}
  \frac{\dd}{\dd s}\left(
                \begin{array}{c}
                  X^1(s;t,x,v) \\
                  v \\
                \end{array}
              \right)=\left(
                        \begin{array}{c}
                          v \\
                          0\\
                        \end{array}
                      \right).
\end{equation}
Therefore, from~\eqref{eqn: trajectory for Xm}, we have
\[X^1(s;t,x,v)=x-(t-s)v.\]
Define the back-time cycle as
\[t_{1}(t,x,v)=\sup\{s<t:X^1(s;t,x,v)\in \partial \Omega\},\]
\[x_{1}(t,x,v)=X^1\left(t_1(t,x,v);t,x,v\right),\]
\[v_1\in \{v_1\in \mathbb{R}^3:n(x_1)\cdot v_1>0\}.\]
Also define
\[\mathcal{V}_1=\{v_1:n(x_1)\cdot v_1>0\},\quad x_1\in \partial \Omega.\]

Inductively, before hitting the boundary for the $k$-th time, define
\[t_k(t,x,v,v_1,\cdots,v_{k-1})=\sup\{s<t_{k-1}:X^k(s;t_{k-1},x_{k-1},v_{k-1})\in \partial \Omega\},\]
\[x_k(t,x,v,v_1,\cdots,v_{k-1})=X^k\left(t_k(t,x,v,v_{k-1});t_{k-1}(t,x,v),x_{k-1}(t,x,v),v_{k-1}\right),\]
\[v_k\in \{v_k\in \mathbb{R}^3:n(x_k)\cdot v_k>0\},\]
\[\mathcal{V}_k=\{v_k:n(x_k)\cdot v_k>0\},\]
\[X^k(s;t_{k-1},x_{k-1},v_{k-1})=x_{k-1}-(t_{k-1}-s)v_{k-1}.\]
Here we set
\[(t_0,x_0,v_0)=(t,x,v).\]
For simplicity, we denote
\[X^k(s):=X^k(s;t_{k-1},x_{k-1},v_{k-1})\]
for the rest lemmas and propositions.

\end{definition}

\subsection{Properties of the C-L scattering kernel}
In this subsection we list some basic properties of the scattering kernel~\eqref{eqn: Formula for R}.

\begin{lemma}\label{Lemma: Prob measure}(Lemma 10 in~\cite{chen})\\
For $R(u\to v;x,t)$ given by~\eqref{eqn: Formula for R} and any $u$ such that $n(x)\cdot u>0$, we have
\begin{equation}\label{eqn: integrate 1}
  \int_{n(x)\cdot v<0}R(u\to v;x,t)\dd v=1.
\end{equation}

\end{lemma}

\begin{lemma}\label{Lemma: abc}(Lemma 11 in~\cite{chen})\\
For any $a>0,b>0$, $\e>0$ such that $a+\e<b$, we have
\begin{equation}\label{eqn: coe abc}
\frac{b}{\pi}\int_{\mathbb{R}^2} e^{\e|v|^2}  e^{a|v|^2}e^{-b|v-w|^2}\dd v\leq \frac{b}{b-a-\e}e^{\frac{(a+\e)b}{b-a-\e}|w|^2}.
\end{equation}

And when $\delta\ll 1$,
\begin{align}
     \frac{b}{\pi}\int_{|v-\frac{b}{b-a-\e}w|>\delta^{-1}} e^{\e|v|^2} e^{a|v|^2}e^{-b|v-w|^2}\dd v  &\leq  e^{-(b-a-\e)\delta^{-2}} \frac{b}{b-a-\e} e^{\frac{(a+\e)b}{b-a-\e}|w|^2} \notag \\
   & \leq  \delta \frac{b}{b-a-\e}e^{\frac {(a+\e)b}{b-a-\e}|w|^2} \label{eqn: coe abc small}.
\end{align}

\end{lemma}

%
%
%
%
%
%
%
%
%
%

\begin{lemma}\label{Lemma: perp abc}(Lemma 12 in~\cite{chen})\\
For any $a>0,b>0,\e>0$ with $a+\e<b$,
\begin{equation}\label{eqn: coe abc perp}
2b\int_{\mathbb{R}^+} e^{\e v^2}e^{av^2} e^{-bv^2}e^{-bw^2}I_0(2bv w)\dd v\leq\frac{b}{b-a-\e}e^{\frac{(a+\e)b}{b-a-\e}w^2}.
\end{equation}
And when $\delta\ll 1$,
\begin{equation}\label{eqn: coe abc perp small}
2b\int_{0< v<\delta}e^{\e v^2}e^{av^2} e^{-bv^2}e^{-bw^2}I_0(2bv w)\dd v\leq \delta\frac{b}{b-a-\e}e^{\frac{(a+\e)b}{b-a-\e}w^2}.
\end{equation}

\end{lemma}

\begin{lemma}\label{Lemma: integrate normal small}(Lemma 13 in~\cite{chen})\\
For any $m,n>0$, when $\delta\ll 1$, we have
\begin{equation}\label{eqn: smallness for i0}
2m^2\int_{\frac{n}{m}u_\perp+\delta^{-1}}^\infty     v_\perp e^{-m^2v_\perp^2}I_0(2mnv_\perp u_\perp)e^{-n^2u_\perp^2}\dd v_\perp \lesssim e^{-\frac{m^2}{4\delta^{2}}}.
\end{equation}
In consequence, for any $a>0,b>0,\e>0$ with $a+\e<b$,
\begin{eqnarray}
 2b\int_{\frac{b}{b-a-\e}w+\delta^{-1}}^\infty v e^{\e v^2}e^{av^2} e^{-bv^2}e^{-bw^2}I_0(2bv w)\dd v  &\leq  & e^{\frac{-(b-a-\e)}{4\delta^2}}\frac{b}{b-a-\e}e^{\frac{(a+\e)b}{b-a-\e}w^2} \label{eqn: coe perp smaller 2}\\
   &\leq  & \delta\frac{b}{b-a-\e}e^{\frac{(a+\e)b}{b-a-\e}w^2}. \label{eqn: coe perp small 2}
\end{eqnarray}

\end{lemma}

To tackle the difficulty mentioned in~\eqref{intro: poly}, we bound the polynomial by exponential in the following lemma.
\begin{lemma}\label{Lemma: extra term}
For $0<c<1$ and $\lambda>1$ we have the following the upper bound:
\begin{equation}\label{bound extra term}
\langle v\rangle^4  e^{\lambda\langle v\rangle t} \leq 2t^{-c/2} e^{t^c|v|^2}\leq t^{-c} e^{t^c|v|^2}.
\end{equation}
\end{lemma}
\begin{proof}
For $t<t^c\ll 1$ and $c<1$, we bound
\begin{align*}
e^{\lambda\langle v\rangle t}   & \leq e^{\lambda t} e^{\lambda|v|t} \leq 2e^{t^{c/2}(|v|^2+\lambda^2)}\leq 4e^{t^{c/2}|v|^2},
\end{align*}
\begin{equation*}
 \langle v\rangle^4    \leq t^{-c/2}e^{t^{c/2} |v|^2}.
\end{equation*}
In the first inequality we have used $t\ll 1$ to have
\[e^{\lambda t}<e^{t^{c/2}\lambda^2}<2,\quad      \lambda|v|<\lambda^2+|v|^2.\]

In the second inequality we have used
\begin{align*}
  \langle v\rangle & =1+|v|\leq 2+|v|^2 \leq t^{-c/8}+|v|^2 \leq t^{-c/8}e^{t^{c/8}|v|^2},
\end{align*}
where we have used the Taylor expansion for $e^{t^{\frac{c}{8}}|v|^2}$ in the last step.

Thus with $t\ll 1$ we conclude the lemma.

\end{proof}

\subsection{Properties of the collision kernel and kinetic weight}

The next lemma indicates the invariant property of $\alpha$ under the operator $v\cdot \nabla_x$.
\begin{lemma}\label{Lemma: velocity lemma}(Lemma 2 in~\cite{GKTT})\\
When the transport operator acts on $\alpha$, we have an upper bound
\begin{equation}\label{v nabla alpha bdd}
v\cdot \nabla_x \alpha \lesssim_\xi |v|\alpha(x,v).
\end{equation}
Moreover, there exists $C=C(\xi)$ such that for all $0\leq s_1,s_2\leq t$,
\begin{equation}\label{invariant}
e^{-C|v||s_1-s_2|}\alpha(s_1;t,x,v)\leq \alpha(s_2;t,x,v)\leq e^{C|v||s_1-s_2|}\alpha(s_1;t,x,v).
\end{equation}

\end{lemma}

We summarize the properties of the collision operator in the following lemma.
\begin{lemma}\label{Lemma: collision operator}(Lemma 12 in~\cite{CK})\\
The linearized Boltzmann operator in~\eqref{eqn: L operator} has the following form:
\begin{equation}\label{nu}
\nu(v)\equiv \iint_{\mathbb{R}^3\times \mathbb{S}^2}B(v-u,w)\mu_0(u)\dd w\dd u\sim 1+|v|,
\end{equation}
\begin{equation}\label{nabla nu}
|\nabla_v  \nu(v)|\lesssim 1,
\end{equation}

\begin{equation}\label{K}
Kf_s = \int_{\mathbb{R}^3}\mathbf{k}(v,u)f_s(u)\dd u,
\end{equation}
where $\mathbf{k}(v,u)\lesssim \mathbf{k}_\varrho(v,u)$ with
\begin{equation}\label{k_varrho}
\mathbf{k}_\varrho(v,u) = \frac{e^{-\varrho|v-u|^2}}{|v-u|} \text{ for some } \varrho>0.
\end{equation}

The $\mathbf{k}_\varrho(v,u)$ satisfies the following condition:
\begin{equation}\label{k_varrho L1}
\mathbf{k}_\varrho(v,u) \in L^1_u,\quad     \frac{\mathbf{k}(v,u)}{|v-u|}\in L^1_u.
\end{equation}
In consequence,
\begin{equation}\label{Kf bdd}
\Vert Kf_s\Vert_\infty \lesssim \Vert w_\vartheta f_s\Vert_\infty.
\end{equation}

The derivative of $\mathbf{k}(v,u)$ satisfies the following condition:
\begin{equation}\label{nabla u k}
|\nabla_u \mathbf{k}(v,u)|\lesssim \frac{\mathbf{k}_\varrho(v,u)}{|v-u|}.
\end{equation}

For $(i,j)=(1,2)$ or $(i,j)=(2,1)$, the nonlinear Boltzmann operator can be bounded as
\begin{equation}\label{gamma gain}
\Gamma_{\text{gain}}(f_1,f_2)\lesssim \Vert w_\theta f_i\Vert_\infty \int_{\mathbb{R}^3} \mathbf{k}_\varrho(v,u) |f_j(x,u)|\dd u.
\end{equation}
In consequence, we have
\begin{equation}\label{gamma bdd}
\Vert \Gamma(f_s,f_s)\Vert_\infty\lesssim \Vert w_\vartheta f_s\Vert_\infty^2,
\end{equation}

\begin{equation}\label{nabla gamma gain}
|\partial_{x,v} \Gamma_{\text{gain}}(f,f)|\lesssim \Vert w_\theta f\Vert_\infty \int_{\mathbb{R}^3} \mathbf{k}_\varrho(v,u) |\partial_{x,v} f(x,u)|\dd u + \Vert w_\theta f_s\Vert_\infty^2,
\end{equation}
\begin{equation}\label{nabla gamma}
|\nabla_x \Gamma(f_s,f_s)(x,v)|\lesssim \Vert w_\vartheta f_s\Vert_\infty \frac{\Vert \alpha \nabla_x f_s\Vert_\infty}{\alpha(x,v)} + \Vert w_\vartheta f_s\Vert_\infty \int_{\mathbb{R}^3} \mathbf{k}_\varrho(v,u) |\nabla_x f_s(x,u)|\dd u.
\end{equation}

\end{lemma}

\begin{lemma}\label{Lemma: k tilde}(Lemma 13 in~\cite{CK})\\
If $0<\frac{\tilde{\theta}}{4}<\varrho$, if $0<\tilde{\varrho}<  \varrho- \frac{\tilde{\theta}}{4}$,
\begin{equation}\label{k_theta}
\mathbf{k}_{\varrho}(v,u) \frac{e^{\tilde{\theta} |v|^2}}{e^{\tilde{\theta} |u|^2}} \lesssim  \mathbf{k}_{\tilde{\varrho}}(v,u) ,
\end{equation}
where $\mathbf{k}_{\varrho}$ is defined in~\eqref{k_varrho}.

\end{lemma}

When we integrate the collision operator $\partial\Gamma_{\text{gain}}(f,f)$ given in~\eqref{gamma gain}, to construct $\alpha-$weighted $C^1$ bound, the extra weight $\alpha$ appears in the denominator. The following lemma is desired to bound the integration of $\frac{1}{\alpha}$.
\begin{lemma}\label{Lemma: NLN}(Lemma 14 in~\cite{CK})
\begin{equation}\label{NLN}
\int_0^t e^{-\nu(t-s)}\int_{\mathbb{R}^3} \frac{\mathbf{k}_\varrho(v,u)}{\alpha(x-(t-s)v,u)} \dd u \lesssim \frac{t}{\alpha(x,v)}.
\end{equation}

\begin{equation}\label{NLN epsilon}
\int_{t-\e}^t e^{-\nu(t-s)}\int_{\mathbb{R}^3} \frac{\mathbf{k}_\varrho(v,u)}{\alpha(x-(t-s)v,u)} \dd u \lesssim \frac{O(\e)}{\alpha(x,v)}.
\end{equation}

\end{lemma}

\subsection{Reparametrization of boundary and stochastic cycle}
In this subsection we reparametrize the boundary and stochastic cycle in Definition \ref{Def:Back time cycle}. We will mainly use the reparametrization in section 4 to prove Theorem \ref{thm: C1 steady}.

We assume that for all $q\in \partial \Omega$, there exists $0<\delta_1 \ll 1$
\Be\label{O_p}
\eta_q:
B_+(0; \delta_1)
\ni \mathbf{x}_q := (\mathbf{x}_{q,1},\mathbf{x}_{q,2},\mathbf{x}_{q,3})
 \rightarrow
\mathcal{O}_q:= \eta_q(
B_+(0; \delta_1))
 \ \textit{ is one-to-one and onto for all $ q \in \partial \Omega$},
\Ee
and $\eta_q(\mathbf{x}_q)\in \partial \Omega$ if and only if $\mathbf{x}_{q,3}=0$ within the range of $\eta_q$.

Since the boundary is compact and $C^2$, for fixed $0<\delta_1 \ll 1$ we may choose a finite number of $p \in \mathcal{P} \subset\p\O$ and $0<\delta_2\ll 1$ such that $\mathcal{O}_p=\eta_p(
B_+(0; \delta_1)) \subset B(p;\delta_2) \cap \bar{\O}$ and $\{\mathcal{O}_p \}$ forms a finite covering of $\partial \Omega$. We define a partition of unity
\Be\label{iota}
\sum_{p \in \mathcal{P}} \iota_p(x) = \left\{
                                       \begin{array}{ll}
                                         1, & \hbox{for $x\in \partial \Omega$} \\
                                         0 & \hbox{for $x \notin \mathcal{O}_p$}
                                       \end{array}
                                     \right.
 \text{ such that }0 \leq \iota_p(x) \leq 1.
 \Ee
 Without loss of generality (see \cite{KL}) we can always reparametrize $\eta_p$ such that $\partial_{\mathbf{x}_{p,i}} \eta_p \neq 0$ for $i=1,2,3$ at $\mathbf{x}_{p,3}=0$, and an \textit{orthogonality} holds as
\Be\label{orthogonal}
\partial_{\mathbf{x}_{p,i}}\eta_p \cdot \partial_{\mathbf{x}_{p,j}}\eta_p =0 \ \ \text{at} \ \ \mathbf{x}_{p,3}=0 \text{  for  } i\neq j \text{ and } i,j\in \{1,2,3\}.
\Ee
For simplicity, we denote
\begin{equation}\label{partial_i eta}
\partial_i \eta_p(\mathbf{x}_p): = \partial_{\mathbf{x}_{p,i}} \eta_p.
\end{equation}

\begin{definition} For $x \in \bar{\O}$, we choose $p \in\mathcal{P}$ as in (\ref{O_p}). We define
\begin{align}
g_{p,ii}(\mathbf{x}_p) &  =\langle \partial_i \eta_p(\mathbf{x}_p),\partial_i \eta_p(\mathbf{x}_p)\rangle \ \ \text{for} \ \ i\in \{1,2,3\}, \\
    T_{\mathbf{x}_p}&
    =\left(
                               \begin{array}{ccc}
           \frac{\p_1 \eta_p(\mathbf{x}_p)}{\sqrt{g_{p,11}(\mathbf{x}_p) }}
           &      \frac{\p_2 \eta_p(\mathbf{x}_p)}{\sqrt{g_{p,22}(\mathbf{x}_p) }}
           &     \frac{\p_3 \eta_p(\mathbf{x}_p)}{\sqrt{g_{p,33}(\mathbf{x}_p) }}
            \\
                               \end{array}
                             \right)^t.\label{T}
\end{align}
Here $A^t$ stands the transpose of a matrix $A$. Note that when $\mathbf{x}_{p,3}=0$, $T_{\mathbf{x}_p}       \frac{\p_i \eta_p(\mathbf{x}_p)}{\sqrt{g_{p,ii}(\mathbf{x}_p) }}
  = e_i$ for $i=1,2,3$ where $\{e_i\}$ is a standard basis of $\R^3$.

We define
\Be\label{bar_v}
\mathbf{v}_j(\mathbf{x}_p) = \frac{\p_j \eta_p(\mathbf{x}_p)}{\sqrt{g_{p,jj}(\mathbf{x}_p) }} \cdot  v.
\Ee
\end{definition}

We note that from (\ref{orthogonal}), the map $T_{\mathbf{x}_p}$ is an orthonormal matrix when $\mathbf{x}_{p,3}=0$. Therefore both maps $v \rightarrow \mathbf{v} (\mathbf{x}_p )$ and $\mathbf{v} (\mathbf{x}_p ) \rightarrow v$ have a unit Jacobian.
Now we reparametrize the stochastic cycle using the local chart defined in Definition \ref{Def:Back time cycle}.

\begin{definition}\label{definition: chart}
Recall the stochastic cycles in Definition \ref{Def:Back time cycle}. For each cycle $x^k$ let us choose $p^k \in \mathcal{P}$ in (\ref{O_p}). Then we denote
\Be\begin{split}\label{xkvk}
\mathbf{x}^k_{p^k}&:= (\mathbf{x}^k_{p^k,1}, \mathbf{x}^k_{p^k,2},0)   \text{ such that }
\eta_{p^k} (\mathbf{x}^k_{p^k}) = x_k, \ \ \text{for} \  k =1, 2,
\\
\mathbf{v}^k_{p^k,i}&:= \frac{\p_i \eta_{p^k}(\mathbf{x}_{p^k}^k)}{\sqrt{g_{p^k,ii}(\mathbf{x}_{p^k}^k) }} \cdot  v_k  \ \ \text{for} \  k = 1,2.
\end{split}
\Ee

From chain rule we define
 \Be
\p_{\mathbf{x}^{k}_{p^{k},i}}[
a( \eta_{p^{k}} ( \mathbf{x}_{p^{k} }^{k}  ),    {v}_{k} ) ]
 :=
\frac{\p \eta_{p^{k}}(\mathbf{x}^{k}_{p^{k},i})}{\p \mathbf{x}^{k}_{p^{k},i}}
\cdot \nabla_x a ( \eta_{p^{k}} ( \mathbf{x}_{p^{k} }^{k}  ), v_{k}) , \ \ i=1,2.
\label{fBD_x1}
\Ee

\end{definition}

When we study the regularity we will need to take derivative to the stochastic cycle. We summarize the derivative in the following lemma.

\begin{lemma}\label{Lemma: nabla tbxb}(Lemma 1 in~\cite{CK})\\
For the $\tb$ and $\xb$ defined in~\eqref{BET}, the derivative reads
\Be
\begin{split}\label{nabla_tbxb}
\nabla_x \tb(x,v) = \frac{n(\xb)}{n(\xb) \cdot v},\ \
\nabla_v \tb(x,v) = - \frac{\tb n(\xb)}{n(\xb) \cdot v},\\
\nabla_x \xb(x,v) = Id_{3\times 3} - \frac{n(\xb) \otimes v}{n(\xb) \cdot v},\ \
\nabla_v \xb(x,v) = - \tb Id + \frac{ \tb n(\xb) \otimes v}{n(\xb) \cdot v}.
\end{split}\Ee

For $i=1,2$, $j=1,2$,

\begin{equation}\label{xip deri xbp}
\frac{\p \mathbf{x}^{2}_{p^{2},i}}{\p{\mathbf{x}^{1  }_{p^{1 },j}}} = \frac{1}{\sqrt{g_{p^{2}, ii} (\mathbf{x}^{2}_{p^{2} } )}}
\left[
\frac{\p_{i} \eta_{p^{2}} (\mathbf{x}^{2}_{p^{2} } ) }{\sqrt{g_{p^{2},ii}(\mathbf{x}^{2}_{p^{2} } ) }}
- \frac{\mathbf{v}^{2}_{p^{2}, i}}{\mathbf{v}^{2}_{p^{2}, 3}}
\frac{\p_{3} \eta_{p^{2}} (\mathbf{x}^{2}_{p^{2} } )  }{\sqrt{g_{p^{2},33}(\mathbf{x}^{2}_{p^{2} } )}}
\right] \cdot \p_j \eta_{p^{1}}(\mathbf{x}^{1}_{p^{1} } ) .
\end{equation}

\begin{equation}\label{xi deri xbp}
\frac{\partial \mathbf{x}_{p^{1},i}^{1}}{\partial [x]_j}=\frac{1}{\sqrt{g_{p^{1}, ii} (\mathbf{x}^{1}_{p^{1} } )}}
\left[
\frac{\p_{i} \eta_{p^{1}} (\mathbf{x}^{1}_{p^{1} } ) }{\sqrt{g_{p^{1},ii}(\mathbf{x}^{1}_{p^{1} } ) }}
- \frac{\mathbf{v}_{p^{1}, i}}{\mathbf{v}_{p^{1}, 3}}
\frac{\p_{3} \eta_{p^{1}} (\mathbf{x}^{1}_{p^{1} } )  }{\sqrt{g_{p^{1},33}(\mathbf{x}^{1}_{p^{1} } )}}
\right] \cdot e_j.
\end{equation}
Here $[x]_j$ is defined as the $j$-th coordinate of $x$ as specified in~\eqref{[x]}.

\end{lemma}

\begin{lemma}\label{Lemma: change of variable}(Lemma 3 in~\cite{CK})\\
The following map is one-to-one
\Be\label{map_v_to_xbtb}
v_{1} \in   \{ n(x_{1}) \cdot v_{1} >0: \xb(x_{1},v_{1}) \in B(p^{2}, \delta_2)\} \mapsto
(\mathbf{x}^{2}_{p^{2},1}, \mathbf{x}^{2}_{p^{2},2}, \tb^{1}),
\Ee
with
\Be\label{jac_v_to_xbtb}
\det\left(\frac{\p (\mathbf{x}^{2}_{p^{2},1}, \mathbf{x}^{2}_{p^{2},2}, \tb^{1})}{\p v_{1}}\right)= \frac{1}{\sqrt{ g_{p^{2},11}(\mathbf{x}^{2}_{p^{2}})  g_{p^{2},22}(\mathbf{x}_{p^{2}}^{2}) }}
\frac{|\tb^{1}|^3}{  |n(x_{2}) \cdot v_{1}| }.
\Ee

Here $\tb^1$ is the as the backward exit time starting from $(x_1,v_1)$:
\begin{equation}\label{tb1}
\tb^1 = \tb(x_1,v_1).
\end{equation}

\end{lemma}

\begin{lemma}\label{Lemma: nv<v2}(Lemma 4 in~\cite{CK})\\
Given a $C^2$ convex domain defined in~\eqref{convex},
\Be\label{nv<v2}
\begin{split}
|n_{p^{j}} (\mathbf{x}_{p^{j}}^{j}) \cdot  (x_{1} -
 \eta_{p^{2}} (\mathbf{x}_{p^{2}}^{2})
 )| \sim  |x_{1} -
 \eta_{p^{2}} (\mathbf{x}_{p^{2}}^{2})
 |^2,
 \ \  &j=1,2.
 \end{split}
\Ee
For $j'=1,2$,
\Be\label{bound_vb_x}
  \bigg| \frac{\p[ n_{p^{j}} (\mathbf{x}_{p^{j}}^{j}) \cdot  (x_{1} -
 \eta_{p^{2}} (\mathbf{x}_{p^{2}}^{2})
 )]}{\p {\mathbf{x}_{p^{2},j^\prime}^{2}}}  \bigg| \lesssim \| \eta \|_{C^2}
 |x_{1} -
 \eta_{p^{2}} (\mathbf{x}_{p^{2}}^{2})|
 ,
  \ \ j=1,2.
\Ee
\end{lemma}

\section{Weighted $C^1$-estimate of the dynamical solution.}
In this section we prove Theorem \ref{Thm: dynamic C1}. We will mainly prove the weighted $C^1$ estimate of the iteration equation~\eqref{eqn:formula of f^(m+1)} in Proposition \ref{proposition: boundedness}.

First we derive the boundary condition for $F=\sqrt{\mu}f$. By the boundary condition of $F$~\eqref{eqn:BC} and the reciprocity property~\eqref{eqn: reciprocity}, the boundary condition for $f$ becomes, for $(x,v)\in \gamma_-$,
\[f(t,x,v)|n(x)\cdot v|=\frac{1}{\sqrt{\mu}}\int_{n(x)\cdot u>0}   R(u\to v;x,t)  f(t,x,u)\sqrt{\mu(u)}\{n(x)\cdot u\}\dd u\]
\[=\frac{1}{\sqrt{\mu}}\int_{n(x)\cdot u>0}   R(-v\to -u;x,t) \frac{e^{-|v|^2/(2T_w(x))}}{e^{-|u|^2/(2T_w(x))}}  f(t,x,u)\sqrt{\mu(u)}\frac{|n(x)\cdot v|}{|n(x)\cdot u|}\{n(x)\cdot u\}\dd u.\]
Thus
\begin{equation}\label{eqn:C-L boundary condition in pro measure}
f(t,x,v)|_{\gamma_-}=e^{[\frac{1}{4T_M}-\frac{1}{2T_w(x)}]|v|^2}\int_{n(x)\cdot u>0} f(t,x,u)e^{-[\frac{1}{4T_M}-\frac{1}{2T_w(x)}]|u|^2}\dd \sigma(u,v).
\end{equation}
Here we denote
\begin{equation}\label{eqn:probability measure}
\dd\sigma(u,v):=R(-v\to -u;x,t)\dd u,
\end{equation}
which is a probability measure in the space $\{(x,u),n(x)\cdot u>0\}$ (well-defined due to~\eqref{eqn: normalization}).

We consider the following iteration equation:
\begin{equation}\label{eqn: Fm+1}
    \partial_t F^{m+1}+v\cdot \nabla_x F^{m+1}=Q_{\text{gain}}(F^m,F^m)-\nu(F^m)F^{m+1},\quad F^{m+1}|_{t=0}=F_0,
\end{equation}
with boundary condition
\[F^{m+1}(t,x,v)|n(x)\cdot v|=\int_{n(x)\cdot u>0}R(u\to v;x,t)F^{m}(t,x,u)\{n(x)\cdot u\}\dd u.\]
For $m\leq 0$ we set
\[F^m(t,x,v)=F_0(x,v).\]

We pose $F^{m+1}=\sqrt{\mu}f^{m+1}$, the equation for $f^{m+1}$ reads
\begin{equation}\label{eqn:formula of f^(m+1)}
\partial_t f^{m+1}+v\cdot \nabla_x f^{m+1}+\nu(F^m)f^{m+1} = \Gamma_{\text{gain}}\left(f^m,f^m\right).
\end{equation}

Taking the derivative $\partial =[\nabla_x ,\nabla_v]$ with the weight $e^{-\lambda\langle v\rangle t}\alpha $ we obtain
\begin{equation}\label{equation for partial fm+1}
\begin{split}
    & [\partial_t +v\cdot \nabla_x +\nu^m] e^{-\lambda\langle v\rangle t}\alpha(x,v)\partial f^{m+1}(t,x,v)=\mathcal{G}^m, \\
     &h(0,x,v) = \alpha(x,v) \partial f_0(x,v).
\end{split}
\end{equation}

In~\eqref{equation for partial fm+1} $\nu^m$ and $\mathcal{G}^m$ are defined as
\begin{align}
   &\nu^m = \nu(F^m)+\lambda\langle v\rangle-\alpha^{-1}[v\cdot \nabla_x \alpha]  \label{nu m}\\
   & \mathcal{G}^m(t,x) = e^{-\lambda\langle v\rangle t}\alpha(x,v) \big[-[\partial v]\cdot \nabla_x f^{m+1}-\partial [\nu \sqrt{\mu}f^m]f^{m+1}+\partial [\Gamma_{\text{gain}}(f^m,f^m)] \big].\label{Gm}
\end{align}

We choose $\lambda=\lambda(\xi)\gg 1$ and apply~\eqref{v nabla alpha bdd} to have
\begin{equation}\label{nu m geq}
\eqref{nu m}=\nu^m\geq \lambda \langle v\rangle -O_\xi(1)|v|\geq | v|.
\end{equation}

The boundary condition is given by the following lemma.

\begin{lemma}\label{Lemma: bc for partial f}(Lemma 12 in \cite{CKQ})\\
For $(x,v)\in \gamma_-$, we have the following bound for $e^{-\lambda\langle v\rangle t}\alpha(x,v) \partial f^{m+1}$ on the boundary:
\begin{equation}\label{eqn: derivative bound on the boundary first part}
|e^{-\lambda\langle v\rangle t}\alpha(x,v) \partial f^{m+1}(t,x,v)|\lesssim P(\Vert w_{\theta} f^m\Vert_\infty) + \langle v\rangle^2e^{[\frac{1}{4T_M}-\frac{1}{2T_w(x)}]|v|^2}\times~\eqref{eqn: derivative bound on the boundary second part}
\end{equation}
with
\begin{equation}\label{eqn: derivative bound on the boundary second part}
  \begin{split}
     & \int_{n(x)\cdot u>0} \langle u\rangle^2 |\partial f^m(t,x,u)|e^{-[\frac{1}{4T_M}-\frac{1}{2T_w(x)}]|u|^2}\dd\sigma(u,v)                                                               .
  \end{split}
\end{equation}

%
%
%
%
%
\end{lemma}

Then we establish the weighted $L^\infty$ bound of the sequence $\partial f^m$ in the following proposition.

\begin{proposition}\label{proposition: boundedness}
Assume $\partial f^{m+1}$ satisfies~\eqref{equation for partial fm+1} with the boundary condition~\eqref{eqn: derivative bound on the boundary first part} and the wall temperature satisfies~\eqref{eqn: r condition},~\eqref{eqn: Constrain on T}. Also assume the initial condition has bound
\begin{equation*}
  \Vert \alpha \partial f_0\Vert_{L^\infty}<\infty.
\end{equation*}
There exists $t_\infty\ll 1$ such that if
\begin{equation}\label{eqn: fm is bounded}
  \sup_{t\leq t_\infty}\sup_{i\leq m}\Vert e^{-\lambda\langle v\rangle t}\alpha\partial f^i(t,x,v)\Vert_{L^\infty}\leq C_\infty [\Vert \alpha \partial f_0\Vert_{L^\infty}+P(\sup_m\Vert w_\theta f^m\Vert_\infty)],
\end{equation}
then
\begin{equation}\label{eqn: L_infty bound for f^m+1}
\sup_{0\leq t\leq t_\infty}\Vert e^{-\lambda\langle v\rangle t}\alpha\partial f^{m+1}(t,x,v)\Vert_{L^\infty} \leq  C_\infty [\Vert \alpha \partial f_0\Vert_{L^\infty}+P(\sup_m\Vert w_\theta f^m\Vert_\infty)].
\end{equation}
Here $C_\infty$ is a constant defined in~\eqref{eqn: Cinfty} and
\begin{equation}\label{eqn: t_1}
t\leq t_{\infty}=t_{\infty}(T_M,T_m,r_\perp,r_\parallel,\Omega,\sup_m \Vert w_\theta f^m\Vert_\infty)\ll 1.
\end{equation}

\end{proposition}

\begin{remark}
The parameters in~\eqref{eqn: t_1} guarantee that the small time only depends on the temperature, accommodation and $L^\infty$ bound $\Vert w_\theta f^m\Vert $. The uniform-in-$m$ $L^\infty$ bound is concluded in~\cite{chen}:
\begin{equation}\label{uniform l infty}
 \sup_m \Vert w_\theta f^m\Vert_\infty <\infty.
\end{equation}

The Proposition \ref{proposition: boundedness} implies the uniform-in-$m$ $L^\infty$ estimate for $e^{-\lambda\langle v\rangle t}\alpha\partial f^{m}(t,x,v)$,
\begin{equation}\label{eqn: theta'}
\sup_m\Vert e^{-\lambda\langle v\rangle t}\alpha\partial f^{m}(t,x,v)\Vert_\infty<\infty.
\end{equation}

The strategy to prove Proposition \ref{proposition: boundedness} is to express $e^{-\lambda\langle v\rangle t}\alpha\partial f^{m+1}(t,x,v)$ along the characteristic using the C-L boundary condition. We present the characteristic formula in Lemma~\ref{lemma: the tracjectory formula for f^(m+1)}. We will use Lemma~\ref{lemma: boundedness} and Lemma~\ref{lemma: t^k} to bound the formula.
\end{remark}

We represent $e^{-\lambda\langle v\rangle t}\alpha\partial f^{m+1}(t,x,v)$ with the stochastic cycles defined as follows.

\begin{lemma}\label{lemma: the tracjectory formula for f^(m+1)}
Assume $e^{-\lambda\langle v\rangle t}\alpha\partial f^{m+1}(t,x,v)$ satisfies~\eqref{equation for partial fm+1} with the Cercignani-Lampis boundary condition~\eqref{eqn: derivative bound on the boundary first part}. If $t_1\leq 0$, then
\begin{equation}\label{eqn: Duhamal principle for case1}
|e^{-\lambda\langle v\rangle t}\alpha(x,v)\partial f^{m+1}(t,x,v)|\leq |\alpha(x,v)\partial f_0(X^1(0),v)| +\int_0^t |\mathcal{G}^m (s,X^1(s),v)|\dd s.
\end{equation}
If $t_1>0$, for $k\geq 2$, then
\begin{equation}\label{eqn: Duhamel principle for case 2}
\begin{split}
  |e^{-\lambda\langle v\rangle t}\alpha(x,v)\partial f^{m+1}(t,x,v)|  & \leq \int_{t_1}^t \mathcal{G}^m(s,X^{1}(s),v)\dd s + P(\Vert w_\theta f^m\Vert_\infty)\\
     & +\langle v\rangle^2 e^{[\frac{1}{4T_M}-\frac{1}{2T_w(x_1)}]|v|^2}\int_{\prod_{j=1}^{k-1}\mathcal{V}_j}H,
\end{split}
\end{equation}
where $H$ is bounded by
\begin{equation}\label{eqn: formula for H}
\begin{split}
   &  \sum_{l=1}^{k-1}\mathbf{1}_{\{t_l>0,t_{l+1}\leq 0\}}|\alpha \partial f_0\left(X^{l+1}(0),v_l\right)|\dd\Sigma_{l}^k\\
    & +\sum_{l=1}^{k-1}\int_{\max\{0,t_{l+1}\}}^{t_l} |\mathcal{G}^{m-l}(s,X^{l+1}(s)|\dd\Sigma_{l}^k \dd s\\
    &+\sum_{l=2}^{k-1} \mathbf{1}_{\{t_l>0\}}  P(\Vert w_{\theta}f^{m-l+1}\Vert_\infty)   \dd \Sigma_{l-1}^k     \\
    &+\mathbf{1}_{\{t_k>0\}}|e^{-\lambda\langle v_k\rangle t_k}\alpha \partial f^{m-k+2}\left(t_k,x_k,v_{k-1}\right)|\dd\Sigma_{k-1}^k,
\end{split}
\end{equation}
with
\begin{equation}\label{eqn:trajectory measure}
\begin{split}
 \dd\Sigma_{l,m}^k=&    \Big\{\prod_{j=l+1}^{k-1}\dd\sigma\left(v_j,v_{j-1}\right)\Big\}\\
 & \times \Big\{e^{\lambda\langle v_l\rangle t_l}\langle v_{l-1}\rangle^2 e^{-[\frac{1}{4T_M}-\frac{1}{2T_w(x_l)}]|v_l|^2}\frac{1}{n(x_l)\cdot v_l}
\dd\sigma(v_l,v_{l-1})\Big\}\\
    & \times\Big\{\prod_{j=1}^{l-1} e^{\lambda\langle v_j\rangle t_j}\langle v_j \rangle^4 e^{[\frac{1}{2T_w(x_j)}-\frac{1}{2T_w(x_{j+1})}]|v_j|^2}\frac{1}{n(x_j)\cdot v_j}  \dd\sigma\left(v_j,v_{j-1}\right)\Big\}.
\end{split}
\end{equation}
\end{lemma}

\begin{proof}

From~\eqref{eqn:formula of f^(m+1)}, for $0\leq s\leq t$, we apply the fundamental theorem of calculus to get
\[\frac{\dd}{\dd s}\int_s^t -\nu^m \dd\tau=\frac{\dd}{\dd s}\int_t^s \nu^md\tau=\nu^m.\]
Thus based on~\eqref{equation for partial fm+1},
\begin{equation}\label{eqn: inte factor}
\frac{\dd}{\dd s }\left[e^{-\int_s^t \nu^m \dd\tau}  e^{-\lambda\langle v\rangle s} \alpha(X^1(s),v)  \partial f^{m+1}(s,X^1(s),v)\right]=e^{-\int_s^t \nu^m d\tau}\mathcal{G}^m(s,X^1(s),v).
\end{equation}
By~\eqref{nu m geq},
\begin{equation}\label{eqn: v^2}
e^{-\int_s^t \nu^m d\tau} \leq e^{-|v|(t-s)}\leq 0.
\end{equation}
Combining~\eqref{eqn: inte factor} and~\eqref{eqn: v^2}, we derive that if $t_1\leq 0$, then we have~\eqref{eqn: Duhamal principle for case1}.

If $t_1(t,x,v)>0$, then
\begin{equation}\label{eqn: proof for the intial step in the Duhamul principle}
\begin{split}
 |e^{-\lambda\langle v\rangle t}\alpha(x,v)\partial f^{m+1}(t,x,v)\textbf{1}_{\{t_1>0\}}|  & \leq  |e^{-\lambda\langle v\rangle t_1}\alpha(x,v) \partial f^{m+1}\left(t_1,x_1,v\right)|e^{-|v|(t-t_1)}\\
    & +\int_{t_1}^t e^{-|v|(t-s)}\mathcal{G}^m(s,X^1(s),v)|\dd s.
\end{split}
\end{equation}
We use an induction of $k$ to prove~\eqref{eqn: Duhamel principle for case 2}. The first term of the RHS of~\eqref{eqn: proof for the intial step in the Duhamul principle} can be bounded by the boundary condition~\eqref{eqn: derivative bound on the boundary first part} as
\begin{equation}\label{eqn: above term}
\begin{split}
    &P(\Vert w_{\theta} f\Vert_\infty)+e^{-\lambda\langle v\rangle(t-t_1)} \langle v\rangle^2 e^{[\frac{1}{4T_M}-\frac{1}{2T_w(x)}]|v|^2}\int_{\mathcal{V}_1}e^{-\lambda\langle v_1\rangle t_1} \alpha(x_1,v_1) |\partial f^m(t_1,x_1,v_1)|   \\
     & \times e^{\lambda\langle v_1\rangle t_1}e^{-[\frac{1}{4T_M}-\frac{1}{2T_w(x_1)}]|v_1|^2} \frac{\langle v_1\rangle^2}{n(x_1)\cdot v_1} \dd\sigma(v_1,v),
\end{split}
\end{equation}
where we have used $n(x_1)\cdot v_1\lesssim \alpha(x_1,v_1)$.

Then we apply~\eqref{eqn: Duhamal principle for case1} and~\eqref{eqn: proof for the intial step in the Duhamul principle} to derive
\begin{equation}\label{from x1 to x2}
\begin{split}
 \eqref{eqn: above term}  & \leq \langle v\rangle^2 e^{-[\frac{1}{4T_M}-\frac{1}{2T_w(x)}]|v|^2}\\
 &\times\Big[\int_{\mathcal{V}_1}\mathbf{1}_{\{t_2\leq 0<t_1\}} \alpha(X^2(t^1),v_1)
 |\partial f^{m}(0,X^{2}(t_1),v_1)| \frac{\langle v_1\rangle^2e^{\lambda\langle v_1\rangle t_1} e^{-[\frac{1}{4T_M}-\frac{1}{2T_w(x_1)}]|v_1|^2}}{n(x_1)\cdot v_1}  \dd\sigma(v_1,v) \\
   & +\int_{\mathcal{V}_1}\int_0^{t_1}\mathbf{1}_{\{t_2\leq 0<t_1\}}e^{-|v_1|(t_1-s)}  |\mathcal{G}^{m-1}(s,X^2(s),v_1)|\frac{\langle v_1\rangle^2e^{\lambda\langle v_1\rangle t_1} e^{-[\frac{1}{4T_M}-\frac{1}{2T_w(x_1)}]|v_1|^2}}{n(x_1)\cdot v_1}\dd\sigma(v_1,v) \dd s\\
   & +\int_{\mathcal{V}_1}\mathbf{1}_{\{t_2>0\}}e^{-|v_1|(t_1-t_2)} e^{-\lambda \langle v_1\rangle t_2} \alpha(x_2,v_1)
| \partial f^{m}(t_2,x_2,v_1)|\frac{\langle v_1\rangle^2e^{\lambda\langle v_1\rangle t_1} e^{-[\frac{1}{4T_M}-\frac{1}{2T_w(x_1)}]|v_1|^2}}{n(x_1)\cdot v_1} \dd\sigma(v_1,v) \\
   &+\int_{\mathcal{V}_1}\int_{t_2}^{t_1}\mathbf{1}_{\{t_2> 0\}}e^{-|v_1|(t_1-s)}  |\mathcal{G}^{m-1}(s,X^2(s),v_1)|\frac{\langle v_1\rangle^2e^{\lambda\langle v_1\rangle t_1} e^{-[\frac{1}{4T_M}-\frac{1}{2T_w(x_1)}]|v_1|^2}}{n(x_1)\cdot v_1}\dd\sigma\left(v_1,v\right) \dd s \Big].
\end{split}
\end{equation}

Therefore, the formula~\eqref{eqn: Duhamel principle for case 2} is valid for $k=2$.

Assume~\eqref{eqn: Duhamel principle for case 2} is valid for $k\geq 2$ (induction hypothesis). Now we prove that~\eqref{eqn: Duhamel principle for case 2} holds for $k+1$. We express the last term in~\eqref{eqn: formula for H} using the boundary condition. Applying~\eqref{eqn: derivative bound on the boundary first part}\eqref{eqn: derivative bound on the boundary second part}, the contribution of constant term is
\begin{align*}
   &\int_{\prod_{j=1}^{k-1}\mathcal{V}_j} \mathbf{1}_{t_k>0} \Vert w_\theta f^{m-k+1}\Vert_\infty \dd\Sigma^k_{k-1}.
\end{align*}
Then the summation in the third line of~\eqref{eqn:trajectory measure} extends to $k$:
\[\sum_{l=2}^{k}\mathbf{1}_{t_l>0} P(\Vert w_{\theta}f^{m-l+1}\Vert_\infty) d \Sigma_{l-1}^k.\]
Since $\int_{\mathcal{V}_k}\dd\sigma(v_k,v_{k-1})=1 $ from~\eqref{eqn:probability measure}, we add $v_k$ integration to derive that for $l\leq k$
\begin{equation}\label{add vk}
\dd\sigma(v_k,v_{k-1}) \Sigma_{l-1}^k = \dd\Sigma_{l-1}^{k+1}.
\end{equation}
Thus the third line of~\eqref{eqn:trajectory measure} is valid for $k+1$.

For the other term in~\eqref{eqn: derivative bound on the boundary first part}\eqref{eqn: derivative bound on the boundary second part}, the front term $\langle v_{k-1}\rangle^2 e^{[\frac{1}{4T_M}-\frac{1}{2T_w(x_{k})}]|v_{k-1}|^2}$ depends on $v_{k-1}$, we move this term to the integration over $\mathcal{V}_{k-1}$ in~\eqref{eqn: Duhamel principle for case 2}. Using the second line of~\eqref{eqn:trajectory measure}, the integration over $\mathcal{V}_{k-1}$ is
\begin{equation}\label{eqn: wl}
\begin{split}
    & \int_{\mathcal{V}_{k-1}} e^{\lambda\langle v_{k-1}\rangle t_{k-1}} \frac{\langle v_{k-1}\rangle^4 }{n(x_{k-1})\cdot v_{k-1}}  e^{-[\frac{1}{4T_M}-\frac{1}{2T_w(x_{k-1})}]|v_{k-1}|^2}  e^{-[\frac{1}{4T_M}-\frac{1}{2T_w(x_{k})}]|v_{k-1}|}   \dd\sigma(v_{k-1},v_{k-2}) \\
     & =\int_{\mathcal{V}_{k-1}} e^{\lambda\langle v_{k-1}\rangle t_{k-1}} \frac{\langle v_{k-1}\rangle^4 }{n(x_{k-1})\cdot v_{k-1}}  e^{[\frac{1}{2T_w(x_{k-1})}-\frac{1}{2T_w(x_{k})}]|v_{k-1}|^2}   \dd\sigma(v_{k-1},v_{k-2}),
\end{split}
\end{equation}
which is consistent with third line in~\eqref{eqn:trajectory measure} with $l=k-1$.

For the remaining integration over $\mathcal{V}_k$ in~\eqref{eqn: derivative bound on the boundary first part}, we split it into two terms as
\begin{equation}\label{eqn: sh}
\int_{\mathcal{V}_k} \langle v_k\rangle^2 |\partial f^{m-k}(t^k,x^k,v^k)|e^{-[\frac{1}{4T_M}-\frac{1}{2T_w(x_k)}]|v_k|^2} \dd\sigma(v_k,v_{k-1})=\underbrace{\int_{\mathcal{V}_k}\mathbf{1}_{\{t_{k+1}\leq 0<t_k\}}}_{\eqref{eqn: sh}_1}+\underbrace{\int_{\mathcal{V}_k}\mathbf{1}_{\{t_{k+1}>0\}}}_{\eqref{eqn: sh}_2}.
\end{equation}
For the first term of the RHS of~\eqref{eqn: sh}, we use a similar bound as~\eqref{from x1 to x2} and derive that
\begin{equation}\label{eqn: wy}
\begin{split}
   &\eqref{eqn: sh}_1\leq  \int_{\mathcal{V}_k} \mathbf{1}_{\{t_{k+1}\leq 0<t_k\}} \alpha(X^{k+1}(t^k),v_k)
 \partial f^{m-k+1}(0,X^{k+1}(t_k),v_k) \frac{\langle v_k\rangle^2e^{\lambda\langle v_k\rangle t_k} e^{-[\frac{1}{4T_M}-\frac{1}{2T_w(x_k)}]|v_k|^2}}{n(x_k)\cdot v_{k}}  \dd\sigma(v_k,v_{k-1}) \\
    & +\int_{\mathcal{V}_k}\int_0^{t_k}\mathbf{1}_{\{t_{k+1}\leq 0<t_k\}}e^{-|v_k|(t_k-s)}  |\mathcal{G}^{m-k}(s,X^{k+1}(s),v_k)|\frac{\langle v_k\rangle^2e^{\lambda\langle v_k\rangle t_k} e^{-[\frac{1}{4T_M}-\frac{1}{2T_w(x_k)}]|v_k|^2}}{n(x_k)\cdot v_k}\dd\sigma(v_k,v_{k-1}) \dd s.
\end{split}
\end{equation}
In~\eqref{eqn: wy},
\[\frac{\langle v_k\rangle^2e^{\lambda\langle v_k\rangle t_k} e^{-[\frac{1}{4T_M}-\frac{1}{2T_w(x_k)}]|v_k|^2}}{n(x_k)\cdot v_{k}}\]
is consistent with the second line of~\eqref{eqn:trajectory measure} with $l=k$.

From the induction hypothesis(~\eqref{eqn: Duhamel principle for case 2} is valid for $k$), we derive the integration over $\mathcal{V}_j$ for $j\leq k-1$ is consistent with the third line of~\eqref{eqn:trajectory measure}.
After taking integration $\int_{\prod_{j=1}^{k-1} \mathcal{V}_j}$ we change $\dd\Sigma_{k-1}^k$ in~\eqref{eqn:trajectory measure} to $\dd\Sigma_{k}^{k+1}$. Thus the contribution of~\eqref{eqn: wy} is
\begin{equation}\label{eqn: qs}
  \begin{split}
     &\int_{\prod_{j=1}^{k} \mathcal{V}_j} \mathbf{1}_{\{t_{k+1}\leq 0<t_k\}}|\alpha \partial f_0\left(X^{k+1}(0),v_k\right)|\dd\Sigma_{k}^{k+1} \\
      & +\int_{\prod_{j=1}^{k} \mathcal{V}_j}\int_{0}^{t_k} \mathcal{G}^{m-k}(s)\dd\Sigma_{k}^{k+1}\dd s.
  \end{split}
\end{equation}

For the second term of the RHS of~\eqref{eqn: sh}, similar to~\eqref{from x1 to x2} we derive
\begin{equation}\label{eqn: wyl}
  \begin{split}
     &  \eqref{eqn: sh}_2 \notag\\
     &\leq \int_{\mathcal{V}_k}\mathbf{1}_{\{t_{k+1}>0\}}e^{-\lambda \langle v_k\rangle t_k}e^{-\lambda\langle v_{k}\rangle t_k}\alpha(x_{k+1},v_k)
 \partial f^{m-k+1}(t_{k+1},x_{k+1},v_k) \frac{\langle v_k\rangle^2e^{\lambda\langle v_k\rangle t_k} e^{-[\frac{1}{4T_M}-\frac{1}{2T_w(x_k)}]|v_k|^2}}{n(x_k)\cdot v_{k}}  \dd\sigma(v_k,v_{k-1})\\
      &+ \int_{\mathcal{V}_k}\int_{t_{k+1}}^{t_k}\mathbf{1}_{\{t_{k+1}>0\}}e^{-|v_k|(t_k-s)}  |\mathcal{G}^{m-k}(s,X^{k+1}(s),v_k)|\frac{\langle v_k\rangle^2e^{\lambda\langle v_k\rangle t_k} e^{-[\frac{1}{4T_M}-\frac{1}{2T_w(x_k)}]|v_k|^2}}{n(x_k)\cdot v_k}\dd\sigma(v_k,v_{k-1}) \dd s.
  \end{split}
\end{equation}
Similar to~\eqref{eqn: qs}, after taking integration over $\int_{\prod_{j=1}^{k-1}\mathcal{V}_j}$ the contribution of~\eqref{eqn: wyl} is
\begin{equation}\label{eqn: nxh}
\begin{split}
   & \int_{\prod_{j=1}^{k} \mathcal{V}_j} \mathbf{1}_{\{t_{k+1}>0\}}e^{-\lambda\langle v_{k}\rangle t_k}\alpha(x_{k+1},v_k)
 \partial f^{m-k+1}(t_{k+1},x_{k+1},v_k)\dd\Sigma_{k}^{k+1}\\
    & +\int_{\prod_{j=1}^{k} \mathcal{V}_j}\int_{t_{k+1}}^{t_k} \mathcal{G}^{m-k}(s)\dd\Sigma_{k}^{k+1} \dd s.
\end{split}
\end{equation}

From~\eqref{eqn: nxh}~\eqref{eqn: qs}, the summation in the first and second lines of~\eqref{eqn: formula for H} extends to $k$. And the index of the fourth line of~\eqref{eqn: formula for H} changes from $k$ to $k+1$.
For the rest terms, the index $l\leq k-1$. We add the $v_k$ integration as~\eqref{add vk} so that the integration change to $\prod_{l=1}^{k}\mathcal{V}_j$.

Therefore, the formula~\eqref{eqn: formula for H} is valid for $k+1$ and we derive the lemma.
\end{proof}

The next lemma is the key to prove the $L^\infty$ bound for $h^{m+1}$. Below we define several notation: let
\begin{equation}\label{eqn: def of r}
r_{max}:=\max(r_\parallel(2-r_\parallel),r_\perp),~~~~~~~~~~~~ r_{min}:=\min(r_\parallel(2-r_\parallel),r_\perp).
\end{equation}
Then we have
\begin{equation}\label{eqn: r>0}
1\geq r_{max}\geq r_{min}>0.
\end{equation}

We inductively define:
\begin{equation}\label{eqn: definition of T_p}
T_{l,l}=2T_M,    \quad T_{l,l-1}=r_{min}T_M +(1-r_{min})T_{l,l}, \cdots,\quad  T_{l,1}= r_{min}T_M +(1-r_{min})T_{l,2}.
\end{equation}
By a direct computation, for $1\leq i\leq l$, we have
\begin{equation}\label{eqn: formula of Tp}
T_{l,i}=2T_M+(T_M-2T_M)[1-(1-r_{min})^{l-i}].
\end{equation}
Moreover, we denote
\begin{equation}\label{eqn: big phi}
\begin{split}
 \dd \Phi_{p,m}^{k,l}(s):=  & \{\prod_{j=l+1}^{k-1}\dd\sigma(v_j,v_{j-1})\}\\
 &\times \{\frac{e^{\lambda\langle v_l\rangle t_l}\langle v_{l-1}\rangle^2 e^{-[\frac{1}{4T_M}-\frac{1}{2T_w(x_l)}]|v_l|^2}}{n(x_l)\cdot v_l}
\dd\sigma(v_l,v_{l-1})\}  \\ \
    & \times \{\prod_{j=p}^{l-1}   e^{[\frac{1}{2T_w(x_j)}-\frac{1}{2T_w(x_{j+1})}]|v_j|^2}  \frac{e^{\lambda\langle v_j\rangle t_j}\langle v_j\rangle^4}{n(x_j)\cdot v_j} \dd\sigma(v_j,v_{j-1})\}.
\end{split}
\end{equation}
Note that if $p=1$, $\dd \Phi_{1,m}^{k,l}=\dd\Sigma_{l}^{k}$ where $\dd\Sigma_{l}^{k}$ is defined in~\eqref{eqn:trajectory measure}. And we denote
\begin{equation}\label{eqn: upsilon}
\dd\Upsilon_{p}^{p'}:=\prod_{j=p}^{p'}   e^{[\frac{1}{2T_w(x_j)}-\frac{1}{2T_w(x_{j+1})}]|v_j|^2} \frac{e^{\lambda\langle v_j\rangle t_j}\langle v_j\rangle^4}{n(x_j)\cdot v_j} \dd\sigma(v_j,v_{j-1}).
\end{equation}
Then by the definition of~\eqref{eqn: big phi} and~\eqref{eqn:trajectory measure}, we have
\begin{equation}\label{eqn: ppt for Phi}
\dd \Phi_{p,m}^{k,l}=\dd\Phi_{p',m}^{k,l}\dd\Upsilon_{p}^{p'-1},
\end{equation}
\begin{equation}\label{eqn: property for Gamma}
\dd\Sigma_{l}^k=\dd\Phi_{p,m}^{k,l}\dd\Upsilon_1^{p-1}.
\end{equation}
\begin{remark}
In Lemma \ref{lemma: the tracjectory formula for f^(m+1)} the integration has multiple fold and each fold contains the variable $T_w(x),T_M,r_\perp,r_\parallel$.
We define these inductive notations to find a pattern to bound such integration.

\end{remark}

Now we state the lemma.

\begin{lemma}\label{lemma: boundedness}
Given the formula for $e^{-\lambda\langle v\rangle t}\alpha f^{m+1}(t)$ in~\eqref{eqn: Duhamal principle for case1} and~\eqref{eqn: Duhamel principle for case 2} in lemma~\ref{lemma: the tracjectory formula for f^(m+1)}, there exists
\begin{equation}\label{eqn: t*}
t_*=t_*(T_M,T_m,k,r_\parallel,r_\perp,c),
\end{equation}
($t_*$ need to satisfy more conditions specified in Lemma \ref{lemma: t^k} and~\eqref{eqn: less than 0}) such that: when $t\leq t_*$, we have
\begin{equation}\label{eqn: boundedness for l-p+1 fold integration}
   \int_{\prod_{j=p}^{k-1}\mathcal{V}_j}     \mathbf{1}_{\{t_l>0\}}  \dd\Phi_{p,m}^{k,l} \leq t_*^{-(l-p+1)c} (C_{T_M,T_m})^{l-p+1}\mathcal{A}_{l,p}.
\end{equation}
Here we define:
\begin{equation}\label{eqn: Elp}
\mathcal{A}_{l,p}:=\exp\left(\big[ \frac{[T_{l,p}-T_w(x_{p})][1-r_{min}]}{2T_w(x_{p})[T_{l,p}(1-r_{min})+r_{min} T_w(x_{p})]} + \mathcal{C}_{l-p+1}t_*^{c}\big]|v_{p-1}|^2\right).
\end{equation}
$C_{T_M,T_m}$ is a constant defined in~\eqref{Constant term} and
\begin{equation}\label{Extra term in expo}
  \mathcal{C}_{n} := \sum_{i=1}^{n} \mathcal{C}^i = \mathcal{C} \frac{\mathcal{C}^n-1}{\mathcal{C}-1},
\end{equation}
where $\mathcal{C}$ is constant defined in~\eqref{eqn: cal C}. And $c<1$ is a constant. We will
specify $c=\frac{1}{15}$ later in~\eqref{c}.

Moreover, for any $p<p'\leq l$, we have
\begin{equation}\label{eqn: structure}
\begin{split}
 \int_{\prod_{j=p}^{k-1}\mathcal{V}_j} \mathbf{1}_{\{t_l>0\}}  \dd\Phi_{p,m}^{k,l}   & \leq t_*^{-(l-p'+1)q}(C_{T_M,T_m})^{2(l-p'+1)} \int_{\prod_{j=p}^{p'-1} \mathcal{V}_j}  \mathbf{1}_{\{t_l>0\}} \mathcal{A}_{l,p'} \dd\Upsilon_p^{p'-1}\\
     & \leq t_*^{-(l-p+1)q}(C_{T_M,T_m})^{2(l-p+1)}\mathcal{A}_{l,p}.
\end{split}
\end{equation}

\end{lemma}
\begin{remark}
To prove Lemma \ref{lemma: boundedness} we do not need the condition~\eqref{eqn: less than 0}. Such condition will be used in the proof of Proposition \ref{proposition: boundedness}.
\end{remark}

\begin{proof}
From~\eqref{eqn: normalization} and~\eqref{eqn:probability measure}, for the first bracket of the first line in~\eqref{eqn:trajectory measure} with $l+1\leq j\leq k-1$, we have
\[\int_{\prod_{j=l+1}^{k-1} \mathcal{V}_j}    \prod_{j=l+1}^{k-1}\dd\sigma(v_j,v_{j-1})=1.\]
Without loss of generality we can assume $k=l+1$. Thus $\dd\Phi_{p,m}^{k,l}=\dd\Phi_{p,m}^{l+1,l}$. We use an induction of $p$ with $1\leq p\leq l$ to prove~\eqref{eqn: boundedness for l-p+1 fold integration}.

When $p=l$, by the second line of~\eqref{eqn: big phi}, the integration over $\mathcal{V}_l$ is bounded by
\begin{equation}\label{eqn: dsigmal}
\int_{\mathcal{V}_l} e^{-[\frac{1}{4T_M}-\frac{1}{2T_w(x_l)}]|v_l|^2}\frac{e^{\lambda\langle v_l\rangle t_l}\langle v_l\rangle^4} {n(x_l)\cdot v_l}\dd\sigma(v_l,v_{l-1}).
\end{equation}

Clearly $e^{\lambda\langle v_l\rangle t_l}\leq e^{\lambda\langle v_l\rangle t_*}.$ We expand $\dd\sigma(v_l,v_{l-1})$ with~\eqref{eqn: Formula for R} and~\eqref{eqn:probability measure}, then we apply~\eqref{bound extra term} in Lemma \ref{Lemma: extra term} to bound~\eqref{eqn: dsigmal} by
\begin{equation}\label{eqn: int over V_l}
\begin{split}
   &t_*^{-c}\int_{\mathcal{V}_{l,\perp}} \frac{2}{r_\perp}\frac{1}{\langle v_{l,\perp}\rangle^2}{2T_w(x_l)}  e^{-[\frac{1}{4T_M}-\frac{1}{2T_w(x_l)}-t_*^c]|v_{l,\perp}|^2}I_0\left(\frac{(1-r_\perp)^{1/2}v_{l,\perp}v_{l-1,\perp}}{T_w(x_l)r_\perp}\right)e^{-\frac{|v_{l,\perp}|^2+(1-r_\perp)|v_{l-1,\perp}|^2}{2T_w(x_l)r_\perp}}  \dd v_{l,\perp} \\
    &\times \int_{\mathcal{V}_{l,\parallel}}\frac{1}{\pi r_\parallel(2-r_\parallel)(2T_w(x_l))}e^{-[\frac{1}{4T_M}-\frac{1}{2T_w(x_l)}-t_*^c]|v_{l,\parallel}|^2}e^{-\frac{1}{2T_w(x_l)}\frac{|v_{l,\parallel}-(1-r_\parallel)v_{l-1,\parallel}|^2}{r_\parallel(2-r_\parallel)}}\dd v_{l,\parallel},
\end{split}
\end{equation}
where $v_{l,\parallel}$, $v_{l,\perp}$, $\mathcal{V}_{l,\perp}$ and $\mathcal{V}_{l,\parallel}$ are defined as
\begin{equation}\label{eqn: Define space}
v_{l,\perp}=v_l\cdot n(x_l),~~ v_{l,\parallel}=v_l-v_{l,\perp}n(x_l),~~\mathcal{V}_{l,\perp}=\{v_{l,\perp}:v_l\in \mathcal{V}_l\},~~\mathcal{V}_{l,\parallel}=\{v_{l,\parallel}:v_l\in \mathcal{V}_l\}.
\end{equation}
$v_{l-1,\parallel}$ and $v_{l-1,\perp}$ are defined similarly.

First we compute the integration over $\mathcal{V}_{l,\parallel}$, the second line of~\eqref{eqn: int over V_l}. To apply~\eqref{eqn: coe abc} in Lemma \ref{Lemma: abc}, we set
\[\e=t_*^c,\quad w=(1-r_\parallel)v_{l-1,\parallel},\quad v=v_{l,\parallel},\]
\begin{equation}\label{eqn: coefficient a and b}
a=-[\frac{1}{4T_M}-\frac{1}{2T_w(x_l)}],\quad b=\frac{1}{2T_w(x_l)r_\parallel(2-r_\parallel)}.
\end{equation}
We take $t_*=t_*(T_M,c)\ll 1$ such that when $t< t_*$, we have
\begin{equation}\label{eqn: b-a-e}
b-a-t_*^c=\frac{1}{2T_w(x_l)r_\parallel(2-r_\parallel)}-\frac{1}{2T_w(x_l)}+\frac{1}{4T_M}-t_*^c
\geq \frac{1}{4T_M}-t_*^c\geq \frac{1}{8T_M}.
\end{equation}
Then we further require $t\leq t_*(T_M,c)\ll 1$ such that $1+8T_M t_*^c<2$, then we have
\begin{align}
  \frac{b}{b-a-t_*^c} &=\frac{b}{b-a}[1+\frac{t_*^c}{b-a-t_*^c}]\leq \frac{2T_M}{2T_M+[T_w(x_l)-2T_M]r_\parallel(2-r_\parallel)}[1+8T_M t_*^c]\notag  \\
   & \leq \frac{4T_M}{2T_M+[T_m-2T_M]r_{max}}:=C_{T_M},\label{eqn: 1 one}
\end{align}
where we used~\eqref{eqn: def of r}.

In regard to~\eqref{eqn: coe abc}, we have
\begin{equation}\label{eqn: com}
\frac{(a+t_*^c)b}{b-a-t_*^c}=\frac{ab}{b-a}[1+\frac{t_*^c}{b-a-t_*^c}]+\frac{b}{b-a-t_*^c}t_*^c.
\end{equation}
By~\eqref{eqn: 1 one} we obtain
\[\frac{b}{b-a-t_*^c}t_*^c\leq \frac{4T_M}{2T_M+[T_m-2T_M]r_{\max}}t_*^c.\]
By~\eqref{eqn: coefficient a and b}, we have
\[\frac{ab}{b-a}=\frac{2T_M-T_w(x_l)}{2T_w(x_l)[2T_M+[T_w(x_l)-2T_M]r_\parallel(2-r_\parallel)]}.\]
Therefore, by~\eqref{eqn: b-a-e} and~\eqref{eqn: com} we obtain
\begin{equation}\label{eqn: 2 two}
\begin{split}
  &  \frac{(a+t_*^c)b}{b-a-t_*^c}\leq \frac{2T_M-T_w(x_l)}{2T_w(x_l)[2T_M+[T_w(x_l)-2T_M]r_\parallel(2-r_\parallel)]}+\mathcal{C} t_*^c,
\end{split}
\end{equation}
where we defined
\begin{equation}\label{eqn: cal C}
\mathcal{C}:=\frac{4T_M\big(2T_M-T_m\big)}{2T_m[2T_M+[T_m-2T_M]r_{max}]}+\frac{4T_M}{2T_M+[T_m-2T_M]r_{max}}.
\end{equation}

By~\eqref{eqn: 1 one},~\eqref{eqn: 2 two} and Lemma~\ref{Lemma: abc}, using $w=(1-r_\parallel)v_{l-1,\parallel}$ we bound the second line of~\eqref{eqn: int over V_l} by
\begin{equation}\label{eqn: result for para}
C_{T_M}\exp\bigg(\Big[\frac{[2T_M-T_w(x_l)]}{2T_w(x_l)[2T_M(1-r_\parallel)^2+r_\parallel(2-r_\parallel) T_w(x_l)]}+\mathcal{C}t_*^c \Big]|(1-r_\parallel)v_{l-1,\parallel}|^2\bigg)
\end{equation}
\begin{equation}\label{eqn: result of dsigmal}
\leq C_{T_M}\exp\bigg(\Big[\frac{[2T_M-T_w(x_l)][1-r_{min}]}{2T_w(x_l)\big[2T_M(1-r_{min})+r_{min} T_w(x_l)\big]}+\mathcal{C}t_*^c\Big]|v_{l-1,\parallel}|^2\bigg),
\end{equation}
where we used~\eqref{eqn: def of r} and~\eqref{eqn: r>0}.

Next we compute the first line of~\eqref{eqn: int over V_l}. To apply~\eqref{eqn: coe abc perp} in Lemma \ref{Lemma: perp abc}, we set
\[\e = t_*^c,\quad w=\sqrt{1-r_\parallel}v_{l-1,\perp},\quad v=v_{l,\perp},\]
\[a=-[\frac{1}{4T_M}-\frac{1}{2T_w(x_l)}],~ b=\frac{1}{2T_w(x_l)r_\perp}.\]
$\frac{(a+\e)b}{b-a-\e}$ can be computed using the same way as~\eqref{eqn: 2 two} with replacing $r_\parallel(2-r_\parallel)$ by $r_\perp$. Here the difference is the constant term becomes
\begin{align}
 \frac{2b}{\sqrt{b-a-t_*^c}}  &=2\sqrt{b} \sqrt{\frac{b}{b-a-t_*^c}}\leq 2\sqrt{b} C_{T_M}  \notag\\
   & \leq \frac{2}{\sqrt{T_m}}\times \sqrt{C_{T_M}}.
\end{align}

Hence replacing $r_\parallel(2-r_\parallel)$ by $r_\perp$ and replacing $v_{l-1,\parallel}$ by $v_{l-1,\perp}$ in~\eqref{eqn: result for para}, we bound the first line of~\eqref{eqn: int over V_l} by
\[\frac{2}{\sqrt{T_m}}\sqrt{C_{T_M}}\exp\bigg(\Big[\frac{[2T_M-T_w(x_l)]}{2T_w(x_l)[2T_M(1-r_\perp)+r_\perp T_w(x_l)]}+\mathcal{C}t_*^c \Big]|\sqrt{1-r_\perp}v_{l-1,\perp}|^2\bigg)\]
\begin{equation}\label{eqn: result of dsigmal normal}
\leq \frac{2}{\sqrt{T_m}} C_{T_M}\exp\bigg(\Big[\frac{[2T_M-T_w(x_l)][1-r_{min}]}{2T_w(x_l)\big[2T_M(1-r_{min})+r_{min} T_w(x_l)\big]}+\mathcal{C}t_*^c\Big]|v_{l-1,\perp}|^2\bigg),
\end{equation}
where we used~\eqref{eqn: def of r} and~\eqref{eqn: r>0}.

Then we define the constant term in~\eqref{eqn: Elp} as
\begin{equation}\label{Constant term}
C_{T_M,T_m} = \frac{2}{\sqrt{T_m}} C_{T_M}^{3/2}.
\end{equation}

Collecting~\eqref{eqn: result of dsigmal}~\eqref{eqn: result of dsigmal normal}, we derive
\[\eqref{eqn: int over V_l}\leq t_*^{-c} C_{T_M,T_m}\exp\left(\left[\frac{[2T_M-T_w(x_l)][1-r_{min}]}{2T_w(x_l)\big[2T_M(1-r_{min})+r_{min} T_w(x_l)\big]}+\mathcal{C}t_*^c\right]|v_{l-1}|^2\right)=C_{T_M,T_m}\mathcal{A}_{l,l},\]
where $\mathcal{A}_{l,l}$ is defined in~\eqref{eqn: Elp} and $T_{l,l}=2T_M$.

Therefore,~\eqref{eqn: boundedness for l-p+1 fold integration} is valid for $p=l$ by $\mathcal{C}_1 = \mathcal{C}$.

Suppose~\eqref{eqn: boundedness for l-p+1 fold integration} is valid for $p=q+1$(induction hypothesis) with $q+1\leq l$, then
\[\int_{\prod_{j=q+1}^{l}\mathcal{V}_j}        \mathbf{1}_{\{t_l>0\}} \dd\Phi_{q+1,m}^{l+1,l}\leq t_*^{-(l-q)c} C_{T_M,T_m}^{l-q}\mathcal{A}_{l,q+1}.\]
We want to show~\eqref{eqn: boundedness for l-p+1 fold integration} holds for $p=q$. By the hypothesis and the third line of~\eqref{eqn: big phi},
\begin{equation}\label{eqn: dsigmaq}
  \int_{\prod_{j=q}^{l}\mathcal{V}_j}        \mathbf{1}_{\{t_l>0\}} \dd\Phi_{q,m}^{l+1,l} \leq  t_*^{-(l-q+1)c} C_{T_M,T_m}^{l-q}\int_{\mathcal{V}_q} \mathcal{A}_{l,q+1} e^{[\frac{1}{2T_w(x_{q})}-\frac{1}{2T_w(x_{q+1})}+t_*^c]|v_{q}|^2}\frac{1}{n(x_q)\cdot v_q}  \dd\sigma(v_q,v_{q-1}),
\end{equation}
where we have applied Lemma \ref{Lemma: extra term}.

Using the definition of $\mathcal{A}_{l,q+1}$ in~\eqref{eqn: Elp}, we obtain
\begin{equation}\label{eqn: dsigmap}
\begin{split}
   &\eqref{eqn: dsigmaq} \leq  t_*^{-(l-q+1)c} C_{T_M,T_m}^{l-q}  \int_{\mathcal{V}_{q}}    \exp\bigg(\frac{(T_{l,q+1}-T_w(x_{q+1}))(1-r_{min})}{2T_w(x_{q+1})[T_{l,q+1}(1-r_{min})+r_{min} T_w(x_{q+1})]}|v_q|^2+\mathcal{C}_{l-q}t_*^c|v_q|^2\bigg)                                     \\
    &   e^{[\frac{1}{2T_w(x_{q})}-\frac{1}{2T_w(x_{q+1})}+t_*^c]|v_{q}|^2}  \dd\sigma(v_q,v_{q-1}).
\end{split}
\end{equation}
We focus on the coefficient of $|v_q|^2$ in~\eqref{eqn: dsigmap}, we derive
\begin{align*}
   &\frac{(T_{l,q+1}-T_w(x_{q+1}))(1-r_{min})}{2T_w(x_{q+1})[T_{l,q+1}(1-r_{min})+r_{min} T_w(x_{q+1})]}|v_q|^2+[\frac{1}{2T_w(x_{q})}-\frac{1}{2T_w(x_{q+1})}]|v_q|^2  \\
   & =       \frac{(T_{l,q+1}-T_w(x_{q+1}))(1-r_{min})-[T_{l,q+1}(1-r_{min})+r_{min} T_w(x_{q+1})]}{2T_w(x_{q+1})[T_{l,q+1}(1-r_{min})+r_{min} T_w(x_{q+1})]}|v_q|^2                             + \frac{|v_q|^2}{2T_w(x_q)} \\
   & =    \frac{ -T_w(x_{q+1})(1-r_{min})-r_{min} T_w(x_{q+1})}{2T_w(x_{q+1})[T_{l,q+1}(1-r_{min})+r_{min} T_w(x_{q+1})]}|v_q|^2                             + \frac{|v_q|^2}{2T_w(x_q)} \\
   &=    \frac{-|v_q|^2}{2[T_{l,q+1}(1-r_{min})+r_{min}T_w(x_{q+1})]}                + \frac{|v_q|^2}{2T_w(x_q)}.
\end{align*}

By the Definition~\ref{Def:Back time cycle}, $x_{q+1}=x_{q+1}(t,x,v,v_1,\cdots,v_{q})$, thus $T_w(x_{q+1})$ depends on $v_{q}$. In order to explicitly compute the integration over $\mathcal{V}_q$, we need to get rid of the dependence of the $T_w(x_{q+1})$ on $v_{q}$. Then we bound
\begin{equation}\label{eqn: Help to integrate x_p over v_p-1}
\begin{split}
   & \exp\left(\frac{-|v_{q}|^2}{2[T_{l,q+1}(1-r_{min})+r_{min} T_w(x_{q+1})]}\right)\leq \exp\left(\frac{-|v_{q}|^2}{2[T_{l,q+1}(1-r_{min})+r_{min}T_M]}\right)= \exp\left(\frac{-|v_{q}|^2}{2T_{l,q}}\right),
\end{split}
\end{equation}
where we used~\eqref{eqn: definition of T_p}.

Hence by~\eqref{eqn:probability measure}~\eqref{eqn: Formula for R} and~\eqref{eqn: Help to integrate x_p over v_p-1}, we derive
\begin{equation}\label{eqn: int over V_p}
\begin{split}
  &\eqref{eqn: dsigmap}\leq t_*^{-(l-q+1)c}C_{T_M,T_m}^{l-q} \\
   &  \times \int_{\mathcal{V}_{q,\perp}} \frac{1}{r_\perp T_w(x_q)}  e^{-[\frac{1}{2T_{l,q}}-\frac{1}{2T_w(x_q)}-\mathcal{C}_{l-q}t_*^c-t_*^c]|v_{q,\perp}|^2}I_0\left(\frac{(1-r_\perp)^{1/2}v_{q,\perp}v_{q-1,\perp}}{T_w(x_q)r_\perp}\right)e^{-\frac{|v_{q,\perp}|^2+(1-r_\perp)|v_{q-1,\perp}|^2}{2T_w(x_q)r_\perp}}  \dd v_{q,\perp} \\
    & \times\int_{\mathcal{V}_{q,\parallel}}\frac{1}{\pi r_\parallel(2-r_\parallel)(2T_w(x_q))}e^{-[\frac{1}{2T_{l,q}}-\frac{1}{2T_w(x_q)}-\mathcal{C}_{l-q}t_*^c-t_*^c]|v_{q,\parallel}|^2}e^{-\frac{1}{2T_w(x_q)}\frac{|v_{q,\parallel}-(1-r_\parallel)v_{q-1,\parallel}|^2}{r_\parallel(2-r_\parallel)}}\dd v_{q,\parallel}.
\end{split}
\end{equation}
In the third line of~\eqref{eqn: int over V_p}, to apply~\eqref{eqn: coe abc} in Lemma \ref{Lemma: abc}, we set
\[a=-[\frac{1}{2T_{l,q}}-\frac{1}{2T_w(x_{q})}],~ b=\frac{1}{2T_w(x_{q})r_\parallel(2-r_\parallel)},~\e=\mathcal{C}_{l-q}t_*^c+t_*^c,~w=(1-r_\parallel)v_{q-1,\parallel}.\]
Taking~\eqref{eqn: coefficient a and b} for comparison, we can replace $2T_M$ by $T_{l,q}$ and replace $t_*^c$ by $\mathcal{C}_{l-q}t_*^c+t_*^c$. Then we apply the replacement to~\eqref{eqn: b-a-e} and obtain
\[b-a-\e\geq \frac{1}{2T_{l,q}}-\mathcal{C}_{l-q}t_*^c-t_*^c\geq \frac{1}{4T_M}-\mathcal{C}_k t_*^c-t_*^c= \frac{1}{4T_M}-\mathcal{C}\frac{\mathcal{C}^k-1}{\mathcal{C}-1}t_*^c-t_*^c> \frac{1}{5T_M},\]
where we applied~\eqref{Extra term in expo} and we take $t_*=t_*(T_M,\mathcal{C},k,c)$ to be small enough with $t\leq t^*$. Also we require the $t<t_*(T_M,\mathcal{C},k,c)$ satisfy
\[\frac{\e}{b-a-\e}\leq 5T_M(1+\mathcal{C}_k)t^c\leq 2.\]
By the definition of $\mathcal{C}$ in~\eqref{eqn: cal C} we conclude the $t_*$ only depends on the parameter in~\eqref{eqn: t*}. Thus by the same computation as~\eqref{eqn: 1 one} we obtain
\[\frac{b}{b-a-\e}\leq \frac{2T_{l,q}}{T_{l,q}+[\min\{T_w(x)\}-T_{l,q}]r_\parallel(2-r_\parallel)}\leq C_{T_M},\]
where we used $T_{l,q}\leq 2T_M$ from~\eqref{eqn: definition of T_p} and \eqref{eqn: def of r}. $C_{T_M}$ is defined in~\eqref{eqn: 1 one}.

By the same computation as~\eqref{eqn: 2 two}, we obtain
\[\frac{(a+\e)b}{b-a-\e}= \frac{ab}{b-a}+\frac{ab}{b-a}\frac{\e}{b-a-\e}+\frac{b}{b-a-\e}\e\]
\[\leq  \frac{T_{l,q}-T_w(x_q)}{2T_w(x_q)[T_{l,q}+[T_w(x_q)-T_{l,q}]r_\parallel(2-r_\parallel)]}+\mathcal{C}_{l-q+1}t_*^c.\]
Here we used $T_{l,q}\leq 2T_M$ and \eqref{eqn: def of r} to obtain
\begin{align*}
 & \frac{ab}{b-a}\frac{\e}{b-a-\e}+ \frac{b\e}{b-a-\e}  \\
& \leq \frac{4T_M\big(T_{l,q}-\min\{T_w(x)\}\big)}{2\min\{T_w(x)\}[T_{l,q}+[\min\{T_w(x)\}-T_{l,q}]r_\parallel(2-r_\parallel)]}[1+\mathcal{C}_{l-q}]t_*^c  \\
   & +\frac{2T_{l,q}}{2+[\min\{T_w(x)\}-T_{l,q}]r_\parallel(2-r_\parallel)}[1+\mathcal{C}_{l-q}]t_*^c   \leq [\mathcal{C}+\mathcal{C}\mathcal{C}_{l-q}]t_*^c = \mathcal{C}_{l-q+1} t_*^c
\end{align*}
with $\mathcal{C}$ defined in~\eqref{eqn: cal C} and $\mathcal{C}_{l-q}$ defined in~\eqref{Extra term in expo}.

Thus by Lemma \ref{Lemma: abc} with $w=(1-r_\parallel)v_{q-1,\parallel}$, the third line of~\eqref{eqn: int over V_p} is bounded by
\[C_{T_M}\exp\left(\big[ \frac{[T_{l,q}-T_w(x_{q})]}{2T_w(x_{q})[T_{l,q}(1-r_\parallel)^2+r(2-r_\parallel) T_w(x_{q})]} + \mathcal{C}_{l-q+1}t_*^c\big]|(1-r_\parallel)v_{q-1,\parallel}|^2\right)\]
\begin{equation}\label{eqn: Vq para}
\leq C_{T_M}\exp\left(\big[ \frac{[T_{l,q}-T_w(x_{q})][1-r_{min}]}{2T_w(x_{q})[T_{l,q}(1-r_{min})+r_{min} T_w(x_{q})]} + \mathcal{C}_{l-q+1}t_*^c\big]|v_{q-1,\parallel}|^2\right).
\end{equation}
By the same computation as~\eqref{eqn: result of dsigmal normal} the second line of~\eqref{eqn: int over V_p} is bounded by
\begin{equation}\label{eqn: Vq perp}
\frac{2}{\sqrt{T_m}}\sqrt{C_{T_M}}\exp\left(\big[ \frac{[T_{l,q}-T_w(x_{q})][1-r_{min}]}{2T_w(x_{q})[T_{l,q}(1-r_{min})+r_{min} T_w(x_{q})]} + \mathcal{C}_{l-q+1}t_*^c\big]|v_{q-1,\perp}|^2\right).
\end{equation}
By~\eqref{eqn: Vq para} and~\eqref{eqn: Vq perp} and the notation~\eqref{Constant term}, we derive that
\begin{align*}
\eqref{eqn: int over V_p}   & \leq t_*^{-(l-q+1)c}(C_{T_M,T_m})^{l-q+1}\exp\left(\big[ \frac{[T_{l,q}-T_w(x_{q})][1-r_{min}]}{2T_w(x_{q})[T_{l,q}(1-r_{min})+r_{min} T_w(x_{q})]} + \mathcal{C}_{l-q+1}t_*^c\big]|v_{q-1}|^2\right) \\
   & =t_*^{-(l-q+1)c} C_{T_M,T_M}^{l-q+1}\mathcal{A}_{l,q},
\end{align*}
which is consistent with~\eqref{eqn: boundedness for l-p+1 fold integration} with $p=q$. The induction is valid we derive~\eqref{eqn: boundedness for l-p+1 fold integration}.

Now we focus on~\eqref{eqn: structure}. The first inequality in~\eqref{eqn: structure} follows directly from~\eqref{eqn: boundedness for l-p+1 fold integration} and~\eqref{eqn: ppt for Phi}. For the second inequality, by~\eqref{eqn: upsilon} and Lemma \ref{Lemma: extra term} we have
\begin{align}
 &t_*^{-(l-p'+1)c} C_{T_M,T_m}^{l-p'+1} \int_{\prod_{j=p}^{p'-1} \mathcal{V}_j}  \mathbf{1}_{\{t_l>0\}} \mathcal{A}_{l,p'} \dd\Upsilon_p^{p'-1}\notag \\
 & \leq t_*^{-(l-p'+1)c} C_{T_M,T_m}^{l-p'+1} \int_{\prod_{j=p}^{p'-2} \mathcal{V}_j} \int_{\mathcal{V}_{p'-1}} \mathbf{1}_{\{t_l>0\}} \mathcal{A}_{l,p'}  \frac{e^{[\frac{1}{2T_w(x_{p'-1})}-\frac{1}{2T_w(x_{p'})}+t_*^c]|v_{p'-1}|^2}}{n(v_{p'-1})\cdot v_{p'-1}} \dd\sigma(v_{p'-1},v_{p'-2})  \dd\Upsilon_p^{p'-2}.\label{eqn: second ineq}
\end{align}
In the proof of~\eqref{eqn: boundedness for l-p+1 fold integration} we have
\[\eqref{eqn: dsigmaq}\leq \eqref{eqn: dsigmap}\leq \eqref{eqn: int over V_p} \leq t_*^{-(l-q+1)c} C_{T_M,T_m}^{l-q+1}\mathcal{A}_{l,q}.\]
Then by replacing $q$ by $p'-1$ in the estimate $~\eqref{eqn: dsigmaq}\leq t_*^{-(l-q+1)c} C_{T_M,T_m}^{l-q+1}\mathcal{A}_{l,q}$ we have

\[\eqref{eqn: second ineq}\leq t_*^{-(l-p'+2)c} C_{T_M,T_m}^{l-p'+2}\int_{\prod_{j=p}^{p'-2} \mathcal{V}_j} \mathbf{1}_{\{t_l>0\}} \mathcal{A}_{l,p'-1} \dd\Upsilon_p^{p'-2}   .\]
Keep doing this computation until integrating over $\mathcal{V}_p$ we obtain the second inequality in~\eqref{eqn: structure}.

\end{proof}

The next lemma conclude the smallness of the last term of~\eqref{eqn: formula for H}.
\begin{lemma}\label{lemma: t^k}
Assume
\begin{equation}\label{eqn: assume T}
\frac{T_m}{T_M}>\max\Big(\frac{1-r_\parallel}{2-r_\parallel},\frac{\sqrt{1-r_\perp}-(1-r_\perp)}{r_\perp}\Big).
\end{equation}
 For the last term of~\eqref{eqn: formula for H}, we require $t_*$ in~\eqref{eqn: t*} further satisfies the condition~\eqref{eqn: B i para} and~\eqref{condition for t*}(these conditions are consistent with the dependent variables in~\eqref{eqn: t*}). Then there exists
\begin{equation}\label{eqn: k_0 dependence}
k_0=k_0(\Omega,C_{T_M,T_m},\mathcal{C},T_M,r_\perp,r_\parallel)\gg 1,
\end{equation}
such that for all $t<t_*$, we have
\begin{equation}\label{eqn: 1/2 decay}
   \int_{\prod_{j=1}^{k_0-1}\mathcal{V}_j} \mathbf{1}_{\{t_{k_0}>0\}}\dd\Sigma_{k_0-1}^{k_0}  \leq (\frac{1}{2})^{k_0} \mathcal{A}_{k_0-1,1},
\end{equation}
where $\mathcal{A}_{k_0-1,1}$ is defined in~\eqref{eqn: Elp}.
\end{lemma}

\begin{remark}
The key difference between Lemma \ref{lemma: t^k} and Lemma \ref{lemma: boundedness} is that we have the small term $(\frac{1}{2})^{k_0}$. With this extra term Lemma \ref{lemma: t^k} implies the measure of the last term of~\eqref{eqn: formula for H} is small provided $k=k_0$ is large enough.
Such property is essential in our analysis since we then only need to consider a finite-fold integration and bound the rest fold by small magnitude number.

The $k_0$ is specified in~\eqref{eqn: k_0 dependence}. Combining with~\eqref{eqn: t*} with $c=\frac{1}{15}$ specified in~\eqref{c} we conclude
\begin{equation}\label{t* dependence}
t_*=t_*(\Omega,T_M,T_m,r_\parallel,r_\perp).
\end{equation}
\end{remark}

We need several preparations to prove Lemma \ref{lemma: t^k}.

\begin{lemma}\label{Lemma: (2)}
For $1\leq i\leq k-1$, if
\begin{equation}\label{eqn: 2 condition}
|v_i\cdot n(x_i)|<\delta,
\end{equation}
then
\begin{equation}\label{eqn: 2}
    \int_{\prod_{j=i}^{k-1} \mathcal{V}_j} \mathbf{1}_{\{v_i\in \mathcal{V}_i:|v_i\cdot n(x_i)|<\delta\}}   \mathbf{1}_{\{t_k>0\}} \dd\Phi_{i,m}^{k,k-1} \leq \delta t_*^{-(k-i)c}   C_{T_M,T_m}^{k-i}\mathcal{A}_{k-1,i}.
\end{equation}

If
\begin{equation}\label{eqn: b condition}
|v_{i,\parallel}-\eta_{i,\parallel}v_{i-1,\parallel}|>\delta^{-1},
\end{equation}
then
\begin{equation}\label{eqn: case b}
\int_{\prod_{j=i}^{k-1} \mathcal{V}_j}\mathbf{1}_{\{t_k>0\}}\mathbf{1}_{\{|v_{i,\parallel}-\eta_{i,\parallel}v_{i-1,\parallel}|>\delta^{-1}\}}\dd\Phi_{i,m}^{k,k-1}
 \leq \delta    t_*^{-(k-i)c}   C_{T_M,T_m}^{k-i}\mathcal{A}_{k-1,i}.
\end{equation}
Here $\eta_{i,\parallel}$ is a constant defined in~\eqref{eqn: eta i para}.

If
\begin{equation}\label{eqn: d condition}
|v_{i,\perp}-\eta_{i,\perp}v_{i-1,\perp}|>\delta^{-1},
\end{equation}
then
\begin{equation}\label{eqn: case d}
\int_{\prod_{j=i}^{k-1} \mathcal{V}_j}\mathbf{1}_{\{t_k>0\}}\mathbf{1}_{\{|v_{i,\perp}-\eta_{i,\perp}v_{i-1,\perp}|>\delta^{-1}\}}\dd\Phi_{i,m}^{k,k-1}
 \leq \delta    t_*^{-(k-i)c}   C_{T_M,T_m}^{k-i}\mathcal{A}_{k-1,i}.
\end{equation}
Here $\eta_{i,\perp}$ is a constant defined in~\eqref{eqn: eta i perp}.

\end{lemma}

\begin{proof}
First we focus on~\eqref{eqn: 2}. By~\eqref{eqn: int over V_p} in Lemma \ref{lemma: boundedness}, we can replace $l$ by $k-1$ and replace $q$ by $i$ to obtain
\begin{equation}\label{eqn: int V_i}
\begin{split}
  & \int_{\prod_{j=i}^{k-1} \mathcal{V}_j}    \mathbf{1}_{\{t_k>0\}} \dd\Phi_{i,m}^{k,k-1}\leq t_*^{-(k-i)c} C_{T_M,T_m}^{k-i} \\
   &  \times \int_{\mathcal{V}_{i,\perp}} \frac{1}{r_\perp T_w(x_i)}  e^{-[\frac{1}{2T_{k-1,i}}-\frac{1}{2T_w(x_i)}-\mathcal{C}_{k-i}t_*^c-t_*^c]|v_{i,\perp}|^2}I_0\left(\frac{(1-r_\perp)^{1/2}v_{i,\perp}v_{i-1,\perp}}{T_w(x_i)r_\perp}\right)e^{-\frac{|v_{i,\perp}|^2+(1-r_\perp)|v_{i-1,\perp}|^2}{2T_w(x)r_\perp}}  \dd v_{i,\perp} \\
    & \times\int_{\mathcal{V}_{i,\parallel}}\frac{1}{\pi r_\parallel(2-r_\parallel)(2T_w(x_i))}e^{-[\frac{1}{2T_{k-1,i}}-\frac{1}{2T_w(x_i)}-\mathcal{C}_{k-i}t_*^c-t_*^c]|v_{i,\parallel}|^2}e^{-\frac{1}{2T_w(x_i)}\frac{|v_{i,\parallel}-(1-r_\parallel)v_{i-1,\parallel}|^2}{r_\parallel(2-r_\parallel)}}\dd v_{i,\parallel}.
\end{split}
\end{equation}
Under the condition~\eqref{eqn: 2 condition}, we consider the second line of~\eqref{eqn: int V_i} with integrating over $\{v_{i,\perp}\in \mathcal{V}_{i,\perp}:|v_i\cdot n(x_i)|<\frac{1-\eta}{2(1+\eta)}\delta\}$. To apply~\eqref{eqn: coe abc perp small} in Lemma \ref{Lemma: perp abc}, we set
\[a=-[\frac{1}{2T_{k-1,i}}-\frac{1}{2T_w(x_{i})}],~ b=\frac{1}{2T_w(x_{i})r_\perp},~\e=\mathcal{C}_{k-i}t_*^c+t_*^c,~w=\sqrt{1-r_\perp}v_{i-1,\perp}.\]
Under the condition $|v_i\cdot n(x_i)|<\frac{1-\eta}{2(1+\eta)}\delta$, applying~\eqref{eqn: coe abc perp small} in Lemma \ref{Lemma: perp abc} and using~\eqref{eqn: Vq perp} with $q=i,l=k-1$, we bound the second line of~\eqref{eqn: int V_i} by
\begin{equation}\label{eqn: perp small}
\delta \frac{2}{\sqrt{T_m}}\sqrt{C_{T_M}}\exp\left(\big[ \frac{[T_{k-1,i}-T_w(x_{i})][1-r_{min}]}{2T_w(x_{i})[T_{k-1,i}(1-r_{min})+r_{min} T_w(x_{i})]} + \mathcal{C}_{k-i+1}t_*^c\big]|v_{i-1,\perp}|^2\right).
\end{equation}
Taking~\eqref{eqn: Vq perp} for comparison, we conclude the second line of~\eqref{eqn: int V_i} provides one more constant term $\delta$. The third line of~\eqref{eqn: int V_i} is bounded by~\eqref{eqn: Vq para} with $q=i,l=k-1$. Therefore, we derive~\eqref{eqn: 2}.

Then we focus on~\eqref{eqn: case b}. We consider the third line of~\eqref{eqn: int V_i}. To apply~\eqref{eqn: coe abc small} in Lemma \ref{Lemma: abc}, we set
\begin{equation}\label{eqn: abe}
     a=-\frac{1}{2T_{k-1,i}} +\frac{1}{2T_w(x_i)},\quad b=\frac{1}{2T_w(x_i)r_\parallel(2-r_\parallel)},\quad \e=\mathcal{C}_{k-i}t_*^c+t_*^c,~w=(1-r_\parallel)v_{i-1,\parallel}.
\end{equation}
We define
\begin{equation}\label{eqn: B i para def}
B_{i,\parallel}:=b-a-\e.
\end{equation}
In regard to~\eqref{eqn: coe abc small},
\[    \frac{b}{b-a-\e}w=\frac{b}{b-a}[1+\frac{\e}{b-a-\e}] w. \]
By~\eqref{eqn: abe},
\[\frac{b}{b-a}=\frac{T_{k-1,i}}{T_{k-1,i}(1-r_\parallel)^2+T_w(x_i)r_\parallel(2-r_\parallel)},\quad \frac{\e}{b-a-\e}=\frac{\mathcal{C}_{k-i}t_*^c+t_*^c}{B_{i,\parallel}}.\]
Thus we obtain
\begin{equation}\label{eqn: constant for the t^k}
    \frac{b}{b-a-\e}w=\eta_{i,\parallel}v_{i-1,\parallel},
\end{equation}
where we defined
\begin{equation}\label{eqn: eta i para}
\eta_{i,\parallel}:=\frac{T_{k-1,i}[1+(\mathcal{C}_{k-i}+1)t_*^c/B_{i,\parallel}]}{T_{k-1,i}(1-r_\parallel)^2+T_w(x_i)r_\parallel(2-r_\parallel)}(1-r_\parallel).
\end{equation}
Thus under the condition~\eqref{eqn: b condition}, applying~\eqref{eqn: coe abc small} in Lemma \ref{Lemma: abc} with $\frac{b}{b-a-\e}w=\eta_{i,\parallel}v_{i-1,\parallel}$ and using~\eqref{eqn: Vq para} with $q=i,l=k-1$, we bound the third line of~\eqref{eqn: int V_i} by
\[\delta C_{T_M}\exp\left(\big[ \frac{[T_{k-1,i}-T_w(x_{i})][1-r_{min}]}{2T_w(x_{i})[T_{k-1,i}(1-r_{min})+r_{min} T_w(x_{i})]} + \mathcal{C}_{k-i+1}t_*^c\big]|v_{i-1,\parallel}|^2\right).\]
Thus we derive~\eqref{eqn: case b} due to the extra constant $\delta$.

Last we focus on~\eqref{eqn: case d}. We consider the second line of~\eqref{eqn: int V_i} with integrating over $\{v_{i,\perp}:v_{i,\perp}\in \mathcal{V}_{i,\perp},|v_{i,\perp}|>\frac{1+\eta}{1-\eta}\delta^{-1}\}$. To apply~\eqref{eqn: coe abc perp small} in Lemma \ref{Lemma: integrate normal small}, we set
\begin{equation}\label{eqn: abe perp}
     a=-\frac{1}{2T_{k-1,i}} +\frac{1}{2T_w(x_i)},\quad b=\frac{1}{2T_w(x_i)r_\perp},\quad \e=\mathcal{C}_{k-i}t_*^c+t_*^c,~ w=\sqrt{1-r_\perp}v_{i-1,\perp}.
\end{equation}
Define
\begin{equation}\label{eqn: B i perp}
B_{i,\perp}:=b-a-\e.
\end{equation}
By the same computation as~\eqref{eqn: constant for the t^k},
\[\frac{b}{b-a-\e}w=\eta_{i,\perp}v_{i-1,\perp},\]
where we defined
\begin{equation}\label{eqn: eta i perp}
\eta_{i,\perp}:= \frac{T_{k-1,i}[1+\frac{(\mathcal{C}_{k-i}+1)t_*^c}{B_{i,\perp}}]}{T_{k-1,i}(1-r_\perp)+T_w(x_i)r_\perp}\sqrt{1-r_\perp}.
\end{equation}
Thus under the condition~\eqref{eqn: d condition}, applying~\eqref{eqn: coe perp small 2} in Lemma~\ref{Lemma: integrate normal small} with $\frac{b}{b-a-\e}w=\eta_{i,\perp}v_{i-1,\perp}$ and using~\eqref{eqn: Vq perp} with $q=i,l=k-1$, we bound the second line of~\eqref{eqn: int V_i} by
\[\delta \frac{2}{\sqrt{T_m}}\sqrt{C_{T_M}}\exp\left(\big[ \frac{[T_{k-1,i}-T_w(x_{i})][1-r_{min}]}{2T_w(x_{i})[T_{k-1,i}(1-r_{min})+r_{min} T_w(x_{i})]} + \mathcal{C}_{k-i+1}t_*^c\big]|v_{i-1,\perp}|^2\right).\]
Then we derive~\eqref{eqn: case b} due to the extra constant $\delta$.

\end{proof}

\begin{lemma}\label{Lemma:  (a)(c)}
For $\eta_{i,\parallel}$ and $\eta_{i,\perp}$ defined in Lemma \ref{Lemma: (2)}, we suppose there exists $\eta<1$ such that
\begin{equation}\label{eqn: eta condition}
  \max\{\eta_{i,\parallel},\eta_{i,\perp}\}<\eta<1.
\end{equation}
Then if
\begin{equation}\label{eqn: (a) condition}
  |v_{i,\parallel}|>\frac{1+\eta}{1-\eta}\delta^{-1} \text{ and } |v_{i,\parallel}-\eta_{i,\parallel}v_{i-1,\parallel}|<\delta^{-1},
\end{equation}
we have
\begin{equation}\label{eqn: (a)}
  |v_{i-1,\parallel}|>|v_{i,\parallel}|+\delta^{-1}.
\end{equation}

Also if
\begin{equation}\label{eqn: (c) condition}
  |v_{i,\perp}|>\frac{1+\eta}{1-\eta}\delta^{-1} \text{ and } |v_{i,\perp}-\eta_{i,\perp}v_{i-1,\perp}|<\delta^{-1},
\end{equation}
then we have
\begin{equation}\label{eqn: (c)}
  |v_{i-1,\perp}|>|v_{i,\perp}|+\delta^{-1}.
\end{equation}

\end{lemma}
\begin{remark}
Lemma \ref{Lemma: (2)} includes the ``good" cases since those extra small factor $\delta$ contributes to the decaying constant in Lemma \ref{lemma: t^k}.
Lemma \ref{Lemma:  (a)(c)} discusses those ``bad" cases since such cases do not directly provide any small factor. Thus those cases are the main difficulty in our estimate. In Lemma \ref{Lemma: Step3} we will specify the way to handle them using the properties in this lemma.
\end{remark}

\begin{proof}
Under the condition~\eqref{eqn: (a) condition} we have
\[\eta_{i,\parallel}|v_{i-1,\parallel}|> |v_{i,\parallel}|-\delta^{-1}.\]
Thus we derive
\begin{align*}
|v_{i-1,\parallel}|   &> |v_{i,\parallel}|+\frac{1-\eta_{i,\parallel}}{\eta_{i,\parallel}}|v_{i,\parallel}|-\frac{1}{\eta_{i,\parallel}}\delta^{-1} \\
   & >|v_{i,\parallel}|+\frac{1-\eta_{i,\parallel}}{\eta_{i,\parallel}}\frac{1+\eta}{1-\eta}\delta^{-1}-\frac{1}{\eta_{i,\parallel}}\delta^{-1}\\
   &>|v_{i,\parallel}|+\frac{1-\eta_{i,\parallel}}{\eta_{i,\parallel}}\frac{1+\eta_{i,\parallel}}{1-\eta_{i,\parallel}}\delta^{-1}-\frac{1}{\eta_{i,\parallel}}\delta^{-1}\\
   &>|v_{i,\parallel}|+\frac{1+\eta_{i,\parallel}}{\eta_{i,\parallel}}\delta^{-1}-\frac{1}{\eta_{i,\parallel}}\delta^{-1}>|v_{i,\parallel}|+\delta^{-1},
\end{align*}
where we used $|v_{i,\parallel}|>\frac{1+\eta}{1-\eta}\delta^{-1}$ in the second line and $1>\eta\geq \eta_{i,\parallel}$ in the third line. Then we obtain~\eqref{eqn: (a)}.

Under the condition~\eqref{eqn: (c) condition}, we apply the same computation above to obtain~\eqref{eqn: (c)}.

\end{proof}

\begin{lemma}\label{Lemma: accumulate}
Suppose there are $n$ number of $v_j$ such that
\begin{equation}\label{eqn: satisfy condition}
|v_{j,\parallel}-\eta_{j,\parallel}v_{j-1,\parallel}|\geq \delta^{-1},
\end{equation}
and also suppose the index $j$ in these $v_j$ are $i_1<i_2<\cdots<i_n$, then
\begin{equation}\label{eqn: claim M}
\int_{\prod_{j={i_1}}^{k-1} \mathcal{V}_j}\mathbf{1}_{\{t_k>0\}}\mathbf{1}_{\{\text{~\eqref{eqn: satisfy condition} holds for $j=i_1,i_2,\cdots, i_n$}\}}\dd\Phi_{i_1,m}^{k,k-1} \leq  (\delta)^{n}   t_*^{-(k-i_1)c}C_{T_M,T_m}^{k-i_1}\mathcal{A}_{k-1,i_1}.
\end{equation}

\end{lemma}

\begin{proof}
By~\eqref{eqn: structure} in Lemma 2 with $l=k-1$, $p=i_1$, $p'=i_n$ and using~\eqref{eqn: case b} with $i=i_n$, we have
\begin{align}
   & \int_{\prod_{j=i_1}^{k-1} \mathcal{V}_j}\mathbf{1}_{\{t_k>0\}}\mathbf{1}_{\{\text{~\eqref{eqn: satisfy condition} holds for $j=i_1,\cdots,i_n$}\}}\dd\Phi_{i_1,m}^{k,k-1}\notag \\
   & \leq \delta t_*^{-(k-i_n)c} C_{T_M,T_m}^{k-i_n} \int_{\prod_{j=i_1}^{i_n-1} \mathcal{V}_j} \mathcal{A}_{k-1,i_n} \mathbf{1}_{\{t_k>0\}}\mathbf{1}_{\{\text{~\eqref{eqn: satisfy condition} holds for $j=i_1,\cdots,i_{n-1}$}\}}        \dd\Upsilon_{i_1}^{i_n-1}\notag\\
   &=\delta t_*^{-(k-i_n)c} C_{T_M,T_m}^{k-i_n} \int_{\prod_{j=i_1}^{i_{n-1}-1}\mathcal{V}_j}\int_{\prod_{j=i_{n-1}}^{(i_n)-1} \mathcal{V}_j} \mathcal{A}_{k-1,i_n} \mathbf{1}_{\{t_k>0\}}\mathbf{1}_{\{\text{~\eqref{eqn: satisfy condition} holds for $j=i_1,\cdots,i_{n-1}$}\}}      \dd\Upsilon_{i_{n-1}}^{(i_{n})-1}\dd\Upsilon_{i_1}^{i_{n-1}-1}.\label{eqn: split iM}
\end{align}

Again by~\eqref{eqn: structure} and~\eqref{eqn: case b} with $i=i_{n-1}$ we have
\[\eqref{eqn: split iM}\leq \delta^2 t_*^{-(k-i_{n-1})c} C_{T_M,T_m}^{k-i_{n-1}}\int_{\prod_{j=i_1}^{i_{n-1}-1} \mathcal{V}_j}\mathcal{A}_{k-1,i_{n-1}} \mathbf{1}_{\{t_k>0\}}\mathbf{1}_{\{\text{~\eqref{eqn: satisfy condition} holds for $j=i_1,\cdots,i_{n-2}$}\}}        \dd\Upsilon_{i_1}^{i_{n-1}-1}.\]
Keep doing this computation until integrating over $\mathcal{V}_{i_1}$ we derive~\eqref{eqn: claim M}.

\end{proof}

\begin{lemma}\label{Lemma: Step3}
For $0<\delta\ll 1$, we define
\begin{equation}\label{eqn: decom}
  \mathcal{V}_{j}^{\delta}:=\{v_j\in \mathcal{V}_j:|v_j\cdot n(x_j)|>\delta,|v_j|\leq \delta^{-1}\}.
\end{equation}
For the sequence $\{v_1,v_2,\cdots,v_{k-1}\}$, consider a subsequence  $\{v_{l+1},v_{l+2},\cdots,v_{l+L}\}$ with $l+1<l+L\leq k-1$ as follows:
\begin{equation}\label{eqn: sequence}
  \underbrace{v_{l},}_{\in \mathcal{V}_l^{\frac{1-\eta}{2(1+\eta)}\delta}}\quad \quad \underbrace{v_{l+1},v_{l+2}\cdots v_{l+L}}_{\text{all}\in \mathcal{V}_{l+j}\backslash \mathcal{V}_{l+j}^{\frac{1-\eta}{2(1+\eta)}\delta}},\quad \quad\underbrace{v_{l+L+1}}_{\in \mathcal{V}_{l+L+1}^{\frac{1-\eta}{2(1+\eta)}\delta}}.
\end{equation}
In~\eqref{eqn: sequence}, if $L\geq 100\frac{1+\eta}{1-\eta}$, then we have
\begin{equation}\label{eqn: Step3}
\int_{\prod_{j={l}}^{k-1} \mathcal{V}_j}\mathbf{1}_{\{t_k>0\}}\mathbf{1}_{\{v_{l+j}\in \mathcal{V}_{l+j}\backslash \mathcal{V}_{l+j}^{\frac{1-\eta}{2(1+\eta)}\delta} \text{ for } 1\leq j\leq L\}}\dd\Phi_{l,m}^{k,k-1} \leq  (3\delta)^{L/2}   t_*^{-(k-l)c}C_{T_M,T_m}^{k-l}\mathcal{A}_{k-1,l}.
\end{equation}
Here the $\eta$ satisfies the condition~\eqref{eqn: eta condition}.

\end{lemma}
\begin{remark}
In this lemma we combine the estimates in Lemma \ref{Lemma: (2)} and Lemma \ref{Lemma:  (a)(c)} and derive the desired decaying term $(3\delta)^{L/2}$. In the proof we will address the difficulty stated in Lemma \ref{Lemma:  (a)(c)}.

\end{remark}

\begin{proof}
By the definition~\eqref{eqn: decom} we have
\[\mathcal{V}_{i}\backslash \mathcal{V}_{i}^{\frac{1-\eta}{2(1+\eta)}\delta}  =\{v_i\in \mathcal{V}_i:|v_i\cdot n(x_i)|<\frac{1-\eta}{2(1+\eta)}\delta \text{ or }|v_i|\geq \frac{2(1+\eta)}{1-\eta}\delta^{-1}\}.\]
Here we summarize the result of Lemma \ref{Lemma: (2)} and Lemma \ref{Lemma:  (a)(c)}.
With $\frac{1-\eta}{1+\eta}\delta<\delta$, when $v_i\in \mathcal{V}_i\backslash \mathcal{V}_i^{\frac{1-\eta}{2(1+\eta)}\delta}$
\begin{enumerate}

  \item When $|v_i\cdot n(x_i)|<\frac{1-\eta}{2(1+\eta)}\delta$, then we have~\eqref{eqn: 2}.

  \item When $|v_{i}|>\frac{2(1+\eta)}{1-\eta}\delta^{-1}$,
   \begin{enumerate}
     \item when $|v_{i,\parallel}|>\frac{1+\eta}{1-\eta}\delta^{-1}$, if $|v_{i,\parallel}-\eta_{i,\parallel}v_{i-1,\parallel}|<\delta^{-1}$, then $|v_{i-1,\parallel}|>|v_{i,\parallel}|+\delta^{-1}$. \\
     \item when $|v_{i,\parallel}|>\frac{1+\eta}{1-\eta}\delta^{-1}$, if $|v_{i,\parallel}-\eta_{i,\parallel}v_{i-1,\parallel}|\geq \delta^{-1}$, then we have~\eqref{eqn: case b}. \\
     \item when $|v_{i,\perp}|>\frac{1+\eta}{1-\eta}\delta^{-1}$, if $|v_{i,\perp}-\eta_{i,\perp}v_{i-1,\perp}|<\delta^{-1}$, then $|v_{i-1,\perp}|>|v_{i,\perp}|+\delta^{-1}$ .\\
     \item when $|v_{i,\perp}|>\frac{1+\eta}{1-\eta}\delta^{-1}$, if $|v_{i,\perp}-\eta_{i,\perp}v_{i-1,\perp}|\geq \delta^{-1}$, then we have~\eqref{eqn: case d}.\\
   \end{enumerate}

\end{enumerate}

We define $\mathcal{W}_{i,\delta}$ as the space that provides the smallness:
\[\mathcal{W}_{i,\delta}:=\{v_i\in \mathcal{V}_i:|v_{i,\perp}|<\frac{1-\eta}{2(1+\eta)}\delta\}\bigcup \{v_i\in \mathcal{V}_i:|v_{i,\perp}|>\frac{1+\eta}{1-\eta}\delta^{-1}\text{ and }|v_{i,\perp}-\eta_{i,\perp}v_{i-1,\perp}|>\delta^{-1}\}\]
\[\bigcup \{v_i\in \mathcal{V}_i:|v_{i,\parallel}|>\frac{1+\eta}{1-\eta}\delta^{-1}\text{ and }|v_{i,\parallel}-\eta_{i,\parallel}v_{i-1,\parallel}|>\delta^{-1}\}.\]
Then we have
\begin{equation}\label{eqn: subset}
  \begin{split}
     & \mathcal{V}_{i}\backslash \mathcal{V}_{i}^{\frac{1-\eta}{2(1+\eta)}\delta} \subset \mathcal{W}_{i,\delta} \bigcup \{v_{i,\perp}\in \mathcal{V}_{i,\perp}|v_{i,\perp}|>\frac{1+\eta}{1-\eta}\delta^{-1}~\text{and}~|v_{i,\perp}-\eta_{i,\perp}v_{i-1,\perp}|<\delta^{-1}\}
 \\
      & \bigcup \{v_{i,\parallel}\in \mathcal{V}_{i,\parallel}|v_{i,\parallel}|>\frac{1+\eta}{1-\eta}\delta^{-1}~\text{and}~|v_{i,\parallel}-\eta_{i,\parallel}v_{i-1,\parallel}|<\delta^{-1}\}.
  \end{split}
\end{equation}
By~\eqref{eqn: 2},~\eqref{eqn: case b} and~\eqref{eqn: case d} with $\frac{1-\eta}{1+\eta}\delta<\delta$, we obtain
\begin{equation}\label{eqn: V_i,delta}
  \int_{\prod_{j=i}^{k-1}\mathcal{V}_j}  \mathbf{1}_{\{v_i\in \mathcal{W}_{i,\delta}\}}   \mathbf{1}_{\{t_k>0\}}  \dd\Phi_{i,m}^{k,k-1} \leq 3\delta   t_*^{-(k-i)c} C_{T_M,T_m}^{k-i}\mathcal{A}_{k-1,i}.
\end{equation}

For the subsequence $\{v_{l+1},\cdots,v_{l+L}\}$ in~\eqref{eqn: sequence}, when the number of $v_j\in \mathcal{W}_{j,\delta}$ is larger than $L/2$, by~\eqref{eqn: claim M} in Lemma \ref{Lemma: accumulate} with $n=L/2$ and replacing the condition~\eqref{eqn: satisfy condition} by $v_j\in \mathcal{W}_{j,\delta}$, we obtain
\begin{align}
   & \int_{\prod_{j=l}^{k-1} \mathcal{V}_j}   \mathbf{1}_{\{\text{Number of }v_j\in \mathcal{W}_{j,\delta} \text{ is larger than }L/2\}}    \mathbf{1}_{\{t_k>0\}} \dd\Phi_{l,m}^{k,k-1} \\
   & \leq (3\delta)^{L/2}    t_*^{-(k-l_i)c}C_{T_M,T_m}^{k-l_i}\mathcal{A}_{k-1,l}.\label{eqn: 3delta}
\end{align}
We finish the discussion with the case(1),(2b),(2d). Then we focus on the case (2a),(2c).

When the number of $v_j \notin \mathcal{W}_{j,\delta}$ is larger than $L/2$, by~\eqref{eqn: subset} we further consider two cases. The first case is that the number of $v_j\in \{v_j:|v_{j,\parallel}|>\frac{1+\eta}{1-\eta}\delta^{-1}~\text{and}~|v_{j,\parallel}-\eta_{j,\parallel}v_{j-1,\parallel}|<\delta^{-1}\}$ is larger than $L/4$. According to the relation of $v_{j,\parallel}$ and $v_{j-1,\parallel}$, we categorize them into
\begin{description}
  \item[Set1] $\{v_j\notin \mathcal{W}_{j,\delta}:|v_{j,\parallel}|>\frac{1+\eta}{1-\eta}\delta^{-1}~\text{and}~|v_{j,\parallel}-\eta_{j,\parallel}v_{j-1,\parallel}|<\delta^{-1}\}$.
\end{description}
Denote $M=|\text{Set1}|$ and the corresponding index in Set1 as $j=p_1,p_2,\cdots,p_{M}$. Then we have
\begin{equation}\label{eqn: Mi'}
  L/4\leq M\leq L.
\end{equation}
By~\eqref{eqn: (a)} in Lemma \ref{Lemma:  (a)(c)}, for those $v_{p_j}$, we have
\begin{equation}\label{eqn: increase large}
|v_{p_j,\parallel}|-|v_{p_j-1,\parallel}|<-\delta^{-1}.
\end{equation}

\begin{description}
  \item[Set2]$\{v_j\in \mathcal{V}_j\backslash \mathcal{V}_j^{\frac{1-\eta}{2(1+\eta)\delta}}:|v_{j,\parallel}|\geq |v_{j-1,\parallel}|\}$.
\end{description}

Denote $\mathcal{M}=|\text{Set2}|$ and the corresponding index in Set2 as $j=q_1,q_2,\cdots,q_{\mathcal{M}}$. By~\eqref{eqn: Mi'} we have
\begin{equation}\label{eqn: mathcal M}
1\leq \mathcal{M}\leq L-M\leq \frac{3}{4}L.
\end{equation}
Then for those $v_{q_j}$ we define
\begin{equation}\label{eqn: ai def}
a_j:=|v_{q_j,\parallel}|-|v_{q_j-1,\parallel}|>0.
\end{equation}

 \begin{description}

  \item[Set3] $\{v_j\in \mathcal{V}_j\backslash \mathcal{V}_j^{\frac{1-\eta}{2(1+\eta)\delta}}:|v_{j,\parallel}|\leq |v_{j-1,\parallel}|\leq |v_{j,\parallel}|+\delta^{-1}\}$.
\end{description}

Denote $N=|\text{Set3}|$ and the corresponding index in Set3 as $j=o_1,o_2,\cdots,o_N$. Then for those $o_j$, we have
\begin{equation}\label{eqn: increase small}
|v_{o_j,\parallel}|\leq |v_{o_j-1,\parallel}|\leq |v_{o_j,\parallel}|+\delta^{-1}.
\end{equation}

From~\eqref{eqn: sequence}, we have $v_{l}\in \mathcal{V}_{l}^{\frac{1-\eta}{2(1+\eta)}\delta}$ and $v_{l+L+1}\in \mathcal{V}_{l+L+1}^{\frac{1-\eta}{2(1+\eta)}\delta}$, thus we can obtain
\begin{equation}\label{eqn: xiangjian}
 -\frac{2(1+\eta)}{1-\eta}\delta^{-1}<|v_{l+L+1,\parallel}|-|v_{l,\parallel}|= \sum_{j=1}^{L+1} |v_{l+j,\parallel}|-|v_{l+j-1,\parallel}|.
\end{equation}
By~\eqref{eqn: increase large},~\eqref{eqn: ai def} and~\eqref{eqn: increase small}, we derive that
\begin{align*}
 \frac{-2(1+\eta)}{1-\eta}\delta^{-1}  &<\sum_{j=1}^{M} \big(|v_{p_j,\parallel}|-|v_{p_j-1,\parallel}|\big)+\sum_{j=1}^{\mathcal{M}} \big(|v_{q_j,\parallel}|-|v_{q_j-1,\parallel}|\big)+\sum_{j=1}^{N} \big(|v_{o_j,\parallel}|-|v_{o_j-1,\parallel}|\big)   \\
   & \leq -M \delta^{-1}+\sum_{j=1}^{\mathcal{M}}a_j.
\end{align*}

Therefore, by $L\geq 100\frac{1+\eta}{1-\eta}$ and~\eqref{eqn: Mi'}, we obtain
\[\frac{2(1+\eta)}{1-\eta}\delta^{-1}\leq \frac{L}{10}\delta^{-1}\leq \frac{M}{2}\delta^{-1}\]
and thus
\begin{equation}\label{eqn: ai sum}
 \sum_{j=1}^{\mathcal{M}}a_j\geq M\delta^{-1}-\frac{2(1+\eta)}{1-\eta}\delta^{-1}>\frac{M\delta^{-1}}{2}.
\end{equation}
We focus on the integration over $\mathcal{V}_{q_i}$, such indexes satisfy~\eqref{eqn: ai def}. Let $1\leq i\leq \mathcal{M}$, we consider the third line of~\eqref{eqn: int V_i} with $i=q_i$ and with integrating over $\{v_{q_i,\parallel}\in \mathcal{V}_{q_i,\parallel}:|v_{q_i,\parallel}|-|v_{q_i-1,\parallel}|= a_i\}$. To apply~\eqref{eqn: coe abc small} in Lemma \ref{Lemma: abc}, we set
\[a=-\frac{1}{2T_{k-1,q_i}}+\frac{1}{2T_w(x_{q_i})},\quad b=\frac{1}{2T_w(x_{q_i})r_\parallel(2-r_\parallel)}, \quad \e=\mathcal{C}_{k-q_i}t_*^c+t_*^c.\]
By the same computation as~\eqref{eqn: B i para}, we have
\begin{equation}\label{eqn: k_1}
  a+\e-b=-\frac{1}{2T_{k-1,q_i}}+\frac{1}{2T_w(x_{q_i})}-\frac{1}{2T_w(x_{q_i})r_\parallel(2-r_\parallel)}+\mathcal{C}_{k-q_i}t_*^c+t_*^c <-\frac{1}{5T_M}.
\end{equation}
Then we use $\eta_{q_i,\parallel}<1$ to obtain
\begin{equation}\label{eqn: bound a_i}
 \mathbf{1}_{\{|v_{q_i,\parallel}|-|v_{q_i-1,\parallel}|= a_i\}}\leq  \mathbf{1}_{\{|v_{q_i,\parallel}|-\eta_{q_i,\parallel}|v_{q_i-1,\parallel}|>a_i\}}  \leq  \mathbf{1}_{\{|v_{q_i,\parallel}-\eta_{q_i,\parallel}v_{q_i-1,\parallel}|>a_i\}}.
\end{equation}
By~\eqref{eqn: coe abc small} in Lemma \ref{Lemma: abc} and~\eqref{eqn: bound a_i}, we apply~\eqref{eqn: Vq para} with $q=q_i$ to bound the third line of~\eqref{eqn: int V_i}( the integration over $\mathcal{V}_{q_i,\parallel}$ ) by
\begin{equation}\label{eqn: 4aiTM}
e^{-\frac{a_i^2}{4T_M}} C_{T_M}\exp\left(\big[ \frac{[T_{k-1,q_i}-T_w(x_{q_i})][1-r_{min}]}{2T_w(x_{q_i})[T_{k-1,q_i}(1-r_{min})+r_{min} T_w(x_{q_i})]} + \mathcal{C}_{k-q_i+1}t_*^c\big]|v_{q_i-1,\parallel}|^2\right).
\end{equation}
Hence by the constant in~\eqref{eqn: 4aiTM} we draw a similar conclusion as~\eqref{eqn: V_i,delta}:
\begin{equation}\label{eqn: case a_i}
   \int_{\prod_{j=q_i}^{k-1} \mathcal{V}_j}\mathbf{1}_{\{t_k>0\}}\mathbf{1}_{\{|v_{q_i,\parallel}|-|v_{q_i-1,\parallel}|= a_i\}}\dd\Phi_{q_i,m}^{k,k-1}
\leq e^{-\frac{a_i^2}{4T_M}} t_*^{-(k-q_i+1)c} C_{T_M,T_m}^{k-q_i+1}\mathcal{A}_{k-1,q_i}.
\end{equation}
Therefore, by Lemma \ref{Lemma: accumulate}, after integrating over $\mathcal{V}_{q_1,\parallel},\mathcal{V}_{q_2,\parallel},\cdots,\mathcal{V}_{q_\mathcal{M},\parallel}$ we obtain an extra constant
\[e^{-[a_i^2+a_2^2+\cdots +a_{\mathcal{M}}^2]/4T_M}\leq e^{-[a_i+a_2+\cdots +a_{\mathcal{M}}]^2/(4T_M\mathcal{M})}\leq e^{-[M\delta^{-1}/2]^2/(4T_M\mathcal{M})}\]
\[\leq e^{-[\frac{L}{8}\delta^{-1}]^2/(4T_M\frac{3}{4}L)}\leq e^{-\frac{1}{96T_M}L(\delta^{-1})^2}\leq e^{-L\delta^{-1}}.\]
Here we used~\eqref{eqn: ai sum} in the last step of first line and use~\eqref{eqn: Mi'},~\eqref{eqn: mathcal M} in the first step of second line and take $\delta\ll 1$ in the last step of second line. Then $e^{-L\delta^{-1}}$ is smaller than $(3\delta)^{L/2}$ in~\eqref{eqn: 3delta} and we conclude
\begin{equation}\label{eqn: 3delta 1}
  \int_{\prod_{j=l}^{k-1} \mathcal{V}_j}   \mathbf{1}_{\{ M=|\text{Set1}|\geq L/4\}}    \mathbf{1}_{\{t_k>0\}} \dd\Phi_{l,m}^{k,k-1}\leq (3\delta)^{L/2}    t_*^{-(k-l_i)c} C_{T_M,T_m}^{k-l_i}\mathcal{A}_{k-1,l}.
\end{equation}

The second case is that the number of $v_j\in \{v_j\notin \mathcal{W}_{j,\delta}:|v_{j,\perp}|>\frac{1+\eta}{1-\eta}\delta^{-1}\}$ is larger than $L/4$. We categorize $v_{j,\perp}$ into

\begin{description}
  \item[Set4] $\{v_j\notin \mathcal{W}_{j,\delta}:|v_{j,\perp}|>\frac{1+\eta}{1-\eta}\delta^{-1}~\text{and}~|v_{j,\perp}-\eta_{j,\perp}v_{j-1,\perp}|<\delta^{-1}\}$.
\end{description}

\begin{description}
  \item[Set5]$\{v_j\in \mathcal{V}_j\backslash \mathcal{V}_j^{\frac{1-\eta}{2(1+\eta)\delta}}:|v_{j,\perp}|>|v_{j-1,\perp}|\}$.
\end{description}

 \begin{description}

  \item[Set6] $\{v_j\in \mathcal{V}_j\backslash \mathcal{V}_j^{\frac{1-\eta}{2(1+\eta)\delta}}:|v_{j,\perp}|\leq |v_{j-1,\perp}|\leq |v_{j,\perp}|+\delta^{-1}\}$.
\end{description}
Denote $|\text{Set4}|=M_1$ and the corresponding index as $p'_1,p'_2,\cdots,p'_{M_1}$, $|\text{Set5}|=\mathcal{M}_1$ and the corresponding index as $q'_1,q'_2,\cdots,q'_{\mathcal{M}_1}$, $|\text{Set6}|=N_1$ and the corresponding index as $o'_1,o'_2,\cdots,o'_{N_1}$. Also define $b_j:=|v_{q'_j,\perp}|-|v_{q'_j-1,\perp}|$. By the same computation as~\eqref{eqn: ai sum}, we have
\[ \sum_{j=1}^{\mathcal{M}_1}b_j\geq M_1\delta^{-1}-\frac{2(1+\eta)}{1-\eta}\delta^{-1}>\frac{M_1\delta^{-1}}{2}.\]
We focus on the integration over $v_{q'_j}$. Let $1\leq i\leq \mathcal{M}_1$, we consider the second line of~\eqref{eqn: int V_i} with $i=q'_i$ and with integrating over $\{v_{q'_i,\perp}\in \mathcal{V}_{q'_i,\perp}:|v_{q'_i,\perp}|-|v_{q'_i-1,\perp}|= b_i\}$. To apply~\eqref{eqn: coe perp smaller 2} in Lemma \ref{Lemma: abc}, we set
\[a=-\frac{1}{2T_{k-1,q_i'}}+\frac{1}{2T_w(x_{q_i'})},\quad b=\frac{1}{2T_w(x_{q_i'})r_\perp}, \quad \e=\mathcal{C}^{k-q_i'}t_*^c+t_*^c.\]
By the same computation as~\eqref{eqn: B i para}, we have
\begin{equation}\label{eqn: k_1}
  a+\e-b=-\frac{1}{2T_{k-1,q_i'}}+\frac{1}{2T_w(x_{q_i'})}-\frac{1}{2T_w(x_{q_i'})r_\perp}+\mathcal{C}_{k-q_i'}t_*^c+t_*^c <-\frac{1}{5T_M}.
\end{equation}
Similar to~\eqref{eqn: bound a_i}, we have
\[ \mathbf{1}_{\{|v_{q'_i,\perp}|-|v_{q'_i-1,\perp}|= b_i\}}\leq \mathbf{1}_{\{|v_{q'_i,\perp}-\eta_{q'_i,\perp}v_{q'_i-1,\perp}|>b_i\}}.\]
Hence by~\eqref{eqn: coe perp smaller 2} in Lemma \ref{Lemma: integrate normal small} and applying~\eqref{eqn: Vq perp}, we bound the integration over $\mathcal{V}_{q'_i,\perp}$ by
\[e^{-\frac{b_i^2}{16T_M}} \frac{2}{\sqrt{T_m}}\sqrt{C_{T_M}}\exp\left(\big[ \frac{[T_{k-1,q_i'}-T_w(x_{q_i'})][1-r_{min}]}{2T_w(x_{q_i'})[T_{k-1,q_i'}(1-r_{min})+r_{min} T_w(x_{q_i'})]} + \mathcal{C}_{k-q_i'+1}t_*^c\big]|v_{q_i'-1,\perp}|^2\right).\]
Therefore,
\[\int_{\prod_{j=q_i'}^{k-1} \mathcal{V}_j}\mathbf{1}_{\{t_k>0\}}\mathbf{1}_{\{|v_{q_i',\perp}|-|v_{q_i'-1,\perp}|= b_i\}}\dd\Phi_{q_i',m}^{k,k-1}\leq e^{-\frac{b_i^2}{16T_M}} t_*^{-(k-q_i')c} C_{T_M,T_m}^{k-q_i'} \mathcal{A}_{k-1,q_i'}.\]
 The integration over $\mathcal{V}_{q'_1,\perp},\mathcal{V}_{q'_2,\perp},\cdots,\mathcal{V}_{q'_{\mathcal{M}_1},\perp}$ provides an extra constant
\[e^{-[b_1^2+b_2^2+\cdots +b_{\mathcal{M}_1}^2]/16T_M}\leq e^{-\frac{1}{400T_M}L (\delta^{-1})^2}\leq e^{-L\delta^{-1}},\]
where we set $\delta\ll 1$ in the last step. Then $e^{-L\delta^{-1}}$ is smaller than $(3\delta)^{L/2}$ in~\eqref{eqn: 3delta} and we conclude
\begin{equation}\label{eqn: 3delta 2}
\int_{\prod_{j=l}^{k-1} \mathcal{V}_j}   \mathbf{1}_{\{ M_1=|\text{Set4}|\geq L/4\}}    \mathbf{1}_{\{t_k>0\}} \dd\Phi_{l,m}^{k,k-1} \leq (3\delta)^{L/2}    t_*^{-(k-l)c} C_{T_M,T_m}^{k-l}\mathcal{A}_{k-1,l}.
\end{equation}

Finally collecting~\eqref{eqn: 3delta},~\eqref{eqn: 3delta 1} and~\eqref{eqn: 3delta 2} we derive the lemma.

\end{proof}

Now we are ready to prove the Lemma \ref{lemma: t^k}.

\begin{proof}[\textbf{Proof of Lemma \ref{lemma: t^k}}]

\textbf{Step 1}

To prove~\eqref{eqn: 1/2 decay} holds for the C-L boundary condition, we mainly use the decomposition~\eqref{eqn: decom} done by \cite{CKL} and~\cite{GKTT} for the diffuse boundary condition. In order to apply Lemma \ref{Lemma: Step3}, here we consider the space $\mathcal{V}_i^{\frac{1-\eta}{2(1+\eta)}\delta}$ and ensure $\eta$ satisfy the condition~\eqref{eqn: eta condition}. In this step we mainly focus on constructing the $\eta$, which will be defined in~\eqref{eqn: eta}.

First we consider $\eta_{i,\parallel}$, which is defined in~\eqref{eqn: eta i para}. In regard to~\eqref{eqn: abe} and~\eqref{eqn: B i para def}, we require $t_*=t_*(k,T_M,c,\mathcal{C})$( consistent with~\eqref{eqn: t*} ) to be small enough such that
\begin{equation}\label{eqn: B i para}
B_{i,\parallel}\geq \frac{1}{2T_{k-1,i}}-\mathcal{C}_{k-i}t_*^c-t_*^c\geq \frac{1}{4T_M}-\mathcal{C}_{k}t_*^c-t_*^c\geq \frac{1}{5T_M}.
\end{equation}
By~\eqref{eqn: formula of Tp}, $T_{k-1,i}\to T_M$ as $k-i\to \infty$. For any $\e_1>0$, there exists $k_1$ s.t when
\begin{equation}\label{eqn: e1}
k\geq k_1,\quad i\leq k/2, \text{ we have }T_{k-1,i}\leq (1+\e_1)T_M.
\end{equation}
Moreover, by~\eqref{eqn: assume T}, there exists $\e_2$ s.t
\begin{equation}\label{eqn: e2 def}
\frac{T_m}{T_M}>\frac{1-r_\parallel}{2-r_\parallel}(1+\e_2)
\end{equation}
and thus
\begin{equation}\label{eqn: e2 dependence}
\e_2=\e_2(T_m,T_M,r_\parallel,r_\perp).
\end{equation}

Thus we can bound $T_w(x_i)$ in the $\eta_{i,\parallel}$( defined in~\eqref{eqn: eta i para}) below as
\begin{equation}\label{eqn: bound below}
T_w(x_i)=T_{k-1,i}\frac{T_w(x_i)}{T_{k-1,i}}\geq T_{k-1,i}\frac{T_w(x_i)}{T_M}\frac{1}{1+\e_1}> \frac{1-r_\parallel}{2-r_\parallel}T_{k-1,i}\frac{1+\e_2}{1+\e_1}.
\end{equation}
Thus we obtain
\begin{equation}\label{eqn: eta i para bounded}
\eta_{i,\parallel}<\frac{1+\frac{(\mathcal{C}_{k-i}+1)t_*^c}{B_{i,\parallel}}}{(1-r_\parallel)^2+ \frac{1-r_\parallel}{2-r_\parallel}\frac{1+\e_2}{1+\e_1}r_\parallel(2-r_\parallel)}(1-r_\parallel)= \frac{1+\frac{(\mathcal{C}_{k-i}+1)t_*^c}{B_{i,\parallel}}}{1-r_\parallel+r_\parallel\frac{1+\e_2}{1+\e_1}}                                                                 .
\end{equation}
By~\eqref{eqn: e1}, we take
\begin{equation}\label{eqn: k_1 dependence}
k=k_1=k_1(\e_2,T_M,r_{\min})
\end{equation}
to be large enough such that $\e_1<\e_2/4$. By~\eqref{eqn: B i para} and~\eqref{eqn: eta i para bounded}, we derive that when $k=k_1$,
\begin{equation}\label{eqn: sup  less 1}
\sup_{i\leq k/2}\eta_{i,\parallel}\leq \frac{1+5T_M(\mathcal{C}_{k}+1)t_*^c}{1-r_\parallel+r_\parallel\frac{1+\e_2}{1+\e_2/4}}<\eta_\parallel<1.
\end{equation}
Here we define
\begin{equation}\label{eqn: eta_paral}
\eta_\parallel:=\frac{1}{1-r_\parallel+r_\parallel\frac{1+\e_2}{1+\e_2/2}}<1
\end{equation}
and we require $t_*=t'(k,T_M,\e_2,\mathcal{C},r_\parallel)$ to be small enough and such that
\begin{equation}\label{condition for t*}
  5T_M\mathcal{C}_{k}t_*^c\ll 1
\end{equation}
to ensure the second inequality in~\eqref{eqn: sup  less 1}.
Combining~\eqref{eqn: e2 dependence} and~\eqref{eqn: k_1 dependence}, we conclude the condition for $t_*$~\eqref{condition for t*} is consistent with~\eqref{eqn: t*}.

Then we consider $\eta_{i,\perp}$, which is defined in~\eqref{eqn: eta i perp}. In regard to~\eqref{eqn: abe perp} and~\eqref{eqn: B i perp}, by~\eqref{eqn: B i para} we have $B_{i,\perp}\geq \frac{1}{5T_M}.$ By $\frac{T_m}{T_M}>\frac{\sqrt{1-r_\perp}-(1-r_\perp)}{r_\perp}$ in~\eqref{eqn: assume T} we can use the same computation as~\eqref{eqn: bound below} to obtain
\[T_w(x_i)>  \frac{\sqrt{1-r_\perp}-(1-r_\perp)}{r_\perp}  T_{k-1,i}\frac{1+\e_2}{1+\e_1},\]
with $\e_1<\e_2/4$. Thus we obtain
\[\eta_{i,\perp}<\eta_\perp <1           ,                                                     \]
where we defined
\begin{equation}\label{eqn: eta perp}
\eta_{\perp}:=\frac{1}{\sqrt{1-r_\perp}+(1-\sqrt{1-r_\perp})\frac{1+\e_2}{1+\e_2/2}}<1
\end{equation}
with small enough $t_*=t_*(k,T_M,\e_2,\mathcal{C},r_\parallel)$( consistent with~\eqref{eqn: t*} ).

Finally we define
\begin{equation}\label{eqn: eta}
  \eta:=\max\{\eta_\perp,\eta_\parallel\}<1.
\end{equation}

\textbf{Step 2}

\textbf{Claim}: We have
\begin{equation}\label{eqn:claim_delta}
  |t_{j}-t_{j+1}|\gtrsim_\Omega \Big( \frac{1-\eta}{2(1+\eta)}\delta\Big)^3,\text{ for }v_j\in \mathcal{V}_j^{\frac{1-\eta}{2(1+\eta)}\delta},0\leq t_j.
\end{equation}
\begin{proof}

For $t_j\leq 1$,
\[|\int_{t_j}^{t_{j+1}}v_j \dd s|^2=|x_{j+1}-x_j|^2\gtrsim |(x_{j+1}-x_j)\cdot n(x_j)|\]
\[=|\int_{t_j}^{t_{j+1}}v_j\cdot n(x_j)\dd s|=|v_j\cdot n(x_j)||t_j-t_{j+1}|.\]
Here we used the fact that if $x,y\in \partial \Omega$ and $\partial \Omega$ is $C^2$ and $\Omega$ is bounded then $|x-y|^2\gtrsim_\Omega |(x-y)\cdot n(x)|$( see the proof in~\cite{EGKM} ). Thus
\begin{equation}\label{}
  |v_j\cdot n(x_j)|\lesssim \frac{1}{|t_j-t_{j+1}|}|\int_{t_j}^{t_{j+1}} v_jds|^2
  \lesssim |t_j-t_{j+1}||v_j|^2.
\end{equation}
Since $v_j\in \mathcal{V}^{\frac{1-\eta}{2(1+\eta)}\delta}_j$, $t_j\leq 0$, let $0\leq t\leq t'$, we have
\begin{equation}\label{}
  |v_j\cdot n(x_j)|\lesssim |t_j-t_{j+1}|\Big(\frac{1-\eta}{2(1+\eta)}\delta \Big)^{-2}.
\end{equation}
Then we prove~\eqref{eqn:claim_delta}.
\end{proof}

In consequence, when $t_k> 0$ and $t<t_*$, by~\eqref{eqn:claim_delta}, there can be at most $t_*\{[C_{\Omega}(\frac{2(1+\eta)}{(1-\eta)\delta})^3]+1\}$ numbers of $v_j\in \mathcal{V}_j^{\frac{1-\eta}{2(1+\eta)}\delta}$.
Equivalently there are at least $k-2-t_*\Big([C_{\Omega}(\frac{2(1+\eta)}{(1-\eta)\delta})^3]+1\Big)$ numbers of $v_j\in \mathcal{V}_j\backslash \mathcal{V}_j^{\frac{1-\eta}{2(1+\eta)}\delta}$.

\textbf{Step 3}

In this step we combine Step 1 and Step 2 and focus on the integration over $\prod_{j=1}^{k-1} \mathcal{V}_j$.

By~\eqref{eqn:claim_delta} in Step 2, we define
\begin{equation}\label{eqn: N}
N:=t_*\Big[C_{\Omega}\big(\frac{2(1+\eta)}{\delta(1-\eta)}\big)^3\Big]+t_*.
\end{equation}
For the sequence $\{v_1,v_2,\cdots,v_{k-1}\}$, suppose there are $p$ number of $v_j\in \mathcal{V}_j^{\frac{1-\eta}{2(1+\eta)}\delta}$ with $p\leq N$, we conclude there are at most $\left(
                                                                                  \begin{array}{c}
                                                                                    k-1 \\
                                                                                    p \\
                                                                                  \end{array}
                                                                                \right)
$ number of these sequences. Below we only consider a single sequence of them.

In order to get~\eqref{eqn: eta_paral},\eqref{eqn: eta perp}$<1$, we need to ensure the condition~\eqref{eqn: e1}. Thus we take $k=k_1(T_M,r_\perp,r_\parallel)$ and only use the decomposition $\mathcal{V}_j=\Big(\mathcal{V}_j\backslash \mathcal{V}_j^{\frac{1-\eta}{2(1+\eta)}\delta} \Big)   \cup \mathcal{V}_j^{\frac{1-\eta}{2(1+\eta)}\delta}$ for $\prod_{j=1}^{k/2} \mathcal{V}_j$. Then we only consider the half sequence $\{v_1,v_2,\cdots,v_{k/2}\}$. We derive that when $t_k>0$, there are at most $N$ number of $v_j\in \mathcal{V}_j^{\frac{1-\eta}{2(1+\eta)}\delta}$ and at least $k/2-1-N$ number of $v_j\in \mathcal{V}_j\backslash \mathcal{V}_j^{\frac{1-\eta}{2(1+\eta)}\delta}$ in $\prod_{j=1}^{k/2}\mathcal{V}_j$.

In this single half sequence $\{v_1,\cdots, v_{k/2}\}$, in order to apply Lemma \ref{Lemma: Step3}, we only want to consider the subsequence~\eqref{eqn: sequence} with $l+1<l+L\leq k/2$ and $L\geq 100\frac{1+\eta}{1-\eta}$. Thus we need to ignore those subsequence with $L<100\frac{1+\eta}{1-\eta}$. By~\eqref{eqn: sequence}, we conclude that at the end of this subsequence, it is adjacent to a $v_l\in \mathcal{V}_{l}^{\frac{1-\eta}{2(1+\eta)}\delta}$. By~\eqref{eqn: N}, we conclude
\begin{equation}\label{Conclude}
      \textit{There are at most $N$ number of subsequences~\eqref{eqn: sequence} with $L\leq 100\frac{1+\eta}{1-\eta}$}.
\end{equation}
We ignore these subsequences. Then we define the parameters for the remaining subsequence( with $L\geq 100\frac{1+\eta}{1-\eta}$ ) as:
\[M_1:= \text{the number of $v_j\in \mathcal{V}_j\backslash \mathcal{V}_j^{\frac{1-\eta}{2(1+\eta)}\delta}$ in the first subsequence starting from $v_1$},\]
\[n:= \text{the number of these subsequences}.\]
Similarly we can define $M_2,M_3,\cdots, M_n$ as the number in the second, third, $\cdots$, $n$-th subsequence. Recall that we only consider $\prod_{j=1}^{k/2} \mathcal{V}_j$, thus we have
\begin{equation}\label{eqn: Mi number}
100\frac{1+\eta}{1-\eta}\leq M_i\leq k/2,   \text{ for } 1\leq i\leq n.
\end{equation}
By~\eqref{Conclude}, we obtain
\begin{equation}\label{eqn: sum of M_i}
  k/2 \geq M_1+\cdots M_n\geq k/2-1-100\frac{1+\eta}{1-\eta}N>\frac{k}{2}-101\frac{1+\eta}{1-\eta}N.
\end{equation}
Take $M_i$ with $1\leq i\leq n$ as an example. Suppose this subsequence starts from $v_{l_i+1}$ to $v_{l_i+M_i}$, by~\eqref{eqn: Step3} in Lemma \ref{Lemma: Step3} with replacing $l$ by $l_i$ and $L$ by $M_i$, we obtain
\begin{equation}\label{eqn: M_i conclusion}
\int_{\prod_{j={l_i}}^{k-1} \mathcal{V}_j}\mathbf{1}_{\{t_k>0\}}\mathbf{1}_{\{v_{l_i+j}\in \mathcal{V}_{l_i+j}\backslash \mathcal{V}_{l_i+j}^{\frac{1-\eta}{2(1+\eta)}\delta} \text{ for } 1\leq j\leq M_i\}}\dd\Phi_{l_i,m}^{k,k-1} \leq  (3\delta)^{M_i/2}  t_*^{-(k-l_i)c} C_{T_M,T_m}^{k-l_i}\mathcal{A}_{k-l,1_i}.
\end{equation}

Since~\eqref{eqn: M_i conclusion} holds for all $1\leq i\leq n$, by Lemma \ref{Lemma: accumulate} we can draw the conclusion for the Step 3 as follows. For a single sequence $\{v_1,v_2,\cdots,v_{k-1}\}$, when there are $p$ number $v_j\in \mathcal{V}_{j}^{\frac{1-\eta}{2(1+\eta)}\delta}$, we have
\begin{align}
   &  \int_{\prod_{j=1}^{k-1} \mathcal{V}_j}   \mathbf{1}_{\{\text{$p$ number $v_j\in \mathcal{V}_{j}^{\frac{1-\eta}{2(1+\eta)}\delta}$ for a single sequence}\}}    \mathbf{1}_{\{t_k>0\}} \dd\Sigma_{k-1}^{k} \notag\\
   & \leq (3\delta)^{(M_1+\cdots+M_n)/2} t_*^{-kc} C_{T_M,T_m}^{k} \mathcal{A}_{k-1,1}.\label{eqn: Step 3 conclusion}
\end{align}

\textbf{Step 4}

Now we are ready to prove the lemma. By~\eqref{eqn: N}, we have
\begin{align}
   & \int_{\prod_{j=1}^{k-1}\mathcal{V}_j} \mathbf{1}_{\{t_k>0\}} \dd\Sigma_{k-1}^k \notag \\
   & \leq \sum_{p=1}^{N}\int_{\{\text{Exactly $p$ number of $v_j\in \mathcal{V}_j^{\frac{1-\eta}{2(1+\eta)}\delta}$ }\}} \mathbf{1}_{\{t_k>0\}} \dd\Sigma_{k-1}^k. \label{eqn: proof step3}
\end{align}

Since~\eqref{eqn: Step 3 conclusion} holds for a single sequence, we derive
\[\eqref{eqn: proof step3}\leq t_*^{-kc } C_{T_M,T_m}^{k}\sum_{p=1}^N\left(
    \begin{array}{c}
      k-1 \\
     p \\
    \end{array}
  \right)(3\delta)^{(M_1+M_2+\cdots M_n)/2} \mathcal{A}_{k-1,1}
\]
\begin{equation}\label{eqn: tk coe}
\leq t_*^{-kc}C_{T_M,T_m}^{k}N(k-1)^N(3\delta)^{k/4-101\frac{1+\eta}{1-\eta}N}\mathcal{A}_{k-1,1},
\end{equation}
where we used~\eqref{eqn: sum of M_i} in the second line.

Now we let
\[\delta=t_*^{1/3}\delta' \text{ with }\delta'\ll 1\]
such that
\[N=\frac{t_*}{t_*}\Big[C_{\Omega}\big(\frac{2(1+\eta)}{\delta'(1-\eta)}\big)^3+1\Big].\]

Using~\eqref{eqn: N} we derive
\[3\delta'= C(\Omega,\eta)N^{-1/3}.\]

Take $k=N^3$, the coefficient in~\eqref{eqn: tk coe} is bounded by
\begin{equation}\label{eqn: Finally}
t_*^{-N^3 c} C_{T_M,T_m}^{N^3}N^{3N+1} (3\delta)^{N^3/4-101\frac{1+\eta}{1-\eta}N}\leq t_*^{-N^3c} C_{T_M,T_m}^{N^3} t_*^{N^3/15} N^{4N}(3\delta')^{N^3/5},
\end{equation}
where we choosed $N=N(\eta)$ large such that $N^3/4-101\frac{1+\eta}{1-\eta}N\geq N^3/5$.

Finally we choose
\begin{equation}\label{c}
c:=\frac{1}{15}.
\end{equation}
We bound~\eqref{eqn: Finally} by
\begin{align*}
   &t_*^{-N^3 (c-\frac{1}{15})} C_{T_M,T_m}^{N^3}N^{4N}(C(\Omega,\eta)N^{-1/3})^{N^3/5}\leq e^{N^3\log(C_{T_M,T_m})} e^{4N\log N}e^{(N^3/5)\log(C(\Omega,\eta)N^{-1/3})} \\
   & =e^{4N \log N}e^{(N^3/5)(log(C(\Omega,\eta))-\frac{1}{3}\log N)}e^{N^3 \log(C_{T_M,T_m})}= e^{4N\log N-\frac{N^3}{15}(\log N-3\log C_{\Omega,\eta}-15\log C_{T_M,T_m})}\\
   &\leq e^{4N\log N-\frac{N^3}{30}\log N}\leq e^{-\frac{N^3}{15}\log N}=e^{-\frac{k}{150}\log k}\leq (\frac{1}{2})^k,
\end{align*}
where we choosed $\delta'$ to be small enough in the second line such that $N=N(\Omega,\eta,C_{T_M,T_m})$ is large enough to satisfy
\[\log N -3\log C(\Omega,\eta)-15 \log C_{T_M,T_m}\geq \frac{\log N}{2},\]
\[4N\log N-\frac{N^3}{30}\log N\leq -\frac{N^3}{50}\log N.\]
And thus we choose $k=N^3=k_2=k_2(\Omega,\eta,C_{T_M,T_m})$ and we also require $\log k>150$ in the last step. Then we get~\eqref{eqn: 1/2 decay}.

Therefore, by the condition~\eqref{eqn: e1}, we choose $k=k_0=\max\{k_1,k_2\}$. By the definition of $\eta$~\eqref{eqn: eta} with~\eqref{eqn: eta_paral} and~\eqref{eqn: eta perp}, we obtain $\eta=\eta(T_M,\mathcal{C},r_\perp,r_\parallel,\e_2)$. Thus by~\eqref{eqn: e2 dependence} and~\eqref{eqn: k_1 dependence}, we conclude the $k_0$ we choose here does not depend on $t$ and only depends on the parameter in~\eqref{eqn: k_0 dependence}. We conclude the lemma.

\end{proof}

\begin{proof}[\textbf{Proof of Proposition \ref{proposition: boundedness}}]
First we take
\begin{equation}\label{eqn: first condition for tinf}
t_{\infty}\leq t_*,
\end{equation}
with $t_*$ defined in~\eqref{t* dependence}. Then we let $k=k_0$ with $k_0$ defined in~\eqref{eqn: k_0 dependence} so that we can apply Lemma \ref{lemma: t^k} and Lemma \ref{lemma: boundedness}. Define the constant in~\eqref{eqn: fm is bounded} as
\begin{equation}\label{eqn: Cinfty}
  C_\infty=8t_*^{-k_0/15}(C_{T_M,T_m})^{k_0}.
\end{equation}

We mainly use the formula given in Lemma \ref{lemma: the tracjectory formula for f^(m+1)}. By~\eqref{Gm} we have
\begin{equation}\label{bound for G}
|\mathcal{G}^m(s)| \leq \Vert w_\theta f^m\Vert_\infty^2+\Vert e^{-\lambda\langle v\rangle s}\alpha\partial f^m\Vert_\infty [ \sup_m\Vert w_\theta f^m\Vert_\infty+1]+\Vert w_\theta f^m\Vert_\infty e^{-\lambda \langle v\rangle s}\alpha(x,v)\int_{\mathbb{R}^3} \mathbf{k}_\varrho(v,u) |\partial f^m(X^1(s),u)| \dd u,
\end{equation}
where we used~\eqref{nabla gamma gain}.

We consider two cases.
\begin{description}
\item[Case1] $t_1\leq 0$,
\end{description}
By~\eqref{eqn: Duhamal principle for case1} and~\eqref{bound for G}, for some polynomial $P$ we have
\begin{align}
&|e^{-\lambda\langle v\rangle t}\alpha(x,v) \partial f^{m+1}(t,x,v)| \notag \\
 & \leq |\alpha \partial f_0(X^1(0),v)| +t \Vert e^{-\lambda\langle v\rangle t}
\alpha\partial f^{m+1}\Vert_\infty+tP(\sup_m\Vert w_\theta f^m\Vert_\infty)\label{eqn: first term} \\
   &    +P(\sup_m\Vert w_\theta f^m\Vert_\infty)\alpha(x,v)\int_0^t \int_{\mathbb{R}^3} e^{-|v|(t-s)}  \mathbf{k}_\varrho(v,u)e^{-\lambda [\langle v\rangle -\langle u\rangle]s} \frac{\Vert e^{-\lambda\langle v\rangle s}\alpha\partial f^m(s)\Vert_\infty  }{\alpha(x-(t-s)v,u)}  \dd u \dd s.\label{eqn: second term}
\end{align}
Since $s\leq t\ll 1$, $e^{-\lambda [\langle v\rangle -\langle u\rangle]s}\lesssim 1+e^{\varrho|v-u|^2/2}$. And thus
\[\mathbf{k}_\varrho e^{-\lambda [\langle v\rangle -\langle u\rangle]s} \lesssim   \mathbf{k}_\varrho + \mathbf{k}_{\varrho/2}.\]
Then applying Lemma \ref{Lemma: NLN} we have
\begin{align*}
 \eqref{eqn: second term}  & \leq tP(\sup_m\Vert w_\theta f^m\Vert_\infty)\sup_{s\leq t}\Vert e^{-\lambda\langle v\rangle s}\alpha\partial f^m(s)\Vert_\infty.
\end{align*}

Collecting~\eqref{eqn: first term} and~\eqref{eqn: second term} we obtain
\begin{equation}\label{eqn: hm+1 bounded case 1}
\begin{split}
 \Vert e^{-\lambda\langle v\rangle t}\alpha\partial f^{m+1}(t)\mathbf{1}_{\{t_1\leq 0\}}\Vert_\infty   & \leq  t\sup_{0\leq s\leq t}\Vert e^{-\lambda\langle v\rangle s}\alpha\partial f^{m+1}(s)\Vert_\infty \\
     & + tP(\sup_m\Vert w_\theta f^m\Vert_\infty) +tP(\sup_m\Vert w_\theta f^m\Vert_\infty)\sup_{0\leq s\leq t}\Vert e^{-\lambda\langle v\rangle s}\alpha\partial f^{m}(s)\Vert_\infty.
\end{split}
\end{equation}
Since~\eqref{eqn: hm+1 bounded case 1} holds for all $t<t_\infty$, we derive
\[\sup_s\Vert e^{-\lambda\langle v\rangle s}\alpha\partial f^{m+1}(s)\mathbf{1}_{\{t_1\leq 0\}}\Vert_\infty \leq  \text{ R.H.S of }\eqref{eqn: hm+1 bounded case 1}.\]
And thus with $t\ll 1$,
\begin{equation}\label{hm+1 bounded case 1}
\sup_{s\leq t} \Vert e^{-\lambda\langle v\rangle s}\alpha\partial f^{m+1}(s)\mathbf{1}_{\{t_1\leq 0\}}\Vert_\infty \leq  2tP(\sup_m\Vert w_\theta f^m\Vert_\infty) +t[1+P(\sup_m\Vert w_\theta f^m\Vert_\infty)]\sup_{s\leq t}\Vert e^{-\lambda\langle v\rangle s}\alpha\partial f^{m}(s)\Vert_\infty.
\end{equation}

\begin{description}
\item[Case2] $t_1\geq 0$,
\end{description}
We consider~\eqref{eqn: Duhamel principle for case 2} in Lemma \ref{lemma: the tracjectory formula for f^(m+1)}. For the first line, by~\eqref{bound for G} and the same computation as~\eqref{eqn: second term} we obtain
\begin{equation}\label{eqn: first line bounded}
\int_{t_1}^t e^{-(t-s)|v|} \mathcal{G}^m(s)\dd s  \leq  t\sup_{0\leq s\leq t}\Vert e^{-\lambda\langle v\rangle s}\alpha\partial f^{m+1}(s)\Vert_\infty+ tP(\sup_m\Vert w_\theta f^m\Vert_\infty) +t\sup_s\Vert e^{-\lambda\langle v\rangle s}\alpha\partial f^{m}(s)\Vert_\infty.
\end{equation}
For the second line of~\eqref{eqn: Duhamel principle for case 2}, we bound it by
\begin{equation}\label{eqn: extra term to cancel}
\exp\bigg(\big[\frac{1}{4T_M}-\frac{1}{2T_w(x_1)}\big]|v|^2\bigg)\int_{\prod_{j=1}^{k_0-1}\mathcal{V}_j}H.
\end{equation}
Now we focus on $\int_{\prod_{j=1}^{k_0-1}\mathcal{V}_j}H$. We compute $H$ term by term using~\eqref{eqn: formula for H}.

First we compute the first line of~\eqref{eqn: formula for H}. By Lemma~\ref{lemma: boundedness} with $p=1$, for every $1\leq l\leq k_0-1$, we have
\[\int_{\prod_{j=1}^{k_0-1}\mathcal{V}_j}  \mathbf{1}_{\{t_{l+1}\leq 0<t_l\}} |\alpha\partial f_0\big(X^{m-l}(0),V^{m-l}(0)\big)|  \dd\Sigma_{l}^{k_0}\leq \Vert \alpha\partial f_0\Vert_\infty  \int_{\prod_{j=1}^{k_0-1}\mathcal{V}_j}  \mathbf{1}_{\{t_{l+1}\leq 0<t_l\}}   \dd\Sigma_{l}^{k_0}\]
\begin{equation}\label{eqn: l term}
     \leq  t_*^{-l/15} C_{T_M,T_m}^l\Vert \alpha\partial f_0\Vert_{\infty}\exp\bigg(\frac{(T_{l,1}-T_w(x_1))(1-r_{min})}{2T_w(x_1)[T_{l,1}(1-r_{min})+r_{min} T_w(x_1)]}|v|^2+\mathcal{C}_{l}t_*^{1/15}|v|^2\bigg).
\end{equation}
In regard to~\eqref{eqn: extra term to cancel} we have
\begin{align*}
   & \exp\bigg(\big[\frac{1}{4T_M-2T_w(x_1)}\big]|v|^2\bigg)\times ~\eqref{eqn: l term} \\
   & =t_*^{-l/15} C_{T_M,T_m}^l\Vert \alpha\partial f_0\Vert_{\infty}\exp\bigg(\Big[\frac{-1}{2\big(T_w(x_1)r_{min}+T_{l,1}(1-r_{min})\big)}+\frac{1}{4T_M}\Big]|v|^2+\mathcal{C}_{l}t_*^{1/15}|v|^2\bigg).
\end{align*}
Using the definition~\eqref{eqn: definition of T_p} we have $T_w(x_1)<2T_M$ and $T_{l,1}<2T_M$. Then we require
\begin{equation}\label{eqn: third condition for tinfty}
t_*=t_*(T_M,k_0,\mathcal{C})
\end{equation}
to be small enough such that the coefficient for $|v|^2$ is
\[\frac{-1}{2\big(T_w(x_1)r_{min}+T_{l,1}(1-r_{min})\big)}+\frac{1}{4T_M}+\mathcal{C}_{l}t_*^{1/15}\]
\begin{equation}\label{eqn: less than 0}
\leq \frac{-1}{2\big(T_Mr_{min}+T_{l,1}(1-r_{min})\big)}+\frac{1}{4T_M}+\mathcal{C}_{k_0}t_*^{1/15}\leq 0.
\end{equation}
Note that the condition~\eqref{eqn: third condition for tinfty} is consistent with~\eqref{t* dependence}.

Since~\eqref{eqn: l term} holds for all $1\leq l\leq k_0-1$, by~\eqref{eqn: less than 0} the contribution of the first line of~\eqref{eqn: formula for H} in ~\eqref{eqn: extra term to cancel} is bounded by
\begin{equation}\label{eqn: first term bounded}
t_*^{-k_0/15}C_{T_M,T_m}^{k_0}\Vert \alpha\partial f_0\Vert_{\infty}.
\end{equation}

Then we compute the second line of~\eqref{eqn: formula for H}:
\begin{align}
   & \int_{\max\{0,t_{l+1}\}}^{t_l}\int_{\prod_{j=1}^{k_0-1}\mathcal{V}_{j}} e^{-(t_l-s)|v_l|} |\mathcal{G}^{m-l}(s)|\dd\Sigma_{l}^{k_0}\dd s\notag \\
   & \leq tP(\sup_m\Vert w_\theta f^m\Vert_\infty)  \sup_{i\leq m} \Vert e^{-\lambda\langle v\rangle t}\alpha\partial f^i\Vert_\infty \int_{\prod_{j=1}^{k_0-1}\mathcal{V}_j} \dd\Sigma_{l}^{k_0}\notag\\
   &\leq t t_*^{-k_0/15}C_{T_M,T_m}^{k_0}   P(\sup_m\Vert w_\theta f^m\Vert_\infty)  \sup_{i\leq m} \Vert e^{-\lambda\langle v\rangle t}\alpha\partial f^i\Vert_\infty  \notag \\
   &\times \exp\bigg(\frac{(T_{l,1}-T_w(x_1))(1-r_{min})}{2T_w(x_1)[T_{l,1}(1-r_{min})+r_{min} T_w(x_1)]}|v|^2+\mathcal{C}_{l}t_*^{1/15}|v|^2\bigg) \notag\\
   &\leq     \frac{\sup_{i\leq m}\Vert e^{-\lambda\langle v\rangle t}\alpha\partial f^i\Vert_\infty    }{5k_0}\exp\bigg(\frac{(T_{l,1}-T_w(x_1))(1-r_{min})}{2T_w(x_1)[T_{l,1}(1-r_{min})+r_{min} T_w(x_1)]}|v|^2+\mathcal{C}_{l}t_*^{1/15}|v|^2\bigg).\label{eqn: last line second}
\end{align}
In the second line we applied the same computation as~\eqref{eqn: second term} to the $s$-integration. In the third line we applied~\eqref{eqn: boundedness for l-p+1 fold integration} In the last line we applied Lemma \ref{lemma: boundedness} and take $t_\infty=t_\infty(t_*,k_0,C_{T_M,T_m},P( \sup_m\Vert w_\theta f^m\Vert_\infty))$ to be small enough such that for $t<t_\infty$,
\begin{equation}\label{condition for tinfty}
t t_*^{-k_0/15}C_{T_M,T_m}^{k_0}P( \sup_m\Vert w_\theta f^m\Vert_\infty) \leq \frac{1}{5k_0}.
\end{equation}

In regard to~\eqref{eqn: extra term to cancel}, by~\eqref{eqn: less than 0} we obtain
\[\exp\bigg(\big[\frac{1}{4T_M}-\frac{1}{2T_w(x_1)}\big]|v|^2\bigg)\times ~\eqref{eqn: last line second}\leq \frac{1}{5k_0} \sup_{i\leq m}\Vert e^{-\lambda\langle v\rangle t}\alpha\partial f^i\Vert_\infty    .\]

Since~\eqref{eqn: last line second} holds for all $1\leq l\leq k_0-1$, the contribution of the second line of~\eqref{eqn: formula for H} in~\eqref{eqn: extra term to cancel} is bounded by
\begin{equation}\label{eqn: second term bounded}
  \frac{k_0-1}{5k_0}\sup_{i\leq m}\Vert e^{-\lambda\langle v\rangle t}\alpha\partial f^i\Vert_\infty     \leq \frac{1}{5}\sup_{i\leq m} \Vert e^{-\lambda\langle v\rangle t}\alpha\partial f^i\Vert_\infty.
\end{equation}

Then we compute the third line of~\eqref{eqn: formula for H}. Directly applying Lemma \ref{lemma: boundedness} we obtain
\begin{align}
   &\int_{\prod_{j=1}^{k_0-1}\mathcal{V}_j} \mathbf{1}_{\{t_{l+1}<0\}} P(\sup_m \Vert w_\theta f^m\Vert_\infty) \dd \Sigma_l^{k_0}  \notag\\
   &  \leq  t_*^{-l/15} C_{T_M,T_m}^lP(\sup_m \Vert w_\theta f^m\Vert_\infty)\exp\bigg(\frac{(T_{l,1}-T_w(x_1))(1-r_{min})}{2T_w(x_1)[T_{l,1}(1-r_{min})+r_{min} T_w(x_1)]}|v|^2+\mathcal{C}_{l}t_*^{1/15}|v|^2\bigg).\label{third term bdd}
\end{align}

In regard to~\eqref{eqn: extra term to cancel}, by~\eqref{eqn: less than 0} we obtain
\begin{equation}\label{third term bound}
\exp\bigg(\big[\frac{1}{4T_M}-\frac{1}{2T_w(x_1)}\big]|v|^2\bigg)\times \eqref{third term bdd}\leq  t_*^{-l/15} C_{T_M,T_m}^lP(\sup_m \Vert w_\theta f^m\Vert_\infty).
\end{equation}

Last we compute the fourth term of~\eqref{eqn: formula for H}. By Lemma \ref{lemma: t^k} and the assumption~\eqref{eqn: fm is bounded} we obtain
\begin{align}
 & \int_{\prod_{j=1}^{k_0-1}\mathcal{V}_j}  \mathbf{1}_{\{0<t_{k_0}\}} \Vert e^{-\lambda\langle v\rangle t_{k_0}}\alpha\partial f^{m-k_0}(t_{k_0})\Vert_\infty      \dd\Sigma_{k_0-1}^{k_0} \notag\\
  & \leq \Vert e^{-\lambda\langle v\rangle t}\alpha\partial f\Vert_\infty      \int_{\prod_{j=1}^{k_0-1}\mathcal{V}_j}  \mathbf{1}_{\{0<t_{k_0}\}}   \dd\Sigma_{k_0-1}^{k_0}\notag \\
  &\leq       (\frac{1}{2})^{k_0}\sup_{i\leq m}\Vert e^{-\lambda\langle v\rangle t}\alpha\partial f^i\Vert_\infty    \exp\bigg(\frac{(T_{l,1}-T_w(x_1))(1-r_{min})}{2T_w(x_1)[T_{l,1}(1-r_{min})+r_{min} T_w(x_1)]}|v|^2+\mathcal{C}_{l}t_*^{1/15}|v|^2\bigg).\label{eqn: third term bounded in H}
\end{align}

In regard to~\eqref{eqn: extra term to cancel}, by~\eqref{eqn: less than 0} we have
\[\exp\bigg(\big[\frac{1}{4T_M}-\frac{1}{2T_w(x_1)}\big]|v|^2\bigg)\times ~\eqref{eqn: third term bounded in H}\leq (\frac{1}{2})^{k_0}\Vert e^{-\lambda\langle v\rangle t}\alpha\partial f\Vert_\infty    .\]
Thus the contribution of the third line of~\eqref{eqn: formula for H} in~\eqref{eqn: extra term to cancel} is bounded by
\begin{equation}\label{eqn: third term bounded}
(\frac{1}{2})^{k_0}\sup_{i\leq m}\Vert e^{-\lambda\langle v\rangle t}\alpha\partial f^i\Vert_\infty.
\end{equation}

Collecting~\eqref{eqn: first term bounded}~\eqref{eqn: second term bounded}~\eqref{third term bound} and~\eqref{eqn: third term bounded} we conclude that the second line of~\eqref{eqn: Duhamel principle for case 2} is bounded by
\begin{equation}\label{eqn: second line bounded}
\eqref{eqn: extra term to cancel}\leq [(\frac{1}{2})^{k_0}+\frac{1}{5}]\sup_{i\leq m}\Vert e^{-\lambda\langle v\rangle t}\alpha\partial f^i\Vert_\infty + t_*^{-k_0/15}C_{T_M,T_m}^{k_0}\big[\Vert \alpha\partial f_0\Vert_{\infty}+P(\sup_m \Vert w_\theta f^m\Vert_\infty)\big].
\end{equation}
Adding~\eqref{eqn: second line bounded} to~\eqref{eqn: first line bounded} we use~\eqref{eqn: Duhamel principle for case 2} and $t\ll 1$ to derive
\begin{equation}\label{eqn: hm+1 bounded case 2}
\begin{split}
&   \Vert e^{-\lambda\langle v\rangle t}\alpha  \partial f^{m+1}(t,x,v)\mathbf{1}_{\{t_{1}\geq 0\}}\Vert_\infty \\
& \leq t\sup_{0\leq s\leq t}\Vert e^{-\lambda\langle v\rangle s}\alpha\partial f^{m+1}(s)\Vert_\infty   \\
     &+ [\frac{1}{4}C_\infty+2t_*^{-k_0/15}C_{T_M,T_m}^{k_0}]\big[\Vert \alpha\partial f_0\Vert_{\infty}+P(\sup_m \Vert w_\theta f^m\Vert_\infty)\big]\\
     & \leq t\sup_{0\leq s\leq t}\Vert e^{-\lambda\langle v\rangle s}\alpha\partial f^{m+1}(s)\Vert_\infty+4t_*^{-k_0/15}C_{T_M,T_m}^{k_0}\big[\Vert \alpha\partial f_0\Vert_{\infty}+P(\sup_m \Vert w_\theta f^m\Vert_\infty)\big],
\end{split}
\end{equation}
where we have used the definition of $C_\infty$~\eqref{eqn: Cinfty} in the last line.

Since~\eqref{eqn: hm+1 bounded case 2} holds for all $t<t_\infty$, we derive that
\begin{align*}
   & \sup_{s\leq t}\Vert e^{-\lambda\langle v\rangle s}\alpha  \partial f^{m+1}(s,x,v)\mathbf{1}_{\{t_{1}\geq 0\}}\Vert_\infty  \leq \text{ Last line of }\eqref{eqn: hm+1 bounded case 2}.
\end{align*}

Therefore, with $t\ll 1$ we conclude
\begin{equation}\label{hm+1 bounded case 2}
\sup_{s\leq t}\Vert e^{-\lambda\langle v\rangle s}\alpha  \partial f^{m+1}(s,x,v)\mathbf{1}_{\{t_{1}\geq 0\}}\Vert_\infty \leq 8t_*^{-k_0/15}C_{T_M,T_m}^{k_0}\big[\Vert \alpha\partial f_0\Vert_{\infty}+P(\sup_m \Vert w_\theta f^m\Vert_\infty)\big].
\end{equation}

Combining~\eqref{hm+1 bounded case 1} and~\eqref{hm+1 bounded case 2} we derive~\eqref{eqn: L_infty bound for f^m+1}.

Last we focus the parameters for $t_\infty$ in~\eqref{eqn: t_1}. In the proof the constraint for $t_\infty$ comes from~\eqref{condition for tinfty}. Thus from the definition of $k_0$ in~\eqref{eqn: k_0 dependence}, definition of $C_{T_M,T_m}$ in~\eqref{eqn: 1 one} and definition of $t_*$ in~\eqref{t* dependence}
\[t_\infty=t_\infty(t_*,k_0,C_{T_M,T_m},P( \sup_m\Vert w_\theta f^m\Vert_\infty))=t_\infty(T_M,T_m,\Omega,r_\perp,r_\parallel,\sup_m\Vert w_\theta f^m\Vert_\infty).\]
Thus we derive~\eqref{eqn: t_1}.

\end{proof}

\begin{proof}[\textbf{Proof of Theorem~\ref{Thm: dynamic C1}}]
The uniform-in-m bound~\eqref{eqn: theta'} follows from Proposition \ref{proposition: boundedness}. Then we follow the same argument to $e^{-\lambda \langle v\rangle t} \alpha [\partial f^{m+1}-\partial f^m]$ and conclude that $e^{-\lambda \langle v\rangle t} \alpha \partial f^m$ is a Cauchy sequence in $L^\infty$.
Then we pass the limit and conclude Theorem~\ref{Thm: dynamic C1}.

\end{proof}

\section{Weighted $C^1$-estimate of the stationary Boltzmann equation}
In this section we prove the weighted $C^1$-estimate of the stationary Boltzmann equation~\eqref{eqn: Steady Boltzmann}. In particular, we will prove Theorem \ref{thm: C1 steady}.

First we give the boundary condition for $f_s$ in the following lemma.

\begin{lemma}\label{Lemma: bc for fs}(Lemma 9 in \cite{chen})\\
The boundary condition for $f_s$ defined in~\eqref{equation of f_s} is given by
\begin{align*}
  f_s(x,v)|_{\gamma_-} &=r_s+e^{[\frac{1}{4T_0}-\frac{1}{2T_w(x)}]|v|^2}\int_{n(x)\cdot u>0} f_s(x,u)e^{-[\frac{1}{4T_0}-\frac{1}{2T_w(x)}]|u|^2}\dd \sigma(u,v).
\end{align*}
Here the remainder term $r_s$ is given by
\begin{equation}\label{rs}
r_s = \frac{\mu_{x,r_\parallel,r_\perp}-\mu_0 }{\sqrt{\mu_0}},
\end{equation}
with
\begin{equation}\label{mu_xr}
\begin{split}
   & \mu_{x,r_\parallel,r_\perp}=\frac{1}{2\pi [T_0(1-r_\parallel)^2+T_w(x)r_\parallel(2-r_\parallel)]}e^{-\frac{|v_\parallel|^2}{2[T_0(1-r_\parallel)^2+T_w(x)r_\parallel(2-r_\parallel)]}}
 \\
    & \times \frac{1}{T_0(1-r_\perp)+T_w(x)r_\perp}e^{-\frac{|v_\perp|^2}{2[T_0(1-r_\perp)+T_w(x)r_\perp]}}.
\end{split}
\end{equation}

\end{lemma}

As mentioned in the introduction, when we perform the integration by parts, polynomial terms appear in the integration. In the next lemma we will bound all the possible integration related to the C-L boundary.

\begin{lemma}\label{Lemma: integrable}
Denote
\begin{equation}\label{front term}
\mathcal{A}:=\frac{2}{r_\perp r_\parallel(2-r_\parallel)\pi} \frac{1}{(2T_w(x_1))^2} e^{[\frac{1}{4T_0}-\frac{1}{2T_w(x_1)}]|v|^2}.
\end{equation}
For~\eqref{back term} given by
\begin{equation}\label{back term}
\begin{split}
    &e^{-[\frac{1}{4T_0}-\frac{1}{2T_w(x_1)}]|v_1|^2}(n(x_1)\cdot v_1)I_0\bigg(\frac{(1-r_\perp)^{1/2}v_{1,\perp}v_{\perp}}{r_\perp T_w(x_1)} \bigg) \\
    &\times \exp\bigg(-\frac{1}{2T_w(x_1)}\Big[ \frac{|v_{1,\perp}|^2+(1-r_\perp)|v_\perp|^2}{r_\perp}+\frac{|v_{1,\parallel}-(1-r_\parallel)v_\parallel|^2}{r_\parallel(2-r_\parallel)} \Big]\bigg),
\end{split}
\end{equation}
under the condition~\eqref{eqn: small pert condition}, we have
\begin{equation}\label{integrate back}
\mathcal{A}\times \int_{n(x_1)\cdot v_1>0}      [1+|v_1|^2+|v|^2]     \eqref{back term} \lesssim 1,
\end{equation}

\begin{equation}\label{integrate back / alpha}
\mathcal{A}\times \int_{n(x_1)\cdot v_1>0}  \frac{1}{\alpha(x_1,v_1)}       \eqref{back term} \lesssim 1,
\end{equation}

\begin{equation}\label{integrate nablav back}
\mathcal{A}\times \int_{n(x_1)\cdot v_1>0}        [1+|v_1|]  \nabla_{v_1} [\eqref{back term}] \lesssim 1.
\end{equation}

For $x_1 = \eta_{p^1}(\mathbf{x}_{p^1}^1)$ and $i=1,2$
\begin{equation}\label{integrate partial_x A}
 \partial_{\mathbf{x}_{p^1,i}^1} \mathcal{A}\times \int_{n(x_1)\cdot v_1>0}          \eqref{back term}  \lesssim 1.
\end{equation}

\end{lemma}

\begin{remark}
The condition~\eqref{eqn: small pert condition} is not necessary in this lemma. Since we will only use this lemma for the stationary problem we impose such condition to simplify the proof.
\end{remark}

\begin{proof}
From condition~\eqref{eqn: small pert condition}, we have $|T_w(x_1)-T_0|,|1-r_\perp|,|1-r_\parallel|\ll 1$. Then for some $\e\ll 1$,
\[|\frac{1}{r_\perp}-1|, \ |1-\frac{1}{r_\parallel(2-r_\parallel)}|=|1-\frac{1}{1-(1-r_\parallel)^2}|\lesssim O(\e),  \]
\[|\frac{(1-r_\perp)}{r_\perp}|, \ |\frac{(1-r_\perp)^{1/2}}{r_\perp}| , \  |\frac{2(1-r_\parallel)}{r_\parallel(2-r_\parallel)}|, \  |\frac{(1-r_\parallel)^2}{r_\parallel(2-r_\parallel)}| \lesssim O(\e),\]
\[|\frac{1}{T_w(x_1)}-\frac{1}{T_0}|\lesssim O(\e).\]
Hence
\begin{align}
\eqref{back term}   &\lesssim |n(x_1)\cdot v_1| e^{-[\frac{1}{4T_0}-\frac{1}{2T_0}]|v_1|^2}e^{\frac{O(\e)}{T_0}|v_1|^2}e ^{-\frac{1}{2T_0}|v_{1,\perp}|^2 }e^{\frac{O(\e)}{T_0}|v_{1,\perp}|^2}e^{\frac{O(\e)}{T_0}|v_\perp|^2} \times e^{\frac{O(\e)}{T_0}v_{1,\perp}v_\perp}  \notag\\
   & \times  e^{-\frac{1}{2T_0}|v_{1,\parallel}|^2} e^{\frac{O(\e)}{T_0}|v_{1,\parallel}v_\parallel|} e^{\frac{O(\e)}{T_0}|v_\parallel|^2} \label{first bound of back} \\
   &  \lesssim  |v_1|e^{-\frac{1}{4T_0}|v_1|^2}    e^{\frac{O(\e)}{T_0}|v|^2} e^{\frac{O(\e)}{T_0}|v_1|^2}, \notag
\end{align}
in the last line we have used $|ab|\lesssim |a|^2+|b|^2$, $|v_\parallel|^2+|v_\perp|^2 = |v|^2$.

Thus using $\e \ll 1$ and~\eqref{front term} we have
\begin{align}
  & \mathcal{A}\int_{n(x_1)\cdot v_1>0}[1+|v|^2+|v_1|^2]\eqref{back term} \\
   & \lesssim \frac{1}{T_0^2}e^{[\frac{1}{4T_0}-\frac{1}{2T_0}]|v|^2}e^{\frac{O(\e)}{T_0}|v|^2}[1+|v|^2]\int_{n(x_1)\cdot v_1>0} |v_1|[1+|v_1|^2] e^{-\frac{1}{4T_0}|v_1|^2}    e^{\frac{O(\e)}{T_0}|v|^2} e^{\frac{O(\e)}{T_0}|v_1|^2}\notag \\
   & \lesssim  \frac{1}{T_0^2}   e^{-\frac{1-O(\e)}{4T_0}|v|^2}[1+|v|^2] \int_{n(x_1)\cdot v_1>0} |v_1|[1+|v_1|^2]e^{-\frac{1-O(\e)}{4T_0}|v_1|^2} \lesssim 1,   \label{bdd of A back}
\end{align}
where we used $T_0\gtrsim 1$. Then we conclude~\eqref{integrate back}.

\eqref{integrate back / alpha} follows from~\eqref{first bound of back}, where $\frac{1}{\alpha(x_1,v_1)}$ is cancelled by $|n(x_1)\cdot v_1|$, and the rest computation is the same.

Then we prove~\eqref{integrate nablav back}. From~\eqref{back term}, taking the $v_1$ derivative we will have extra term
\[[\frac{1}{4T_0}-\frac{1}{2T_w(x_1)}]|v_1|, \quad \frac{|v_{1,\perp}|}{T_w(x_1)r_\perp}, \quad \frac{|v_{1,\parallel}-(1-r_\parallel)v_\parallel|}{T_w(x_1)r_\parallel(2-r_\parallel)},\]
and from~\eqref{I0},
\begin{align*}
\nabla_{v_1}I_0\bigg(\frac{(1-r_\perp)^{1/2}v_{1,\perp}v_{\perp}}{r_\perp T_w(x_1)}\bigg)   & =\pi^{-1}\int_0^{\pi}e^{\frac{(1-r_\perp)^{1/2}v_{1,\perp}v_{\perp}}{r_\perp T_w(x_1)}\cos \phi} \nabla_{v_1} \Big(\frac{(1-r_\perp)^{1/2}v_{1,\perp}v_{\perp}}{r_\perp T_w(x_1)}\Big)\cos \phi \dd \phi \\
   & \lesssim \nabla_{v_1} \Big(\frac{(1-r_\perp)^{1/2}v_{1,\perp}v_{\perp}}{r_\perp T_w(x_1)}\Big) I_0\bigg(\frac{(1-r_\perp)^{1/2}v_{1,\perp}v_{\perp}}{r_\perp T_w(x_1)}\bigg),
\end{align*}
the extra term is
\[\frac{(1-r_\perp)^{1/2}|v_\perp|}{T_w(r_1)r_\perp}.\]
Thus all the extra term can be bounded as
\[\frac{|v|+|v_1|}{T_0}\lesssim  |v|+|v_1|\lesssim [1+|v|^2+|v_1|^2].\]
This upper bound is already included in~\eqref{integrate back}. Thus we conclude~\eqref{integrate nablav back}.

Last we prove~\eqref{integrate partial_x A}. From~\eqref{front term} taking $\partial_{\mathbf{x}_{p^1,i}^1}$ derivative we have extra term
\[\frac{\partial_i \eta_{p^1}(\mathbf{x}_{p^1}^1)}{T_w(x_1)^3}, \quad \frac{\partial_i \eta_{p^1}(\mathbf{x}_{p^1}^1)}{T_w(x_1)^2}|v|^2.\]

From~\eqref{back term}, taking $\partial_{\mathbf{x}_{p^1,i}^1}$ derivative we have extra term
\[\frac{\partial_i \eta_{p^1}(\mathbf{x}_{p^1}^1)}{T^2_w(x_1)}|v_1|^2,\quad \frac{\partial_i \eta_{p^1}(\mathbf{x}_{p^1}^1)}{T^2_w(x_1)}.  \]
The extra term are bounded by
\[\Vert \eta\Vert_{C^1}[\frac{1}{T_0^3}+\frac{1}{T_0^2}][1+|v|^2],\]
which is included in~\eqref{integrate back}. Thus we conclude~\eqref{integrate partial_x A}.

\end{proof}

Then we start to prove Theorem \ref{thm: C1 steady}. The main idea is to express the characteristic of~\eqref{equation of f_s} by using the Duhamel's principle:
\begin{align}
  f_s(x,v) &=\mathbf{1}_{t\geq \tb} e^{-\nu(v)\tb}f_s(\xb,v) \notag \\
   &+\mathbf{1}_{t<\tb}e^{-\nu(v)t}f_s(x-tv,v)  \label{trajectory} \\
   & +\int_{\max\{0,t-\tb\}}^t e^{-\nu(v)(t-s)}h(x-(t-s)v,v)\dd s,\notag
\end{align}
where $h = K(f_s)+\Gamma(f_s,f_s)$.

Here in order to distinguish between Euclidean coordinate and the backward cycles, we denote
\begin{equation}\label{[x]}
x=([x]_1,[x]_2,[x]_3).
\end{equation}

Thus
\[\nabla_x = (\partial_{[x]_1}, \partial_{[x]_2},\partial_{[x]_3}).\]

We take the spatial derivative to~\eqref{trajectory} to have
\begin{align}
  \partial_{[x]_j}f_s(x,v) & =\mathbf{1}_{t\geq \tb}e^{-\nu(v)\tb}\partial_{[x]_j}[f_s(\xb,v)] \label{f1} \\
   & -\mathbf{1}_{t\geq \tb}\nu(v)\partial_{[x]_j}\tb(x,v)e^{-\nu(v)\tb}f_s(\xb,v) \label{f2}\\
   & +\mathbf{1}_{t<\tb}e^{-\nu(v)t}\partial_{[x]_j}[f_s(x-tv,v)] \label{f3}\\
   & +\int^t_{\max\{0,t-\tb\}} e^{-\nu(v)(t-s)}\partial_{[x]_j}[h(x-(t-s)v,v)]\dd s \label{f4}\\
   & -\mathbf{1}_{t\geq \tb}\partial_{[x]_j}\tb e^{-\nu(v)\tb}h(x-\tb v,v).\label{f5}
\end{align}

First we give an estimate for~\eqref{f1}-\eqref{f5}.
\begin{lemma}\label{Lemma: prior estimate}
For $h=Kf_s+\Gamma(f_s,f_s)$, we can express $\partial_{[x]_j}f_s(x,v)$ as
\begin{align}
\partial_{[x]_j}f_s(x,v) & = \frac{O(1)[\Vert w_\vartheta f_s\Vert_\infty+\Vert w_\vartheta f_s\Vert_\infty^2]+o(1)\Vert \alpha\nabla_x f_s\Vert_\infty }{\alpha(x,v)} \label{cons}\\
   &+\int^t_{\max\{0,t-\tb\}} e^{-\nu(v)(t-s)}\int_{\mathbb{R}^3} \mathbf{k}(v,u) \partial_{[x]_j}f(x-(t-s)v,u)\dd u \dd s \label{collision} \\
   & +\sum_{i=1,2}\frac{\partial \mathbf{x}^1_{p^1,i}}{\partial [x]_j}\partial_{\mathbf{x}^1_{p^1,i}}[f_s(\eta_{p^1}(\mathbf{x}_{p^1}^1),v)]    . \label{bdr term}
\end{align}

\end{lemma}

\begin{proof}
We estimate every term in~\eqref{f1}-\eqref{f5}.

From the chain rule and the definition of $\eta_{p^1}(\mathbf{x}_{p^1}^1)=x_1$ in~\eqref{xkvk}, the contribution of \eqref{f1} is \eqref{bdr term}.

For~\eqref{f2}, since $\nu(v)\lesssim w_{\vartheta}(v)$, we apply~\eqref{nabla_tbxb} to get
\begin{align*}
 \eqref{f2} &=\frac{O(1)\Vert \nu(v) f_s(\xb,v)\Vert}{\alpha(x,v)}   \lesssim  \frac{O(1)\Vert w_{\vartheta} f_s\Vert}{\alpha(x,v)}.
\end{align*}
Such contribution is included in~\eqref{cons}.

For~\eqref{f3}, using $t\gg 1$ we get
\begin{align*}
  \eqref{f3} &=o(1)\partial_{[x]_j}f_s(x-tv,v)=o(1)\frac{\Vert \alpha\nabla_x f_s\Vert_\infty}{\alpha(x,v)}.
\end{align*}

For~\eqref{f4}, we first consider the contribution of $h=K(f_s)$, which reads
\begin{align*}
 \eqref{f4}_{h=K}  &=\int^t_{\max\{0,t-\tb\}}e^{-\nu(v)(t-s)} \partial_{[x]_j}[\int_{\mathbb{R}^3} \mathbf{k}(v,u)f_s(x-(t-s)v,u)  ]   \dd s  \\
   & =\int^t_{\max\{0,t-\tb\}}e^{-\nu(v)(t-s)}\int_{\mathbb{R}^3} \mathbf{k}(v,u) \partial_{[x]_j} f_s(x-(t-s)v,u)     \dd s.
\end{align*}
Such contribution is included in~\eqref{collision}.

Then we consider the contribution of $h=\Gamma(f_s,f_s)$. By~\eqref{nabla gamma} we have
\begin{align*}
\eqref{f4}_{h=\Gamma}   &\lesssim \frac{\Vert w_\vartheta f_s\Vert_\infty \Vert \alpha\nabla_x f_s\Vert_\infty }{\alpha(x,v)}   + \Vert w_\vartheta f_s\Vert_\infty \int^t_{\max\{0,t-\tb\}}e^{-\nu(v)(t-s)} \int_{\mathbb{R}^3} \mathbf{k}_\rho(v,u) |\partial_{[x]_j}f_s(x-(t-s)v,u)|     \dd s       \\
   & =\frac{\Vert w_\vartheta f_s\Vert_\infty \Vert \alpha\nabla_x f_s\Vert_\infty }{\alpha(x,v)}   +\Vert w_\vartheta f_s\Vert_\infty \int^t_{\max\{0,t-\tb\}}e^{-\nu(v)(t-s)} \int_{\mathbb{R}^3} \mathbf{k}_\rho(v,u) \frac{\Vert \alpha \nabla_x f_s  \Vert_\infty}{\alpha(x-(t-s)v,u)}    \dd s\\
   &\lesssim     \frac{\Vert w_\vartheta f_s\Vert_\infty \Vert \alpha\nabla_x f_s\Vert_\infty
   }{\alpha(x,v)},
\end{align*}
where we have applied Lemma \ref{Lemma: NLN} in the last line. Since $\Vert w_\vartheta f_s\Vert \ll 1$ from Corollary \ref{Thm: steady solution}, the contribution of $h=\Gamma(f_s,f_s)$ of~\eqref{f4} in included in~\eqref{cons}.

For the last term~\eqref{f5}, we apply~\eqref{nabla_tbxb} and~\eqref{Kf bdd} \eqref{gamma bdd} to get
\begin{align*}
  \eqref{f5} & =\frac{O(1)\Vert h\Vert_\infty}{\alpha(x,v)}=\frac{O(1)\Vert w_\vartheta f_s \Vert_\infty^2}{\alpha(x,v)}.
\end{align*}
Such contribution is included in~\eqref{cons}.

Then we conclude the lemma.

\end{proof}

Then we start the proof of Theorem \ref{thm: C1 steady}.

\begin{proof}[\textbf{Proof of Theorem \ref{thm: C1 steady}}]

By Lemma \ref{Lemma: prior estimate}, we only need to estimate~\eqref{collision} and~\eqref{bdr term}.

First we estimate~\eqref{bdr term}. By~\eqref{xi deri xbp} in Lemma \ref{Lemma: nabla tbxb} we have
\begin{align}
 \eqref{bdr term}  & =\frac{O(1)}{\alpha(x,v)}\sum_{i=1,2}\underbrace{\partial_{\mathbf{x}^1_{p^1,i}}[f_s(\eta_{p^1}(\mathbf{x}_{p^1}^1),v)]}_{\eqref{partial  bdr}_*}.\label{partial bdr}
\end{align}

Using the notation~\eqref{bar_v} and Lemma \ref{Lemma: bc for fs}, the boundary condition at $f_s(\eta_{p^1}(\mathbf{x}_{p^1}^1),v)$ can be written as
\begin{align*}
 &f_s(\eta_{p^1}(\mathbf{x}_{p^1}^1),v)  \\ &=r_s+\mathcal{A}\int_{n(\eta_{p^1}(\mathbf{x}_{p^1}^1))\cdot v_1>0} f_s(\eta_{p^1}(\mathbf{x}_{p^1}^1),v_1)e^{-[\frac{1}{4T_0}-\frac{1}{2T_w(\eta_{p^1}(\mathbf{x}_{p^1}^1))}]|v_1|^2}  I_0\bigg(\frac{(1-r_\perp)^{1/2}v_{1,\perp}v_{\perp}}{r_\perp T_w(\eta_{p^1}(\mathbf{x}_{p^1}^1))} \bigg) \\
   & \times |n(\eta_{p^1}(\mathbf{x}_{p^1}^1))\cdot v_1|\exp\bigg(-\frac{1}{2T_w(\eta_{p^1}(\mathbf{x}_{p^1}^1))}\Big[ \frac{|v_{1,\perp}|^2+(1-r_\perp)|v_\perp|^2}{r_\perp}+\frac{|v_{1,\parallel}-(1-r_\parallel)v_\parallel|^2}{r_\parallel(2-r_\parallel)} \Big]\bigg)\\
  & =r_s+\mathcal{A} \times \int_{\mathbf{v}_{p^1,3}^1>0}    f_s(\eta_{p^1}(\mathbf{x}_{p^1}^1),T^t_{\mathbf{x}_{p^1}^1}\mathbf{v}^1_{p^1}) \times \eqref{back term coordinate} \dd \mathbf{v}_{p^1}^1,\\
\end{align*}
with
\begin{equation}\label{back term coordinate}
\begin{split}
    & \mathbf{v}_{p^1,3}^1e^{-[\frac{1}{4T_0}-\frac{1}{2T_w(\eta_{p^1}(\mathbf{x}_{p^1}^1))}]|\mathbf{v}^1_{p^1}|^2}   I_0\bigg(\frac{(1-r_\perp)^{1/2}\mathbf{v}_{p^1,3}^1v_{\perp}}{r_\perp T_w(\eta_{p^1}(\mathbf{x}_{p^1}^1))} \bigg) \\
     & \times \exp\bigg(-\frac{1}{2T_w(\eta_{p^1}(\mathbf{x}_{p^1}^1))}\Big[ \frac{|\mathbf{v}_{p^1,3}^1|^2+(1-r_\perp)|v_\perp|^2}{r_\perp}+\frac{|(T^t_{\mathbf{x}_{p^1}^1}\mathbf{v}^1_{p^1}-\mathbf{v}_{p^1,3}^1n(\eta_{p^1}(\mathbf{x}_{p^1}^1)))-(1-r_\parallel)v_\parallel|^2}{r_\parallel(2-r_\parallel)} \Big]\bigg).
\end{split}
\end{equation}

Taking $\partial_{\mathbf{x}_{p^1,i}^1}$ to $f_s(\eta_{p^1}(\mathbf{x}_{p^1}^1),v)$ we get
\begin{align}
\partial_{\mathbf{x}_{p^1,i}^1} f_s(\eta_{p^1}(\mathbf{x}_{p^1}^1),v)  &\lesssim \partial_{\mathbf{x}_{p^1,i}^1} r_s  \label{bc1}\\
   &+  \partial_{\mathbf{x}_{p^1,i}^1}[\mathcal{A}]\int_{\mathbf{v}_{p^1,3}^1>0} \Vert w_\vartheta f_s\Vert_\infty \times \eqref{back term coordinate} \label{bc2}\\
   & + \mathcal{A}\int_{\mathbf{v}_{p^1,3}^1>0} \partial_{\mathbf{x}_{p^1,i}^1}[f_s(\eta_{p^1}(\mathbf{x}_{p^1}^1),T^t_{\mathbf{x}_{p^1}^1}\mathbf{v}^1_{p^1})] \times \eqref{back term coordinate}\label{bc3}\\
   &+\mathcal{A} \int_{\mathbf{v}_{p^1,3}^1>0} f_s(\eta_{p^1}(\mathbf{x}_{p^1}^1),T^t_{\mathbf{x}_{p^1}^1}\mathbf{v}^1_{p^1})  \partial_{\mathbf{x}_{p^1,i}^1} [\eqref{back term coordinate}]\label{bc4}.
\end{align}

Since
\[\partial_{\mathbf{x}_{p^1,i}^1} T_w(\eta_{p^1}(\mathbf{x}_{p^1}^1))=\nabla T_w \cdot \partial_3 \eta_{p^1}(\mathbf{x}_{p^1}^1)\lesssim \Vert T_w\Vert_{C^1}\Vert \eta\Vert_{C^1},\]
applying~\eqref{partial_i eta} we have
\begin{equation}\label{bound of bc1}
\eqref{bc1}\lesssim \Vert T_w\Vert_{C^1}\Vert \eta\Vert_{C^1}.
\end{equation}

For~\eqref{bc2} we change the $\mathbf{v}_{p^1}^1$ integration back to $v_1$ integration, thus $\eqref{back term coordinate}$ is replaced by $\eqref{back term}$, with integral domain changing back to $n(x_1)\cdot v_1>0$. Applying~\eqref{integrate partial_x A} we conclude
\begin{equation}\label{bound of bc2}
\eqref{bc2}\lesssim  \Vert w_\vartheta f_s\Vert_\infty\partial_{\mathbf{x}_{p^1,i}^1}[\mathcal{A}]\int_{n(x_1)\cdot v_1>0} \eqref{back term} \dd v_1 \lesssim \Vert w_\vartheta f_s\Vert_\infty .
\end{equation}

For~\eqref{bc4}, taking $\partial_{\mathbf{x}_{p^1,i}^1}$ to~\eqref{back term coordinate} we have extra term
\[\frac{\nabla_x T_w(\eta_{p^1}(\mathbf{x}_{p^1}^1)) \partial_i \eta_{p^1}(\mathbf{x}_{p^1}^1)|\mathbf{v}_{p^1}^1|^2}{T^2_w(\eta_{p^1}(\mathbf{x}_{p^1}^1))},\]
\[\frac{\nabla_x T_w(\eta_{p^1}(\mathbf{x}_{p^1}^1)) \partial_i \eta_{p^1}(\mathbf{x}_{p^1}^1)}{T^2_w(\eta_{p^1}(\mathbf{x}_{p^1}^1))} \Big[ \frac{|\mathbf{v}_{p^1,3}^1|^2+(1-r_\perp)|v_\perp|^2}{r_\perp}+\frac{|(T^t_{\mathbf{x}_{p^1}^1}\mathbf{v}^1_{p^1}-\mathbf{v}_{p^1,3}n(\eta_{p^1}(\mathbf{x}_{p^1}^1)))-(1-r_\parallel)v_\parallel|^2}{r_\parallel(2-r_\parallel)} \Big],\]
\[\frac{[(T^t_{\mathbf{x}_{p^1}^1}\mathbf{v}^1-\mathbf{v}_{p^1,3}^1n(\eta_{p^1}(\mathbf{x}_{p^1}^1)))-(1-r_\parallel)v_\parallel][\partial_{\mathbf{x}_{p^1,i}^1}T^t_{\mathbf{x}_{p^1}}\mathbf{v}_{p^1}^1-\mathbf{v}_{p^1,3}^1 \partial_{\mathbf{x}_{p^1,i}^1} n(\eta_{p^1}(\mathbf{x}_{p^1}))]}{T_w(\eta_{p^1}(\mathbf{x}_{p^1}^1))r_\parallel(2-r_\parallel)},\]
and from~\eqref{I0}
\begin{align*}
 & \partial_{\mathbf{x}_{p^1,i}^1} I_0\bigg(\frac{(1-r_\perp)^{1/2}\mathbf{v}_{p^1,3}^1v_{\perp}}{r_\perp T_w(\eta_{p^1}(\mathbf{x}_{p^1}^1))} \bigg) \\
  & =\pi^{-1} \int_0^{\pi}   e^{\frac{(1-r_\perp)^{1/2}\mathbf{v}_{p^1,3}^1v_{\perp}}{r_\perp T_w(\eta_{p^1}(\mathbf{x}_{p^1}^1))}\cos\phi} \partial_{\mathbf{x}_{p^1,i}^1} \Big( \frac{(1-r_\perp)^{1/2}\mathbf{v}_{p^1,3}^1v_{\perp}}{r_\perp T_w(\eta_{p^1}(\mathbf{x}_{p^1}^1))}\Big) \cos\phi \dd \phi \\
   & \lesssim \partial_{\mathbf{x}_{p^1,i}^1} \Big( \frac{(1-r_\perp)^{1/2}\mathbf{v}_{p^1,3}^1v_{\perp}}{r_\perp T_w(\eta_{p^1}(\mathbf{x}_{p^1}^1))}\Big)I_0\bigg(\frac{(1-r_\perp)^{1/2}\mathbf{v}_{p^1,3}^1v_{\perp}}{r_\perp T_w(\eta_{p^1}(\mathbf{x}_{p^1}^1))} \bigg),
\end{align*}
the extra term is
\[ \frac{\nabla_x T_w(\eta_{p^1}(\mathbf{x}_{p^1}^1)) \partial_i \eta_{p^1}(\mathbf{x}_{p^1}^1) (1-r_\perp)^{1/2} \mathbf{v}_{p^1,3}^1v_\perp}{r_\perp T^2_w(\eta_{p^1}(\mathbf{x}_{p^1}^1))}.\]
All the extra terms are bounded by
\[\Vert T_w\Vert_{C^1}\Vert \eta\Vert_{C^1}[\frac{1}{T_0}+\frac{1}{T_0^2}][1+|v|^2+|\mathbf{v}_{p^1}^1|^2]\lesssim 1+|v|^2+|\mathbf{v}_{p^1}^1|^2.\]
Thus
\begin{align}
  \eqref{bc4} & \lesssim \mathcal{A}\int_{\mathbf{v}_{p^1,3}>0} \Vert w_\vartheta f_s\Vert_\infty  [1+|v|^2+|\mathbf{v}_{p^1}^1|^2] \times \eqref{back term coordinate}  \notag\\
   & =  \mathcal{A}\int_{n(x_1)\cdot v_1>0} \Vert w_\vartheta f_s\Vert_\infty  [1+|v|^2+|v_1|^2] \times \eqref{back term}\lesssim \Vert w_\vartheta f_s\Vert_\infty .\label{bound of bc4}
\end{align}
In the second line we changed $\mathbf{v}_{p^1}$ integration back to $v_1$ integration and used $|v_1|^2 = |\mathbf{v}_{p^1}^1|^2$ from~\eqref{T}. In the last step we applied~\eqref{integrate back}.

Then we focus on~\eqref{bc3}, which reads
\begin{align}
\eqref{bc3}   &= \mathcal{A}\times \int_{\mathbf{v}_{p^1,3}^1>0} \eqref{back term coordinate} \notag\\
   &\times \bigg[\underbrace{\Big( \partial_{\mathbf{x}_{p^1,i}^1} T^t_{\mathbf{x}_{p^1}^1}\mathbf{v}_{p^1}^1 \Big)  \nabla_v f_s(\eta_{p^1}(\mathbf{x}_{p^1}^1),T^t_{\mathbf{x}_{p^1}^1}\mathbf{v}_{p^1}^1)}_{\eqref{deri in the bdr}_1} + \underbrace{\partial_i \eta_{p^1}(\mathbf{x}_{p^1}^1) \nabla_{x}  f_s(\eta_{p^1}(\mathbf{x}_{p^1}^1),T^t_{\mathbf{x}_{p^1}^1}\mathbf{v}_{p^1}^1)    }_{\eqref{deri in the bdr}_2} \bigg].\label{deri in the bdr}
\end{align}

First we estimate the contribution of$~\eqref{deri in the bdr}_1$. We change the $\mathbf{v}_{p^1}^1$-integration back to $v_1$ integration. The extra term $\partial_{\mathbf{x}_{p^1,i}^1} T^t_{\mathbf{x}_{p^1}^1}\mathbf{v}_{p^1} $ becomes
\[  \Big(\partial_{\mathbf{x}_{p^1,i}^1} T_{\mathbf{x}_{p^1}^1}\Big) T_{\mathbf{x}_{p^1}^1} v_1. \]
Thus such contribution is bounded as
\begin{align}
    & \mathcal{A}\times \int_{n(x_1)\cdot v_1>0} \Big(\partial_{\mathbf{x}_{p^1,i}^1} T_{\mathbf{x}_{p^1}^1}\Big) T_{\mathbf{x}_{p^1}^1} v_1   \nabla_v f_s(x_1,v_1) \times\eqref{back term} \dd v_1 \notag\\
    &  \lesssim  \Vert \eta\Vert_{C^2}\Vert \eta\Vert_{C^1} \mathcal{A}\times \int_{n(x_1)\cdot v_1>0} |\nabla_{v_1} [v_1 \times \eqref{back term}]| f_s(x_1,v_1) \dd v_1 \notag\\
    &\lesssim \Vert \eta\Vert_{C^2}\Vert \eta\Vert_{C^1}\Vert w_\vartheta f_s\Vert_\infty \mathcal{A}\times \int_{n(x_1)\cdot v_1>0} |\big[\eqref{back term} +|v_1|\nabla_{v_1} [\eqref{back term}] \big] | \dd v_1  \notag\\
    &\lesssim     \Vert \eta\Vert_{C^2}\Vert \eta\Vert_{C^1}\Vert w_\vartheta f_s\Vert_\infty  \lesssim \Vert w_\vartheta f_s\Vert_\infty.  \label{v contribution}
\end{align}
In the second line we used~\eqref{T} to get $ \Big(\partial_{\mathbf{x}_{p^1,i}^1} T_{\mathbf{x}_{p^1}^1}\Big) T_{\mathbf{x}^1_{p^1}} \lesssim \Vert \eta\Vert_{C^2}\Vert \eta\Vert_{C^1}$. In the third line we used $\nabla_v f_s(x_1,v_1) = \nabla_{v_1} [f_s(x_1,v_1)]$ and performed the integration by parts with respect to $\dd v_1$, and used $\eqref{back term} = 0$ for $n(x_1)\cdot v_1 = 0.$ In the last line we used~\eqref{integrate back} and~\eqref{integrate nablav back}.

Then we estimate the contribution of$~\eqref{deri in the bdr}_2$. We change the integration $\int_{\mathbf{v}^{1}_{p^{1},3}>0}$ back to $\int_{n(x_1)\cdot v_1>0}$, such contribution in~\eqref{bc3} reads
\begin{align}
   &\mathcal{A}\times \int_{n(x_1)\cdot v_1>0}\partial_i \eta_{p^1}(\mathbf{x}_{p^1}^1)\nabla_{x_1} f_s(x_1,v_1)\times [\eqref{back term}] \dd v_1 \notag \\
   &=\mathcal{A}\times \int_{n(x_1)\cdot v_1>0}|\partial_i \eta_{p^1}(\mathbf{x}_{p^1}^1)\frac{O(1)[\Vert w_\vartheta f_s\Vert_\infty+\Vert w_\vartheta f_s\Vert_\infty^2]+o(1)\Vert \alpha\nabla_x f_s\Vert_\infty}{\alpha(x_1,v_1)} |  \times [\eqref{back term}] \dd v_1 \label{deri bdr twice 1}\\
   &+\mathcal{A}\times \int_{n(x_1)\cdot v_1>0}\partial_i \eta_{p^1}(\mathbf{x}_{p^1}^1)\int^{t_1}_{\max\{0,\tb^1\}}e^{-\nu(v_1)(t_1-s)}\int_{\mathbb{R}^3} \mathbf{k}(v_1,u) \nabla_{x_1}f_s(x_1-(t_1-s)v_1,u) \dd u \dd s    \times [\eqref{back term}] \dd v_1  \label{deri bdr twice 2}\\
   &+\mathcal{A}\times \int_{n(x_1)\cdot v_1>0}\partial_i \eta_{p^1}(\mathbf{x}_{p^1}^1)\nabla_{x_1} f_s(\eta_{p^2}(\mathbf{x}_{p^2}^2),v_1)\times [\eqref{back term}] \dd v_1.\label{deri bdr twice 3}
\end{align}
Here we applied Lemma \ref{Lemma: prior estimate} to $\nabla_{x} f_s(x_1,v_1) = \nabla_{x_1}[f_s(x_1,v_1)]$. Then we estimate~\eqref{deri bdr twice 1}-\eqref{deri bdr twice 3}.

First we estimate~\eqref{deri bdr twice 1}. We use~\eqref{integrate back / alpha} to get
\begin{align}
  \eqref{deri bdr twice 1} & \lesssim \Vert \eta\Vert_{C^1}O(1)[\Vert w_\vartheta f_s\Vert_\infty+\Vert w_\vartheta f_s\Vert_\infty^2]+o(1)\Vert \alpha\nabla_x f_s\Vert_\infty.\label{bound of deri bdr twice 1}
\end{align}

Then we estimate~\eqref{deri bdr twice 2}. We split $\dd s$ integration into
\begin{equation}\label{split t1 integral}
\int_{\max\{0,\tb^1\}}^{t_1} = \underbrace{\int_{\max\{0,\tb^1\}}^{t_1-\e}}_{\eqref{split t1 integral}_1}+\underbrace{\int^{t_1}_{t_1-\e}}_{\eqref{split t1 integral}_2}.
\end{equation}
For$~\eqref{split t1 integral}_2$, we apply Lemma \ref{Lemma: NLN} to get such contribution in~\eqref{deri bdr twice 2} is bounded by
\begin{align}
   & \mathcal{A}\times \int_{n(x_1)\cdot v_1>0}\Vert \eta\Vert_{C^1}\int^{t_1}_{t_1-\e}e^{-\nu(v_1)(t_1-s)}\int_{\mathbb{R}^3}\frac{\mathbf{k}(v_1,u)\Vert \alpha \nabla_x f_s\Vert_\infty}{\alpha(x_1-(t_1-s)v_1,u)} \dd u \dd s    \times [\eqref{back term}] \dd v_1 \notag \\
   & \lesssim  O(\e)\mathcal{A}\times \int_{n(x_1)\cdot v_1>0}\Vert \eta\Vert_{C^1}\frac{\Vert \alpha \nabla_x f_s\Vert_\infty}{\alpha(x_1,v_1)}  \times [\eqref{back term}] \dd v_1  \notag\\
   &\lesssim        O(\e)\Vert \alpha \nabla_x f_s\Vert_\infty. \label{bound of split 1}
\end{align}
In the third line we used~\eqref{integrate back / alpha}.

For$~\eqref{split t1 integral}_1$ we exchange the $x_1$ derivative to $v_1$ derivative:
\begin{equation}\label{change x1 to v1}
\nabla_{x_1}f_s(x_1-(t_1-s)v_1,u) = -\frac{\nabla_{v_1}[f_s(x_1-(t_1-s)v_1,u)]}{t_1-s}.
\end{equation}
In this case $t_1-s\geq \e$. The contribution of$~\eqref{split t1 integral}_2$ in~\eqref{deri bdr twice 2} is
\begin{align}
   & \mathcal{A}\times \int_{n(x_1)\cdot v_1>0}\partial_i \eta\int^{t_1-\e}_{\max\{0,\tb^1\}}e^{-\nu(v_1)(t_1-s)}\int_{\mathbb{R}^3} \mathbf{k}(v_1,u)\frac{\nabla_{v_1}[f_s(x_1-(t_1-s)v_1,u)]}{t_1-s} \dd u \dd s    \times [\eqref{back term}] \dd v_1   \notag\\
   &\lesssim \mathcal{A}\times \int_{n(x_1)\cdot v_1>0}\Vert \eta\Vert_{C^1}\int^{t_1-\e}_{\max\{0,\tb^1\}}\int_{\mathbb{R}^3} \mathbf{k}(v_1,u)\frac{f_s(x_1-(t_1-s)v_1,u)}{t_1-s}| \nabla_{v_1} e^{-\nu(v_1)(t_1-s)} |\dd u \dd s    \times [\eqref{back term}] \dd v_1   \label{IBP v1 1} \\
   & +\mathcal{A}\times \int_{n(x_1)\cdot v_1>0}\Vert \eta\Vert_{C^1}\int^{t_1-\e}_{\max\{0,\tb^1\}}\int_{\mathbb{R}^3} \mathbf{k}(v_1,u)\frac{f_s(x_1-(t_1-s)v_1,u)}{t_1-s} e^{-\nu(v_1)(t_1-s)} \dd u \dd s    \times |\nabla_{v_1}[\eqref{back term}]| \dd v_1 \label{IBP v1 2}\\
   &+\mathcal{A}\times \int_{n(x_1)\cdot v_1>0}\Vert \eta\Vert_{C^1}\int_{\mathbb{R}^3} \mathbf{k}(v_1,u)\frac{f_s(x_1-(t_1-\tb^1)v_1,u)}{t_1-s}  e^{-\nu(v_1)(t_1-\tb^1)} |\nabla_{v_1}\tb^1| \dd u \times [\eqref{back term}] \dd v_1 \label{IBP v1 3}\\
   &+\mathcal{A}\times \int_{n(x_1)\cdot v_1>0}\Vert \eta\Vert_{C^1}\int^{t_1-\e}_{\max\{0,\tb^1\}}\int_{\mathbb{R}^3} |\nabla_{v_1}  \mathbf{k}(v_1,u)|\frac{f_s(x_1-(t_1-s)v_1,u)}{t_1-s} e^{-\nu(v_1)(t_1-s)} \dd u \dd s    \times \eqref{back term} \dd v_1 . \label{IBP v1 4}
\end{align}
Here we applied the integration by parts with respect to $\dd v_1$. And we used $\eqref{back term} = 0$ when $n(x_1)\cdot v_1 = 0$. We apply~\eqref{k_varrho L1} and~\eqref{nabla nu} to bound
\begin{align}
 \eqref{IBP v1 1}  &\lesssim      \mathcal{A}\times \int_{n(x_1)\cdot v_1>0}\Vert \eta\Vert_{C^1}\int^{t_1-\e}_{\max\{0,\tb^1\}}\int_{\mathbb{R}^3} \mathbf{k}(v_1,u)\frac{\Vert w_\vartheta f_s\Vert_\infty }{t_1-s} |\nabla_{v_1}\nu(v_1)|(t_1-s) e^{-\nu(v_1)(t_1-s)} \dd u \dd s    \times [\eqref{back term}] \dd v_1 \notag          \\
 &\lesssim \Vert \eta\Vert_{C^1}\Vert w_\vartheta f_s\Vert_\infty \mathcal{A}\times \int_{n(x_1)\cdot v_1>0}      \eqref{back term} \dd v_1\notag \\
   & \lesssim  O(1)     \Vert \eta\Vert_{C^1}\Vert w_\vartheta f_s\Vert_\infty , \label{bound of IBP v11}
\end{align}
where we have used~\eqref{integrate back / alpha} in the third line.

For~\eqref{IBP v1 2} we apply~\eqref{NLN epsilon} in Lemma \ref{Lemma: NLN} to bound
\begin{align}
  \eqref{IBP v1 2} & \lesssim     \mathcal{A}\times \int_{n(x_1)\cdot v_1>0}\Vert \eta\Vert_{C^1}\int^{t_1-\e}_{\max\{0,\tb^1\}}\int_{\mathbb{R}^3}\mathbf{k}(v_1,u)\frac{f_s(x_1-(t_1-s)v_1,u)}{\e} e^{-\nu(v_1)(t_1-s)} \dd u \dd s    \times |\nabla_{v_1}[\eqref{back term}]| \dd v_1 \notag \\
  &\lesssim    O(\e^{-1})\Vert \eta\Vert_{C^1}\Vert w_\vartheta f_s\Vert_\infty   \mathcal{A}\times \int_{n(x_1)\cdot v_1>0}    \nabla_{v_1}[\eqref{back term}] \dd v_1 \notag \\
  &\lesssim       O(\e^{-1})\Vert \eta\Vert_{C^1}\Vert w_\vartheta f_s\Vert_\infty, \label{bound of IBP v12}
\end{align}
where we have used~\eqref{integrate nablav back} in the third line.

For~\eqref{IBP v1 3} we apply~\eqref{k_varrho L1} to get
\begin{align}
  \eqref{IBP v1 3} & \lesssim   \mathcal{A}\times \int_{n(x_1)\cdot v_1>0}\Vert \eta\Vert_{C^1} \int_{\mathbb{R}^3}\mathbf{k}(v_1,u)\frac{\Vert w_\vartheta f_s\Vert_\infty}{\e}  e^{-\nu(v_1)(t_1-\tb^1)} \frac{1}{\alpha(x_1,v_1)}\dd u \times [\eqref{back term}] \dd v_1  \notag \\
  &\lesssim O(\e^{-1}) \Vert \eta\Vert_{C^1} \Vert w_\vartheta f_s\Vert_\infty \mathcal{A}\times \int_{n(x_1)\cdot v_1>0}  \frac{1}{\alpha(x_1,v_1)} \times [\eqref{back term}] \dd v_1  \notag\\
   & \lesssim   O(\e^{-1})\Vert \eta\Vert_{C^1}\Vert w_\vartheta f_s\Vert_\infty, \label{bound of IBP v13}
\end{align}
where we have used~\eqref{integrate back} in the third line.

Collecting~\eqref{bound of IBP v11}, \eqref{bound of IBP v12}, \eqref{bound of IBP v13} and~\eqref{bound of split 1} we conclude that
\begin{equation}\label{bound for deri bdr twice 2}
\eqref{deri bdr twice 2}\lesssim   O(\e)\Vert \alpha\nabla_x f_s\Vert_\infty + O(\e^{-1})\Vert w_\vartheta f_s\Vert_\infty.
\end{equation}

Last we estimate~\eqref{deri bdr twice 3}. Applying chain rule we have
\begin{align*}
  \partial_i \eta_{p^1}(\mathbf{x}_{p^1}^1)\nabla_{x_1}f_s(\eta_{p^2}(\mathbf{x}_{p^2}^2),v_1) &=\partial_i \eta_{p^1}(\mathbf{x}_{p^1}^1) \sum_{j=1,2}\nabla_{x_1}\mathbf{x}_{p^2,j}^2 \partial_{\mathbf{x}_{p^2,j}^2} f(\eta_{p^2}(\mathbf{x}_{p^2}^2),v_1)  \\
   & =\sum_{j=1,2} \partial_{\mathbf{x}_{p^1,i}^1}\mathbf{x}_{p^2,j}^2  \partial_{\mathbf{x}_{p^2,j}^2} f_s(\eta_{p^2}(\mathbf{x}_{p^2}^2),v_1).
\end{align*}

Note that
\begin{equation}\label{replace v2}
v_1 = (x_1-\eta_{p^2}(\mathbf{x}_{p^2}^2))/\tb^1.
\end{equation}
Applying Lemma \ref{Lemma: change of variable} we have
\begin{align*}
 \eqref{deri bdr twice 3}  &=\mathcal{A}\times \int_{n(x_1)\cdot v_1>0}\sum_{j=1,2} \partial_{\mathbf{x}_{p^1,i}^1}\mathbf{x}_{p^2,j}^2  \partial_{\mathbf{x}_{p^2,j}^2} f(\eta_{p^2}(\mathbf{x}_{p^2}^2),v) \sum_{p^2\in \mathcal{P}}\iota_{p^2}(x_2)[\eqref{back term}]  \\
   &=\mathcal{A}\times \sum_{p^2\in \mathcal{P}} \iint_{|\mathbf{x}_{p^2}^2|<\delta_1} \int_0^{t_1}e^{-\nu(v_1)\tb^1} \iota_{p^2}(x_2)  \times \sum_{j=1,2} \partial_{\mathbf{x}_{p^1,i}^1}\mathbf{x}_{p^2,j}^2  \partial_{\mathbf{x}_{p^2,j}^2} f_s(\eta_{p^2}(\mathbf{x}_{p^2}^2),v)   \\
   & \times    \sqrt{ g_{p^{2},11}(\mathbf{x}^{2}_{p^{2}})  g_{p^{2},22}(\mathbf{x}_{p^{2}}^{2}) }\frac{n(x_1)\cdot(x_1 -\eta_{p^2}(\mathbf{x}_{p^2}^2))}{\tb^1}
\frac{  |n(x_{2}) \cdot (x_1-\eta_{p^2}(\mathbf{x}_{p^2}^2))|}{|\tb^{1}|^4} \times         \frac{\eqref{back term}}{n(x_1)\cdot v_1}.
\end{align*}
In the last step we used $n(x_1)\cdot v_1=\frac{n(x_1)\cdot(x_1 -\eta_{p^2}(\mathbf{x}_{p^2}^2))}{\tb^1}$.

We apply the integration by parts with respect to $\partial_{\mathbf{x}_{p^2,j}^2}$ for $j=1,2$. For $\iota_p^2(\eta_{p^2}(\mathbf{x}_{p^2}^2))=0$ when $|\mathbf{x}_{p^2}^2|=\delta_1$ from~\eqref{O_p}, the contribution of $|\mathbf{x}_{p^2}^2|=\delta_1$ vanishes. Thus we derive
\begin{align*}
  \eqref{deri bdr twice 3} & \lesssim \mathcal{A}\Vert w_\vartheta f_s\Vert_\infty\times \sum_{p^2\in \mathcal{P}}\sum_{j=1,2}  \\
   &\times \bigg[\iint \int_0^{t_1} \partial_{\mathbf{x}_{p^2,j}^2} \big[\frac{n(x_1)\cdot(x_1 -\eta_{p^2}(\mathbf{x}_{p^2}^2))}{\tb^1}   \frac{  |n_{p^2}(\mathbf{x}_{p^2}^2) \cdot (x_1-\eta_{p^2}(\mathbf{x}_{p^2}^2))|}{|\tb^{1}|^4}     \big]\cdots   \\
   &+\iint \int_0^{t_1} \partial_{\mathbf{x}_{p^2,j}^2} \big[\sum_{j=1,2}\frac{\partial \mathbf{x}_{p^2,j}^2}{\partial \mathbf{x}_{p^1,i}^1}\sqrt{g_{p^2,11 g_{p^2,22}}} \big] \cdots \\
   & +\iint \int_0^{t_1} \partial_{\mathbf{x}_{p^2,j}^2}[\frac{\eqref{back term}}{n(x_1)\cdot v_1}] \cdots \bigg]  .
\end{align*}

Applying Lemma~\ref{Lemma: nv<v2} we have
\[\frac{n(x_1)\cdot(x_1 -\eta_{p^2}(\mathbf{x}_{p^2}^2))}{\tb^1} \lesssim \frac{|x_1-\eta_{p^2}(\mathbf{x}_{p^2}^2)|^2}{\tb^1},\quad  \frac{  |n_{p^2}(\mathbf{x}_{p^2}^2) \cdot (x_1-\eta_{p^2}(\mathbf{x}_{p^2}^2))|}{|\tb^{1}|^4}  \lesssim \frac{|x_1-\eta_{p^2}(\mathbf{x}_{p^2}^2)|^2}{|\tb^1|^4}.\]
\begin{align}
   & \partial_{\mathbf{x}_{p^2,j}^2} \big[\frac{n(x_1)\cdot(x_1 -\eta_{p^2}(\mathbf{x}_{p^2}^2))}{\tb^1}   \frac{  |n_{p^2}(\mathbf{x}_{p^2}^2) \cdot (x_1-\eta_{p^2}(\mathbf{x}_{p^2}^2))|}{|\tb^{1}|^4}     \big] \lesssim      \frac{|x_1-\eta_{p^2}(\mathbf{x}_{p^2}^2)|^3}{|\tb^1|^5}. \label{partial nv}
\end{align}

From~\eqref{xip deri xbp}, (\ref{nv<v2}) and (\ref{bound_vb_x}), we derive that
\Be\begin{split}\label{partial partial x}
&\bigg| \frac{\p }{\p {\mathbf{x}_{p^{2},j}^{2}}} \bigg(
\sum_{j=1,2}
\frac{\p \mathbf{x}^{2}_{p^{2},j}}{\p{\mathbf{x}^{1  }_{p^{1 },i}}}
\sqrt{g_{p^{2},11}g_{p^{2},22}  }
\bigg) \bigg|\\
 \lesssim& \  \| \eta \|_{C^2} \Big\{1+
\frac{|\mathbf{v}^{2}_{p^{2} ,\parallel}|}{|\mathbf{v}^{2}_{p^{2}, 3}|^2  } | \p_3 \eta_{p^{2}}(\mathbf{x}^{2  }_{p^{2 }})\cdot \p_i \eta_{p^{1}} (\mathbf{x}^{1}_{p^{1}})|\Big\}\\
  \leq & \ O(\| \eta \|_{C^2})
   \Big\{1+  \frac{|\mathbf{v}^{2}_{p^{2}  }|}{|\mathbf{v}^{2}_{p^{2}, 3}|^2  }
   |x_{1} - \eta_{p^{2}}(\mathbf{x}^{2}_{p^{2}})|\Big\}\\
   \leq & \  O(\| \eta\|_{C^2}) \frac{1}{|\mathbf{v}^{2}_{p^{2}, 3}|}=O(\Vert \eta\Vert_{C^2})\frac{1}{|n(x_2)\cdot v_1|} .
\end{split}
\Ee

Such term will be cancelled by $n(x_1)\cdot v_1$ as:
\[\frac{|n(x_1)\cdot v_1|}{|n(x_2)\cdot v_1|}\lesssim \frac{\alpha(x_1,v_1)}{\alpha(x_2,v_1)}\lesssim 1.\]

For the derivative to $\frac{\eqref{back term}}{n(x_1)\cdot v_1}$, we note that
\[v_{1,\perp} = v_1\cdot n(x_1),\quad   v_{1,\parallel}=v_1-(n(x_1)\cdot v_1) n(x_1).\]
Using~\eqref{replace v2} we get
\[|\partial_{\mathbf{x}_{p^2,j}^2} v_{1}| \lesssim \frac{1}{\tb^1} ,\quad |\partial_{\mathbf{x}_{p^2,j}^2} v_{1,\perp}| \lesssim \frac{1}{\tb^1}, \quad |\partial_{\mathbf{x}_{p^2,j}^2} v_{1,\parallel}| \lesssim
 \frac{1}{\tb^1}.\]

Then taking the derivative to $\frac{\eqref{back term}}{n(x_1)\cdot v_1}$ we have extra term
\[ [\frac{1}{4T_0}-\frac{1}{2T_w(x_1)}](\partial_{\mathbf{x}_{p^2,j}^2}v_1)v_1, \quad \frac{(1-r_\parallel)^2 \partial_{\mathbf{x}_{p^2,j}^2} v_{1,\parallel} (v_{1,\parallel}-(1-r_\parallel)v_\parallel)}{T_w(x_1)r_\parallel(2-r_\parallel)},\quad \frac{\partial_{\mathbf{x}_{p^2,j}^2} v_{1,\perp} v_{1,\perp}}{T_w(x_1)r_\perp}.\]
The extra term comes from $I_0$ is
\[\partial_{\mathbf{x}_{p^2,j}^2} \Big(\frac{(1-r_\perp)^{1/2}v_{1,\perp}v_\perp}{r_\perp T_w(x_1)} \Big)\lesssim    \frac{(1-r_\perp)^{1/2}v_\perp \partial_{\mathbf{x}_{p^2,j}^2} v_{1,\parallel} }{r_\perp}.\]
Thus all of them are bounded by
\[\frac{[1+|v|^2+|v_1|^2]}{\tb^1}.\]

Then for $\e \ll 1$ applying~\eqref{first bound of back} we get
\begin{align}
 \partial_{\mathbf{x}_{p^2,j}^2} [\frac{\eqref{back term}}{n(x_1)\cdot v_1}]   & \lesssim [1+|v|^2+|v_1|^2][\frac{\eqref{back term}}{n(x_1)\cdot v_1}] \notag \\
   & \lesssim [1+|v|^2+|v_1|^2]e^{-\frac{1}{4T_0}|v_1|^2}e^{\frac{O(\e)}{T_0}|v|^2}e^{\frac{O(\e)}{T_0}|v_1|^2}. \label{deri back}
\end{align}

Collecting~\eqref{partial nv}, \eqref{partial partial x} and \eqref{deri back} we obtain
\begin{align}
  \eqref{deri bdr twice 3} & \lesssim \mathcal{A}\times \Vert w_\vartheta f_s\Vert_\infty e^{O(\e) |v|^2} \notag\\
   & \times \iint \int_0^{t_1}e^{-\nu \tb^1}\Big[\frac{|x_1-\eta_{p^2}(\mathbf{x}_{p^2}^2)|^3}{\tb^5}+\frac{|x_1-\eta_{p^2}(\mathbf{x}_{p^2}^2)|^2}{\tb^4} +\frac{|x_1-\eta_{p^2}(\mathbf{x}_{p^2}^2)|^4}{\tb^6} \Big]e^{-\frac{1}{8T_0}\frac{|x_1-\eta_{p^2}(\mathbf{x}_{p^2}^2)|^2}{|\tb^1|^2}    } \notag\\
   &\lesssim  \Vert w_\vartheta f_s\Vert_\infty \int_0^\infty   \frac{e^{-\nu(v_1)\tb^1}}{|\tb^1|^{1/2}} \iint \frac{1}{|x_1-\eta_{p^2}(\mathbf{x}_{p^2}^2)|^{3/2}}\lesssim \Vert w_\vartheta f_s\Vert_\infty \label{bound for deri bdr twice 3}.
\end{align}
In the last line we have used the definition of $\mathcal{A}$ in~\eqref{front term} to have
\[\mathcal{A}\times e^{O(\e)|v|^2}\lesssim 1.\]
And we used
\Be
\begin{split}\notag
&\Big[
\frac{|x_{1} - \eta_{p^{2} } (\mathbf{x}^{2}_{p^{2}})|^3}{|\tb^{1}|^5}+ \frac{|x_{1} - \eta_{p^{2} } (\mathbf{x}^{2}_{p^{2}})|^2}{|\tb^{1}|^4}+\frac{|x_1-\eta_{p^2}(\mathbf{x}_{p^2}^2)|^4}{\tb^6}\Big]
e^{-\frac{|x_{1} - \eta_{p^{2} } (\mathbf{x}^{2}_{p^{2}})|^2}{8T_0|\tb^{1}|^2}}\\
&\leq \frac{1}{|\tb^{1}|^{1/2}} \frac{1}{|x_{1} - \eta_{p^{2} } (\mathbf{x}^{2}_{p^{2}})|^{3/2}}
\Big[\frac{|x_{1} - \eta_{p^{2} } (\mathbf{x}^{2}_{p^{2}})|^{9/2}}{|\tb^{1}|^{9/2}}+\frac{|x_{1} - \eta_{p^{2} } (\mathbf{x}^{2}_{p^{2}})|^{7/2}}{|\tb^{1}|^{7/2}}+\frac{|x_{1} - \eta_{p^{2} } (\mathbf{x}^{2}_{p^{2}})|^{11/2}}{|\tb^{1}|^{11/2}}\Big]
\\
&\times e^{-\frac{|x_{1} - \eta_{p^{2} } (\mathbf{x}^{2}_{p^{2}})|^2}{8T_0|\tb^{1}|^2}}\\
&\lesssim \frac{1}{|\tb^{1}|^{1/2}} \frac{1}{|x_{1} - \eta_{p^{2} } (\mathbf{x}^{2}_{p^{2}})|^{3/2}}.
\end{split}
\Ee

Then we combine~\eqref{bound of deri bdr twice 1}, \eqref{bound for deri bdr twice 2},~\eqref{bound for deri bdr twice 3} and \eqref{v contribution} to get
\begin{equation}\label{bound of bc3}
\eqref{bc3}\lesssim       o(1)\Vert \alpha\nabla_x f_s\Vert_\infty   O(\e^{-1})\Vert w_\vartheta f_s\Vert_\infty.
\end{equation}

Finally combining~\eqref{bound of bc1}, \eqref{bound of bc2}, \eqref{bound of bc3} and~\eqref{bound of bc4} we conclude that
\begin{equation*}
|\partial_{\mathbf{x}_{p^1,i}}[f_s(\eta_{p^1}(\mathbf{x}_{p^1},v))]|\lesssim   O(\e^{-1})[\Vert w_\vartheta f_s\Vert_\infty + \Vert w_\vartheta f_s\Vert_\infty^2]+o(1)\Vert \alpha\nabla_x f_s\Vert_\infty.
\end{equation*}
This, with~\eqref{partial bdr}, conclude that the boundary term is bounded by
\begin{equation}\label{bound for bdr}
\eqref{bdr term}\lesssim       \frac{O(\e^{-1})[\Vert w_\vartheta f_s\Vert_\infty + \Vert w_\vartheta f_s\Vert_\infty^2]+o(1)\Vert \alpha\nabla_x f_s\Vert_\infty}{\alpha(x,v)}.
\end{equation}

\textbf{Estimate of~\eqref{collision}}. For the collision term we apply Lemma \ref{Lemma: prior estimate} to $\partial_{[x]_j}f_s(x-(t-s)v,u)$ to get
\begin{align}
 \eqref{collision}  &=\int^t_{\max\{0,t-\tb\}}e^{-\nu(v)(t-s)}\int_{\mathbb{R}^3} \mathbf{k}(v,u) \frac{O(1)[\Vert w_\vartheta f_s\Vert_\infty + \Vert w_\vartheta
  f_s\Vert_\infty^2]+o(1)\Vert \alpha\nabla_x f_s\Vert_\infty}{\alpha(x-(t-s)v,u)} \label{kk1}  \\
   &+\int^t_{\max\{0,t-\tb\}}e^{-\nu(v)(t-s)}\int_{\mathbb{R}^3} \mathbf{k}(v,u) \int^{s}_{\max\{0,s-\tb(x-(t-s)v,u)\}}e^{-\nu(u)(s-s')}\notag\\
&   \times \int_{\mathbb{R}^3} \mathbf{k}(u,u')\partial_{[x]_j} f(x-(t-s)v-(s-s')u,u')\dd u' \dd s'\label{kk2}  \\
   & +\int^t_{\max\{0,t-\tb\}}e^{-\nu(v)(t-s)}\int_{\mathbb{R}^3} \mathbf{k}(v,u) \frac{\partial f_s(\xb(x-(t-s)v,u),u)}{\partial [x]_j}.\label{kk3}
\end{align}

Directly applying~\eqref{NLN} in Lemma \ref{Lemma: NLN} we bound
\begin{equation}\label{bound of kk1}
\eqref{kk1}\lesssim    \frac{O(1)[\Vert w_\vartheta f_s\Vert_\infty + \Vert w_\vartheta
  f_s\Vert_\infty^2]+o(1)\Vert \alpha\nabla_x f_s\Vert_\infty}{\alpha(x,v)}.
\end{equation}

For~\eqref{kk3}, let $y = x-(t-s)v$, then
\[\frac{\partial f_s(\xb(x-(t-s)v,u))}{\partial [x]_j} = \frac{\partial f_s(\xb(y,u))}{\partial [y]_j},\]
which is exactly the same as~\eqref{bdr term} with replacing $x$ by $y$, $v$ by $u$. Note that we already derive the upper bound for $\eqref{bdr term}$ in~\eqref{bound for bdr}, such estimate works for any $x\in \Omega$, $v\in \mathbb{R}^3$. Thus we can also bound $\frac{\partial f_s(\xb(x-(t-s)v,u))}{\partial [x]_j}$ by~\eqref{bound for bdr}. Therefore,
\begin{equation}\label{bound of kk3}
\begin{split}
   \eqref{kk3}\lesssim   &  \int^t_{\max\{0,t-\tb\}}e^{-\nu(v)(t-s)}\int_{\mathbb{R}^3} \mathbf{k}(v,u) \frac{O(\e^{-1})[\Vert w_\vartheta f_s\Vert_\infty+\Vert w_\vartheta f_s\Vert_\infty]+o(1)\Vert \alpha\nabla_x f_s  \Vert_\infty}{\alpha(y,u)}\dd u \dd s\\
     &\lesssim   \frac{O(\e^{-1})[\Vert w_\vartheta f_s\Vert_\infty+\Vert w_\vartheta f_s\Vert_\infty+o(1)\Vert \alpha\nabla_x f_s  \Vert_\infty]}{\alpha(x,v)},
\end{split}
\end{equation}
where we have used Lemma \ref{Lemma: NLN} in the second line.

Last we estimate~\eqref{kk2}. We split the $s'$-integration into two cases:
\begin{equation}\label{split s}
\int^s_{\max\{...\}} = \underbrace{\int^{s}_{s-\e}}_{\eqref{split s}_1}+\underbrace{\int^{s-\e}_{\max\{...\}}}_{\eqref{split s}_2}.
\end{equation}
For $\eqref{split s}_1$, we bound
\[|\partial_{[x]_j}f_s(x-(t-s)v-(s-s')u,u')|\lesssim \frac{\Vert \alpha\nabla_x f_s\Vert_\infty}{\alpha(x-(t-s)v-(s-s')u,u')}.\]

Then we apply~\eqref{NLN epsilon} in Lemma \ref{Lemma: NLN} to derive
\begin{align}
 \eqref{kk2}\mathbf{1}_{s-\e\leq  s'\leq s}  & \lesssim     \int^t_{\max\{0,t-\tb\}} e^{-\nu(v)(t-s)}\int_{\mathbb{R}^3}\mathbf{k}(v,u)\frac{O(\e)\Vert \alpha\nabla_x f_s\Vert_\infty}{\alpha(x-(t-s)v,u)}  \notag\\
   & \lesssim \frac{O(\e)\Vert \alpha\nabla_x f_s\Vert_\infty}{\alpha(x,v)}, \label{bound of s integration1}
\end{align}
where we have used~\eqref{NLN} in the second line.

For$~\eqref{split s}_2$ we exchange the $x$ derivative into $u$ derivative:
\[\partial_{[x]_j} f_s(x-(t-s)v-(s-s')u,u') =   \frac{\partial_{u_j}[f_s(x-(t-s)v-(s-s')u,u')]}{s'-s}. \]
In this case we have $s-s' \geq \e$. Applying the integration by parts with respect to $\dd u$ we get
\begin{align}
  \eqref{kk2}\mathbf{1}_{s'\leq s-\e} & \lesssim O(\e^{-1})\int^t_{\max\{0,t-\tb\}}e^{-\nu(v)(t-s)}\int_{\mathbb{R}^3}  \int^{s}_{\max\{0,s-\tb(x-(t-s)v,u)\}}e^{-\nu(u)(s-s')}\notag\\
&   \times \int_{\mathbb{R}^3} |\nabla_u [\mathbf{k}(v,u)\mathbf{k}(u,u')]| f_s(x-(t-s)v-(s-s')u,u')\dd u' \dd s'\label{kk2 1}  \\
  & + \int^t_{\max\{0,t-\tb\}}e^{-\nu(v)(t-s)}\int_{\mathbb{R}^3} \mathbf{k}(v,u)   \int^{s}_{\max\{0,s-\tb(x-(t-s)v,u)\}}   e^{-\nu(u)(s-s')}(s-s')|\nabla_u \nu(u)| \notag\\
  & \times \int_{\mathbb{R}^3} \mathbf{k}(u,u') \frac{f_s(x-(t-s)v-(s-s')u,u')}{s'-s}\dd u' \label{kk2 2}  \\
  & +  \int^t_{\max\{0,t-\tb\}}e^{-\nu(v)(t-s)}\int_{\mathbb{R}^3}  \mathbf{k}(v,u)    e^{-\nu(u)\tb(x-(t-s)v,u)}  |\nabla_u \tb(x-(t-s)v,u)| \notag\\
  & \times \int_{\mathbb{R}^3} \mathbf{k}(u,u') \frac{f_s(\xb(x-(t-s)v,u)u,u')}{s'-s}\dd u'. \label{kk2 3}
\end{align}

First we estimate~\eqref{kk2 1}. For some $c\ll \vartheta$ we have
\begin{align}
  \eqref{kk2 1} &  \lesssim O(\e^{-1})\int^t_{\max\{0,t-\tb\}}e^{-\nu(v)(t-s)}\int_{\mathbb{R}^3}  \int^{s}_{\max\{0,s-\tb(x-(t-s)v,u)\}}e^{-\nu(u)(s-s')} \notag\\
   & \times \int_{\mathbb{R}^3} |\nabla_u [\mathbf{k}(v,u)\mathbf{k}(u,u')]|\frac{1}{e^{c|v|^2}}\frac{e^{c|v|^2}}{e^{c|u|^2}}\frac{e^{c|u|^2}}{e^{c|u'|^2}} \Vert e^{c|v|^2}f_s\Vert_\infty \dd u' \dd s' \notag\\
   &\lesssim   \frac{O(\e^{-1})}{e^{c|v|^2}} \int^t_{\max\{0,t-\tb\}}e^{-\nu(v)(t-s)}\int_{\mathbb{R}^3}  \int^{s}_{\max\{0,s-\tb(x-(t-s)v,u)\}}e^{-\nu(u)(s-s')} \mathbf{k}_{\tilde{\varrho}}(v,u)[1+\frac{1}{|v-u|}] \notag\\
   & \times \int_{\mathbb{R}^3} \mathbf{k}_{\tilde{\varrho}}(u,u')[1+\frac{1}{|u-u'|}]\Vert w_\vartheta f_s\Vert_\infty \dd u' \dd s' \notag\\
   &     \lesssim  \frac{O(\e^{-1})\Vert w_\vartheta f_s\Vert_\infty }{e^{c|v|^2}} \times \frac{\alpha(x,v)}{\alpha(x,v)}\lesssim \frac{O(\e^{-1})\Vert w_\vartheta f_s\Vert_\infty }{\alpha(x,v)}.\label{bound of kk2 1}
\end{align}
In the third line we applied~\eqref{k_theta}. In the last line we applied Lemma \ref{Lemma: NLN} and used $\frac{\alpha(x,v)}{e^{c|v|^2}}\lesssim 1$.

Then we estimate~\eqref{kk2 2}. Similar to computation of~\eqref{bound of kk2 1}, we have
\begin{align}
  \eqref{kk2 2} &  \lesssim \int^t_{\max\{0,t-\tb\}}e^{-\nu(v)(t-s)}\int_{\mathbb{R}^3} \mathbf{k}(v,u) \int^{s}_{\max\{0,s-\tb(x-(t-s)v,u)\}}e^{-\nu(u)(s-s')} \notag\\
   & \times \int_{\mathbb{R}^3} \mathbf{k}(u,u')\frac{1}{e^{c|v|^2}}\frac{e^{c|v|^2}}{e^{c|u|^2}}\frac{e^{c|u|^2}}{e^{c|u'|^2}} \Vert e^{c|v|^2}f_s\Vert_\infty \dd u' \dd s' \notag\\
   &\lesssim   \frac{1}{e^{c|v|^2}} \int^t_{\max\{0,t-\tb\}}e^{-\nu(v)(t-s)}\int_{\mathbb{R}^3}  \int^{s}_{\max\{0,s-\tb(x-(t-s)v,u)\}}e^{-\nu(u)(s-s')} \mathbf{k}_{\tilde{\varrho}}(v,u)\notag\\
   & \times \int_{\mathbb{R}^3} \mathbf{k}_{\tilde{\varrho}}(u,u')\Vert w_\vartheta f_s\Vert_\infty \dd u' \dd s' \notag\\
   &     \lesssim  \frac{O(\e^{-1})\Vert w_\vartheta f_s\Vert_\infty }{e^{c|v|^2}}\lesssim \frac{O(\e^{-1})\Vert w_\vartheta f_s\Vert_\infty }{\alpha(x,v)}. \label{bound of kk2 2}
\end{align}

Last we estimate~\eqref{kk2 3}. Applying~\eqref{nabla_tbxb} we have
\begin{align}
  \eqref{kk2 3} &\lesssim O(\e^{-1})\Vert w_\vartheta f_s\Vert_\infty \int^t_{\max\{0,t-\tb\}}e^{-\nu(v)(t-s)}\int_{\mathbb{R}^3}  \mathbf{k}_\varrho(v,u)     \frac{1}{\alpha(x-(t-s)v,u)}  \notag\\
   & \lesssim \frac{O(\e^{-1})\Vert w_\vartheta f_s\Vert_\infty}{\alpha(x,v)}. \label{bound of kk2 3}
\end{align}
In the second line we used Lemma \ref{Lemma: NLN}.

Combining~\eqref{bound of kk2 1}, \eqref{bound of kk2 2} and~\eqref{bound of kk2 3} we conclude
\begin{equation}\label{kk2 bound 2}
\eqref{kk2}\mathbf{1}_{s-s'\geq \e}\lesssim \frac{O(\e^{-1})\Vert w_\vartheta f_s\Vert_\infty}{\alpha(x,v)}.
\end{equation}

Then we combine~\eqref{bound of kk1},\eqref{bound of kk3},\eqref{bound of s integration1},\eqref{kk2 bound 2} and conclude
\begin{equation}\label{bound of collision}
\eqref{collision}\lesssim \frac{o(1)\Vert \alpha\nabla_x f_s\Vert_\infty + O(\e^{-1})[\Vert w_\vartheta f_s\Vert_\infty +\Vert w_\vartheta f_s\Vert_\infty^2]}{\alpha(x,v)}.
\end{equation}

Finally from Lemma \ref{Lemma: prior estimate}, \eqref{bound for bdr} and~\eqref{bound of collision} we conclude Theorem \ref{thm: C1 steady}.

\end{proof}

\textbf{Acknowledgements.} The author thanks his advisors Chanwoo Kim and Qin Li for helpful discussion. This research is partly supported by NSF DMS-1501031, DMS-1900923, DMS 1750488 and UW-Madison Data Science Initiative.

\bibliography{CLre}
\bibliographystyle{amsplain}

\end{document}